\documentclass[12pt]{article}

\def\papertitle{Wilson loop expectations\\ as sums over surfaces on the plane}

\usepackage[T1]{fontenc}
\usepackage[utf8]{inputenc}
\usepackage[english]{babel}
\usepackage{lmodern}

\usepackage{amsmath,amsthm,amssymb}
\usepackage{mathtools, etoolbox, xparse}
\usepackage{graphicx}
\usepackage{subcaption}
\usepackage[usenames,dvipsnames]{xcolor}
\usepackage{comment}
\usepackage{bm, bbm}
\usepackage{enumerate}
\usepackage{tikz}
\usetikzlibrary{arrows,positioning,automata}
\usepackage{longtable}
\usepackage{mhequ}

\usepackage[margin=1in]{geometry}
\usepackage[pdftitle={\papertitle}, pdfauthor={Minjae Park, Joshua Pfeffer, Scott Sheffield, Pu Yu}, colorlinks=true,linkcolor=NavyBlue,urlcolor=RoyalBlue,citecolor=PineGreen,bookmarks=true,bookmarksopen=true,bookmarksopenlevel=2,unicode=true,linktocpage]{hyperref}

\usepackage{array} 
\usepackage{booktabs} 
\usepackage[column=0]{cellspace} 
\theoremstyle{plain}
\newtheorem{thm}{Theorem}[section]
\newtheorem{cor}[thm]{Corollary}
\newtheorem{lem}[thm]{Lemma}
\newtheorem{prop}[thm]{Proposition}

\theoremstyle{definition}
\newtheorem{defn}[thm]{Definition}
\newtheorem{remark}[thm]{Remark}

\newtheorem{ques}{Question}[section]

\newtheorem{eg}[thm]{Example}

\numberwithin{equation}{section} 

\def\@rst #1 #2other{#1}
\newcommand\MR[1]{\relax\ifhmode\unskip\spacefactor3000 \space\fi
\MRhref{\expandafter\@rst #1 other}{#1}}
\newcommand{\MRhref}[2]{\href{http://www.ams.org/mathscinet-getitem?mr=#1}{MR#2}}

\newcommand{\ep}{\varepsilon}
\DeclareMathOperator{\NC}{\mathcal{NC}}

\newcommand{\BB}{\mathbbm}
\newcommand{\mcl}{\mathcal}
\newcommand{\frk}{\mathfrak}
\newcommand{\ol}{\overline}
\newcommand{\ul}{\underline}
\newcommand{\wt}{\widetilde}
\newcommand{\wh}{\widehat}

\DeclareMathOperator{\1}{\mathbf{1}}

\DeclareMathOperator{\tr}{tr}
\DeclareMathOperator{\Tr}{Tr}
\DeclareMathOperator{\sgn}{sgn}
\newcommand{\nabs}[1]{#1}
\newcommand{\Cep}{\mcl C_{\ep}}

\DeclarePairedDelimiterX{\abs}[1]{\lvert}{\rvert}{\ifblank{#1}{{}\cdot{}}{#1}} 
\DeclarePairedDelimiterX{\nrm}[1]{\lVert}{\rVert}{\ifblank{#1}{{}\cdot{}}{#1}} 
\DeclarePairedDelimiterX{\set}[1]{\{}{\}}{#1} 
\DeclarePairedDelimiterX{\block}[1]{[}{]}{\ifblank{#1}{{}\cdot{}}{#1}} 
\newcommand{\avg}[1]{\left\langle#1\right\rangle} 
\newcommand{\ETr}[1]{\avg{\Tr\left(#1\right)}} 
\DeclarePairedDelimiterX{\inp}[2]{\langle}{\rangle}{#1,#2} 

\RenewDocumentCommand{\Pr}{ e{^} s o >{\SplitArgument{1}{|}}m }{
\operatorname{\mathbb{P}}
\IfValueT{#1}{{\!}^{#1}}
\IfBooleanTF{#2}{
  \expectarg*{\expectvar#4}%
}{
  \IfNoValueTF{#3}{
    \expectarg{\expectvar#4}%
  }{
    \expectarg[#3]{\expectvar#4}%
  }%
}%
}
\NewDocumentCommand{\Ex}{ e{^} s o >{\SplitArgument{1}{|}}m }{
  \operatorname{\mathbb{E}}
  \IfValueT{#1}{{\!}^{#1}}
  \IfBooleanTF{#2}{
    \expectarg*{\expectvar#4}%
  }{
    \IfNoValueTF{#3}{
      \expectarg{\expectvar#4}%
    }{
      \expectarg[#3]{\expectvar#4}%
    }%
  }%
}
\NewDocumentCommand{\expectvar}{mm}{%
  #1\IfValueT{#2}{\nonscript\;\delimsize\vert\nonscript\;#2}%
}
\DeclarePairedDelimiterX{\expectarg}[1]{[}{]}{#1}

\newcommand{\GLNR}{\operatorname{GL}(N,\mathbb{R})}
\newcommand{\GLNC}{\operatorname{GL}(N,\mathbb{C})}

\newcommand{\UN}{\operatorname{U}(N)}
\newcommand{\uN}{\mathfrak{u}(N)}
\newcommand{\SUN}{\operatorname{SU}(N)}
\newcommand{\suN}{\mathfrak{su}(N)}
\newcommand{\SON}{\operatorname{SO}(N)}
\newcommand{\soN}{\mathfrak{so}(N)}
\newcommand{\ON}{\operatorname{O}(N)}
\newcommand{\oN}{\mathfrak{o}(N)}
\newcommand{\SpN}{\operatorname{Sp}(N)}
\newcommand{\SphN}{\operatorname{Sp}(N/2)}

\newcommand{\sphN}{\mathfrak{sp}(N/2)}
\newcommand{\fc}{\mathfrak{c}}
\newcommand{\fg}{\mathfrak{g}}

\newcommand{\fS}{\mathfrak{S}}
\newcommand{\cF}{\mathcal{F}}
\newcommand{\cM}{\mathcal{M}}
\newcommand{\cG}{\mathcal{G}}
\newcommand{\cJ}{\mathcal{J}}
\newcommand{\cR}{\mathcal{R}}

\newcommand{\mbf}[1]{\mathbf{#1}}
\newcommand{\mrm}[1]{\mathrm{#1}}

\newcommand{\varep}{ \varepsilon }
\newcommand{\ovl}[1]{\overline{#1}}

\newcommand{\E}{\mathbb{E}}
\newcommand{\mc}[1]{\mathcal{#1}}

\newcommand{\numcomp}{\#\mrm{comp}}
\newcommand{\ind}{\mathbbm{1}}

\makeatletter
\newcommand{\mytag}[2]{%
  \text{#1}%
  \@bsphack
  \begingroup
    \@onelevel@sanitize\@currentlabelname
    \edef\@currentlabelname{%
      \expandafter\strip@period\@currentlabelname\relax.\relax\@@@%
    }%
    \protected@write\@auxout{}{%
      \string\newlabel{#2}{%
        {#1}%
        {\thepage}%
        {\@currentlabelname}%
        {\@currentHref}{}%
      }%
    }%
  \endgroup
  \@esphack
} 

\newcommand*{\rom}[1]{\expandafter\@slowromancap\romannumeral #1@} 
\newcommand{\romnum}[1]{\textup{\rom{#1}}} 

\makeatother

\definecolor{color1}{RGB}{255, 0, 0}
\definecolor{color2}{RGB}{0, 0, 255}
\definecolor{color3}{RGB}{0, 255, 0}
\definecolor{color4}{RGB}{255, 127, 0}
\definecolor{color5}{RGB}{191, 255, 0}
\definecolor{color6}{RGB}{255, 0, 170}
\newcommand{\wA}{\textcolor{color1}{\mathbf{A}}}
\newcommand{\wAi}{\textcolor{color1}{\mathbf{A^{-1}}}}
\newcommand{\wB}{\textcolor{color2}{\mathbf{B}}}

\newcommand{\lr}[1]{\lambda_{#1}}

\newcommand{\ls}[1]{\abs{\lambda_{#1}}}

\newcommand{\loopimg}[1]{
\begin{minipage}{20mm}
  \includegraphics[width=20mm, height=20mm]{chart/#1.pdf}
\end{minipage}
}
\newcommand{\otoprule}{\midrule[\heavyrulewidth]}

\title{\papertitle}
\author{ 
    \begin{tabular}{c}{Minjae Park}\\[-4pt]\small University of Chicago\end{tabular}
    \begin{tabular}{c}{Joshua Pfeffer}\\[-4pt]\small Columbia University\end{tabular}
    \begin{tabular}{c}{Scott Sheffield}\\[-4pt]\small MIT\end{tabular}
    \begin{tabular}{c}{Pu Yu}\\[-4pt]\small MIT\end{tabular}
}

\date{\today}

\begin{document}
\maketitle

\begin{abstract}
Although lattice Yang-Mills theory on finite subgraphs of $\mathbb Z^d$ is easy to rigorously define, the construction of a satisfactory continuum theory on $\mathbb R^d$ is a major open problem when $d \geq 3$. Such a theory should in some sense assign a Wilson loop expectation to each suitable finite collection $\mathcal L$ of loops in $\mathbb R^d$. One classical approach is to try to represent this expectation as a sum over surfaces with boundary $\mathcal L$. There are some formal/heuristic ways to make sense of this notion, but they typically yield an ill-defined difference of infinities.

In this paper, we show how to make sense of Yang-Mills integrals as surface sums for $d=2$, where the continuum theory is more accessible. Applications include several new explicit calculations, a new combinatorial interpretation of the master field, and a new probabilistic proof of the Makeenko-Migdal equation.
\end{abstract}

\tableofcontents

\bigskip
\noindent\textbf{Acknowledgments.} We thank Sky Cao, Sourav Chatterjee, Ewain Gwynne, Thierry L\'evy, Hao Shen, and Xin Sun for helpful discussions, and further thank Sky Cao for providing feedback on an early draft. The authors were partially supported by NSF grants DMS 1712862 and DMS 2153742. J.P. was partially supported by the NSF Postdoctoral Research Fellowship under grant 2002159. The authors also thank the Institute for Advanced Study at Princeton, where this work was partially completed.

\section{Introduction}\label{sec-intro}
\subsection{Euclidean Yang-Mills and Wilson loop expectations}\label{sec-intro-ym}

The four-dimensional quantum Yang-Mills theory is extremely important in mathematical physics because it provides the core of the standard model, which describes the properties of elementary particles in nature. Still, much remains unknown about its mathematical foundation, and rigorously understanding this foundation is a famous Clay Millennium Problem~\cite{Jaffe2006a}.

The physically relevant Yang-Mills theory is a \emph{quantum} theory defined on \emph{Minkowski} space. On the other hand, there are some standard arguments within the physics literature indicating that, in some sense, one can reduce this problem to the study of a \emph{probabilistic} Yang-Mills theory on \emph{Euclidean space}~\cite{glimm2012quantum, seiler1982gauge}. For this reason, the probabilistic/Euclidean theory is itself of tremendous interest and will be the main focus of this paper. See e.g.\ a recent survey article by Chatterjee~\cite{chatterjee2016a} for additional probabilistic perspective on this subject. We will not further address quantum wave functions or Minkowski spaces in this article.

Fix a Riemannian manifold $\cM$, a gauge group $G$, and a principle $G$-bundle $\pi:P\to \cM$. The Euclidean Yang-Mills theory concerns the stochastic connection $\omega$ on $P$ sampled from the ``Yang-Mills measure''
\begin{equation}\label{eqn-the-YM-measure}
 d\mu_{\text{YM}}(\omega) = \frac{1}{Z}e^{-\frac{1}{2T}S_{\text{YM}}(\omega)}d\omega 
\end{equation}
where $T$ is the coupling constant, and the Yang-Mills action $S_{\text{YM}}(\omega)$ is the $L^2$ norm of the curvature of $\omega$, where $d\omega$ is interpreted heuristically as a ``uniform measure'' on the space of connections on $P$.
The path integral formulation of~\eqref{eqn-the-YM-measure} has been extremely successful in physics, yielding accurate computations of physical quantities in terms of Feynman diagrams and leading to a substantial understanding of the quantum Yang-Mills theory, e.g.\ see~\cite{Faddeev2016}.
However, from a mathematical perspective, it is difficult to directly make sense of the measure on the infinite-dimensional space of connections on $P$.
Currently, there is no known mathematical construction of the Yang-Mills measure for $d\ge 3$, although there has been some recent progress in the case $d=3$~\cite{shen2018,shen2021large,cao2021state}.

For dimension $d=2$, following a series of works by Driver, Fine, Witten, Sengupta, L\'{e}vy, and others~\cite{Gross1989, Driver1989, fine1991, Witten91, sengupta1997, levy2003, levy2010a, shen2018, che19, chandra2020}, the functional integrals~\eqref{eqn-the-YM-measure} can be mathematically rigorously defined using stochastic calculus. (See~\cite{Levy2019a} for a recent survey on two-dimensional Yang-Mills theory.) Hence, the Euclidean two-dimensional Yang-Mills measure is well understood for any surface. In light of these works, some aspects of the two-dimensional theory are simple; this is because ``gauge-fixing'' arguments allow one to reduce the problem to the analysis of matrices obtained by running Brownian motion on a Lie group for a fixed amount of time --- or to the expected products of traces of words produced from finitely many such matrices. For example, one may consider $A_1, A_2, A_3$, with each $A_i$ obtained independently by running a Brownian motion on a Lie group for time $t_i$ (starting at the identity) and seek to understand quantities like:
$\mathbb E[(\tr A_1 A_3)( \tr A_1 A_2^{-1} A_3 A_2 A_3^{-1} )]$. We will explore this random matrix problem in this paper, and will explain how to geometrically interpret the resulting Wilson loop expectations (defined below) for arbitrary collections of loops, expressing them in terms of sums over surfaces. 

Let us also stress that the theory developed in this paper is being used in a follow up paper (by Cao and two of the current authors) to explore problems in lattice Yang-Mills theory in higher dimensions \cite{cao2023random}. In particular, the results in Appendix~\ref{sec-appendix} are used in the companion paper.

We review the necessary background on two-dimensional Yang-Mills theory in Section~\ref{sec-background}. As we explain further in Section~\ref{sec-background}, the key observables in the theory of two-dimensional Yang-Mills are the \emph{Wilson loop observables} of a collection of loops in the plane---defined as products of the trace of the \emph{holonomies} of the loops with respect to the Yang-Mills connection. We call the expectations of the product of these traces the \emph{Wilson loop expectations} of the collection of loops. These expectations play the role of $n$-point functions of the theory.

At the end of a long and groundbreaking work, L\'evy presents a table of explicit computations for the single-loop Wilson loop expectations that one obtains in the $N \to \infty$ limit~\cite{Levy2011a}. Using the ``surface sum'' approach explained below, we will derive a similar table for $G=\UN$ with $N$ finite.  (L\'evy's table can be recovered as the $N\to \infty$ limit of our table.)
Readers interested in these explicit calculations may proceed directly to Section~\ref{sec-chart} where this chart is presented and discussed, along with an intuitive explanation of what happens in the limiting cases where $N=\infty$, or $N=1$, or the loop areas are small or large. We also offer in Section~\ref{sec-forest} a combinatorial method for computing and interpreting L\'evy's $N=\infty$ limits in terms of so-called forest polynomials. Although the techniques in this section are somewhat different, they provide new ways to compute and understand the $N=\infty$ surface sums described in Section~\ref{sec-chart} and elsewhere. This section makes use of combinatorial \emph{parking functions} as well as a somewhat different Poisson point process calculation.

Let us also mention that a recent groundbreaking work by Magee and Puder~\cite{Magee2019,Magee2019a} has proved a result that corresponds to the limit of our result when loop areas tend to infinity---i.e., when the random matrices defined using heat kernels are replaced with Haar measure random matrices. While their result is also expressed in terms of sums over surfaces, the techniques they use are completely different, relying more heavily on representation theory and the so-called Weingarten calculus. Both~\cite{Levy2011a} and~\cite{Magee2019,Magee2019a} address limiting cases of the results given here---in the large $N$ and large area limits, respectively. Furthermore, recent groundbreaking works by Chatterjee and others on lattice gauge theory provide intuition for what should happen in the continuum when both $N=\infty$ \emph{and} the loop areas are large~\cite{Chatterjee2019a, chatterjee2016, jafarov2016}. We explain further in Section~\ref{sec-chart}  what happens to the surface expansions in these limiting cases, and how this relates to the other approaches mentioned above. The above-mentioned follow up work in \cite{cao2023random} explains in more depth the relationship between our point of view and the approaches in ~\cite{Chatterjee2019a, chatterjee2016, jafarov2016} and in~\cite{Magee2019,Magee2019a}.

\subsection{The main theorem: a gauge-string duality}\label{sec-intro-duality}

One approach to studying the Yang-Mills measure, as proposed by Wilson~\cite{Wilson1974, wilson2004origins}, is to study~\eqref{eqn-the-YM-measure} on the discrete lattice $\mathbb{Z}^d$. One assigns an independent Haar element $U_e$ in $G$ to each edge, and the holonomy $U_p$ on each unit square (or plaquette) is the product of $U_e$ over its 4 surrounding edges. The Wilson action is defined by summing (the real part of) the trace of $I-U_p$ over all plaquettes, and fitting this into~\eqref{eqn-the-YM-measure} gives the lattice gauge theory (with coupling strength $\beta = \frac{1}{2T}$). The recent breakthrough work mentioned above~\cite{Chatterjee2019a, chatterjee2016, jafarov2016} treats the lattice gauge theory for any dimension and establishes several mathematical results, including the lattice master loop equation and the area law. We refer to the survey~\cite{chatterjee2016a} for an explanation of these results.

It is natural to ask whether we can obtain a continuous theory for any dimension by taking the continuum limit of lattice gauge theory. Unfortunately, the answer is again unknown for $d\ge 3$. In two dimensions, the continuum limit of the lattice theory was developed in the physics literature on the plane by Migdal~\cite{migdal1996recursion} and on general surfaces by Witten~\cite{Witten91}. These results have been systematically organized by L\'evy~\cite{Levy2011a} into a mathematically rigorous framework.


Chatterjee in~\cite{Chatterjee2019a} expressed Wilson loop expectations in lattice gauge theory (for small $\beta$ and the limiting case of $N \to \infty$) as sums over certain sequences of loops, called string trajectories, of their associated weights. Each string trajectory is given a real-valued \emph{weight} and each string trajectory in some sense traces out an embedded discretized surface. The surface traced out by the string trajectory is called a \emph{worldsheet} and integrals over a space of string trajectories are sometimes interpreted as integrals over a space of surfaces.\footnote{To make sense of this interpretation, for any lattice string model, one should imagine that the ``weight'' of a surface---which is formally a planar (or higher genus) map somehow embedded in $\mathbb Z^d$---depends on the ``weighted number of ways'' the surface can be traced out by a string trajectory, which could be interpreted as the partition function for a growth model on the surface. However, this partition function could in principle be simple, e.g.\ if the growth model is a variant of a probability model (like the Eden model) with partition function 1, perhaps with additional weighting depending only on the surface genus and area. Recall that in the Eden model, the faces of the planar map are added to a ``growth set'' one at a time; the boundary of the growth set is a collection of loops, whose evolution over time can be interpreted as a discrete string trajectory. This is at least the heuristic justification for replacing a ``sum over string trajectories'' with a ``sum over spanning surfaces.'' It would take some effort to formalize this idea for the string trajectories in~\cite{Chatterjee2019a} (with its particular string evolution rules, which somehow involve erasing backtracking edges as one goes), and we will not attempt this here. The follow up work \cite{cao2023random} presents several different ways to express a lattice Wilson loop expectation as a sum over discretized surfaces.} This relationship between Wilson loop expectations and sums over string trajectories is a discrete example of a \emph{gauge-string duality}.
Dualities between gauge and string theories---most famously, the \emph{Ads/CFT duality}---remain a popular area of research in theoretical physics following the intensive exploration in~\cite{maldacena1999large,gubser1998gauge}.
However, Chatterjee's formula does not extend to the fine-mesh continuum limit even for $d=2$, since to obtain an interesting limit, one has to let $\beta$ get larger as the mesh size gets finer. In some sense, the problem is that (when $\beta$ is large and/or $N$ is finite) there are \emph{too many} surfaces in the sum, so many that the surface sums become ill-defined differences of infinities.

In dimension two, it was first asserted by Gross and Taylor~\cite{Gross1993,Gross1993a} that 2D Euclidean Yang-Mills theory (or equivalently, 2D quantum chromodynamics) for $G=\UN$ is a string theory, in the sense of having a surface sum interpretation. Their work interprets the known representation-theoretic formula for the partition function and Wilson loop expectations on any surface~(see, e.g.,~\cite{Witten91}) in terms of sums over surfaces. More specifically, the Gross-Taylor expansion expresses the partition function in terms of sums (of quantities involving certain characters of the associated Lie group) over ``ramified covers from worldsheets to a target space'' and Wilson loop expectations in a similar manner.
There are a few limitations of this pioneering work: the form of the summands is a bit complicated and hard to interpret, and the presentation is rather heuristic (it suggests an outline/algorithm for the calculation but is not fully detailed and does not provide an explicit formula). The works~\cite{ramgoolam1996wilson,cordes1997large} address the first issue by interpreting the summands as ``orbifold Euler characteristics of Wilson-Hurwitz spaces,'' though they still somewhat rely on algorithmic descriptions. However, as these works point out,
their interpretations are somewhat \emph{ad hoc}, and each summand is a topological observable associated to a space of surfaces, rather than a single surface. To the best of our knowledge, the Gross-Taylor formula has been made mathematically rigorous only for certain simple collections of loops: those in which each loop traverses the same fixed circle some number of times~\cite{Levy2008}. See the further explanation in Section~\ref{sec-previous-work}. Another issue is that when $G$ is not $\UN$ or $\SUN$, one needs to sum over \emph{non-orientable} surfaces as explained by physicists~\cite{naculich1993two, ramgoolam1994comment},
while ramified covers are always orientable because of their complex structure. This means that the Gross-Taylor gauge-string formulation could not directly extend to other classical Lie groups like $\SON$ or $\SpN$.

The main goal of this paper is to give another string interpretation for Wilson loop expectations on the plane $\BB R^2$. We emphasize that our description of our summands as signed Euler characteristics of surfaces is straightforward, and that the sum includes all surfaces obtained by adding finitely many ramification points to a certain ``reference surface'', which will be explained in Section~\ref{sec-spanning}. Let us also note that while much of the previous work treats heat kernels on Lie groups using representation theory (e.g.\ diagonalizing the Laplace-Beltrami operator in terms of irreducible characters) our approach will be more probabilistic (stochastic calculus, Wick's formula, planar-and-higher-genus map expansions, Poisson point processes, etc.) and may be more accessible to readers with that background.

\begin{thm}\label{thm-main-simple}
	Let $G = \UN$. Consider the Euclidean Yang-Mills theory on the plane with gauge group $G$.
	Let $\bm{\Gamma}$ be an ordered collection of $n$ rectifiable\footnote{The rectifiability is not really necessary; it is just a sufficient condition that allows for the holonomies to be well-defined, see~\cite{levy2010a}.} loops on the plane. The Wilson loop expectation $\Phi_N(\bm{\Gamma})$ of $\bm{\Gamma}$ 
	is equal to a positive constant times the expectation of $\sgn(S)N^{-n+\chi(S)}$ where $S$ is a collection of random topological surfaces spanning $\bm{\Gamma}$ as its boundary (in the sense of Definition~\ref{defn-spanning}) with $\sgn(S)\in\{-1,1\}$ and $\chi(S)$ the sum of the Euler characteristics of the surfaces in $S$. Similar results hold when the gauge group $G$ is one of the classical Lie groups $\SUN,$ $\SON,$ and $\SphN$.
\end{thm}
We state this theorem more precisely for a single loop in Section~\ref{sec-overview} (Lemmas~\ref{lem-main-poisson},~\ref{lem-un} and~\ref{lem-other}), and for a collection of loops in Section~\ref{sec-multiple} (Lemmas~\ref{lem-main-poisson-gen},~\ref{lem-un-gen} and~\ref{lem-other-gen}). 

The idea of expressing Wilson loop expectations in terms of sums of spanning surfaces is classical, but it remains to be seen how useful surface sums will be in higher dimensions---see e.g.\ the “skeptic vs.\ enthusiast” dialog in Section 2 of~\cite{cordes1995lectures}, written in 1995. In this paper, we sum only over flat spanning surfaces, which have minimal area locally (though they do not all have minimal areas in a global sense). As mentioned above, all the surfaces we consider can be obtained by starting with a ``reference'' spanning surface and adding finitely many ramification points in a methodical way.  The sums in this paper are different from the random surface models surveyed, e.g., in~\cite{she2022}, which involve random fractal surfaces that are not flat. It would be interesting if the sums in this paper could be expressed as sums over a larger class of surfaces, including non-flat surfaces, with the non-flat surfaces somehow canceling each other out. (Note that the lattice string trajectories surveyed in e.g.\ \cite{Chatterjee2019a} in some sense trace out spanning surfaces, but in that context the corresponding sum is not limited to flat surfaces.)

\subsection{Applications of the surface interpretation}\label{sec-previous-work}

In addition to its relevance to the overall idea of gauge-string duality, our surface sum has several other immediate corollaries and applications, which we can describe in the context of the existing literature on Euclidean two-dimensional Yang-Mills theory.

\medskip
\noindent
\textit{The master field (i.e., the large-$N$ limit of Yang-Mills).}
\medskip

The large-$N$ limit of Yang-Mills theory was first observed by 't Hooft in~\cite{Hooft1973} and mathematically rigorously stated in~\cite{singer1995}. It was conjectured by Singer that, at least for two-dimensional surfaces, the holonomy of loops under the $\UN$ Yang-Mills measure should converge to a deterministic limit, which is called the \emph{master field}, as $N\to\infty$. When the manifold $\cM$ is the plane, an extension of the convergence of Brownian motion on $\UN$ to free multiplicative Brownian motion is proved in~\cite{Biane1997a}, and L\'{e}vy~\cite{Levy2011a} proved the convergence to the master field not only for $\UN$, but also for $\ON$ and $\SpN$. For the spherical case, this conjecture has been proved in~\cite{dahlqvist2020} and also studied in~\cite{Hall2018a}. Recently,~\cite{lemoine2022large, dahlqvist2022large} proved it for the torus and discussed relevant results for higher genus surfaces.
In Section~\ref{sec-corollaries}, we apply our surface interpretation of the Wilson loop expectation to provide alternative proofs for most of the already known results on the fundamental properties of the master field and Wilson loop expectations on the plane. 

\medskip
\noindent
\textit{The Makeenko-Migdal equation.}
\medskip

Makeenko and Migdal~\cite{makeenko1979} discovered that the Yang-Mills holonomy process satisfies a particular set of differential equations. These equations have been proved in~\cite{Levy2011a} for the plane case and later in~\cite{driver2017a,Driver2017} for general compact surfaces. Furthermore, Driver~\cite{driver2019functional} derived the same result by justifying Makeenko and Migdal's heuristic arguments and making sense of stochastic quantization. The Makeenko-Migdal equation has a particularly simple form for the master field and was used by L\'{e}vy to compute the master field for single loops with up to three self-intersections; see~\cite[Appendix B]{Levy2011a}.  In Section~\ref{sec-mm}, we present a new proof of the Makeenko-Migdal equation in our framework.

\medskip
\noindent
\textit{Explicit Wilson loop expectations.}
\medskip

As we noted above, one special case of the Gross-Taylor expansion~\cite{Gross1993a} has been rigorously established in~\cite{Levy2008} by applying the so-called \emph{Schur-Weyl duality}. In this special case, the gauge group is $\UN$, and each loop in the collection has the form $\ell^n$, where $\ell$ is a fixed simple planar loop. For these loops,~\cite{Levy2008} formulates the Wilson loop expectation explicitly as a sum over ramified covers of the disk.
This is done first by writing the Wilson loop expectations as sums over paths on the symmetric group, and then identifying ramified covers from the bijection between paths on the symmetric group and the corresponding monodromies. See~\cite[page 8]{Biane1997a} for a closed expression for the Wilson loop expectation in this setting.
In Section~\ref{sec-permutations}, we give another proof, which is expressed in terms of a random walk on permutations. Our result provides another way of computing Wilson loop expectations from the transition probability matrix of this random walk. We implement an algorithm that allows a more straightforward computation of various Wilson loop expectations for $G=\UN$ for finite $N$. As mentioned above, Section~\ref{sec-chart} contains a table of Wilson loop expectations generated by Mathematica, which extends the ``master field'' table in~\cite{Levy2011a} (which corresponds to $N=\infty$) to the finite $N$ regime.

\medskip
\noindent
\textit{The master field and the counting of non-crossing forests.}
\medskip

Finally, in Section~\ref{sec-forest} we show that the master field can be computed from the partition function of a model of non-crossing same-color forests (see Section~\ref{sec-forest-cont}), which in particular provides an alternative way of computing the master field for general loops. The techniques in Section~\ref{sec-forest} are somewhat different from those in the rest of the paper.
The proof of the main result of Section~\ref{sec-forest} is based on the moments of the free multiplicative Brownian motions in~\cite[page 8]{Biane1997} together with the moment-cumulant relations from free probability theory (see e.g.\ \cite{Mingo2017}).

\subsection{Basic notation}\label{sec-notation}

\begin{itemize}
 \item For $n \in \BB{N}$, we denote the set $[1,n]\cap\BB Z = \{1,\ldots,n\}$ by $[n]$. 
 \item For a set $A$, we let $\binom{A}{2}$ denote the unordered set of ordered pairs of elements of $A$.
\end{itemize}

\subsection{Background: Yang-Mills holonomy process in 2D}\label{sec-background}

In this section, we briefly review the rigorous construction of the Yang-Mills holonomy process in two dimensions, and we define several important concepts, such as a lasso representation of a loop. The exposition that follows is based mainly on the introductory notes~\cite{Levy2019a} on two-dimensional Yang-Mills theory, and we refer the reader to these notes for further details. We have not included any proofs in our review of these concepts.

Let $\cM$ be a surface (i.e.\ a two-dimensional manifold), and let $G$ be a compact Lie group whose Lie algebra $\frk g$ has an invariant scalar product $\left\langle \cdot, \cdot \right\rangle_{\frk g}$. The well-known works~\cite{sengupta1997,Driver1989,levy2003} showed that, in this setting, we can rigorously define the continuum Yang-Mills measure in terms of the standard heat kernel $p_t(g)$ on $G$. The heat kernel $p: (0,\infty) \times G \to (0,\infty)$ is the solution to 
\begin{equation*} 
\partial_t p_t(g) = \frac{1}{2}\Delta p_t(g)
\end{equation*}
with initial condition $\lim_{t \to \infty} p_t(g) = \delta_1(g)$, where $\Delta$ is the Laplace-Beltrami operator on $\cM$ and $\delta_1$ is the $\delta$-distribution at the identity in $G$. Crucially, the heat kernel is invariant under conjugation, meaning that $p_t(xyx^{-1}) = p_t(y)$ for all $x,y \in G$.

We begin by defining a \emph{lattice Yang-Mills measure} for a fixed graph $\BB{G}$ on $\cM$. We first introduce some preliminary notation. We let $\mathsf{P}(\BB{G})$ denote the set of paths formed by concatenating edges of $\BB{G}$. For a path $\alpha \in \mathsf{P}(\BB{G})$, we denote its starting point by $\ul \alpha$ and its finishing point by $\ol \alpha$. If $\alpha, \beta \in \mathsf{P}(\BB{G})$ satisfy $\ol \alpha = \ul \beta$, we denote their concatenation by $\alpha \beta$.

The lattice Yang-Mills measure assigns a random element $h(\alpha)$ of $G$ to each path $\alpha \in \mathsf{P}(\BB{G})$, such that the following holds almost surely. 
\begin{enumerate}
\item\label{item-ym1}
The element of $G$ assigned to the empty path in $\BB{G}$ is the identity element of $G$.
\item\label{item-ym2}
If $\ol \alpha = \ul \beta$, then $h(\alpha \beta) = h(\alpha) h(\beta)$.
\end{enumerate}
By~\eqref{item-ym1} and~\eqref{item-ym2}, we can characterize the lattice Yang-Mills measure by the random variables it assigns to each edge of $\BB{G}$ (with some chosen orientation). Let $\BB{E}^+$ be the set of edges of $\BB{G}$ with each edge given an arbitrary orientation. We define the lattice Yang Mills measure as a probability measure on this set $G^{\BB{E}^+}$. To describe this measure, we first let $\BB{F}^b$ denote the set of bounded faces of $\BB{G}$, and for each $F \in \BB{F}^b$, we denote the area of $F$ by $|F|$. We let $\partial F$ be a path in $\BB{G}$ that goes once around the face $F$ in the positive direction. The path $\partial F$ is not unique, but the element $h(\partial F)$ of $G$ is defined up to conjugation by an element of $G$, so $p_t(h(\partial F))$ is well-defined.

 

\begin{defn}\label{def-lym}
We define the \textbf{lattice Yang-Mills measure} on $\BB{G}$ with parameter $T>0$ as the probability measure on $G^{\BB{E}^+}$ with density
\begin{equation}\label{eqn-dsf}
Z^{-1} \prod_{F\in\mathbb{F}^b} p_{T|F|}(g_{\partial F})\prod_{e\in\BB{E}_+}dg_e,
\end{equation}
with respect to Haar measure $\prod_{e\in\BB{E}_+}dg_e$ on $G^{\BB{E}^+}$. (The normalizing constant $Z = Z(\BB{G},T)$ is chosen so that the measure is a probability measure.)
\end{defn}

The expression~\eqref{eqn-dsf} is known as the \emph{Driver-Sengupta formula}, and was originally formulated in~\cite{sengupta1997,levy2003}. Throughout the remainder of this work, we set $T=1$ to simplify the resulting computations, though our results easily extend to general $T$. We denote the lattice Yang-Mills measure with $T=1$ by $\mu_{\text{YM}}^{\BB{G}}$.

We note that the lattice Yang-Mills measure on $\BB{G}$ has the following gauge invariance property. Let $(h(\alpha))_{\alpha}$ be the $G$-random variables that $\mu_{\text{YM}}^{\BB{G}}$ associates to the paths $\alpha$ in $\BB{G}$. Then, for any mapping $f$ from the set of vertices of $\BB{G}$ to the gauge group $G$,
\begin{equation} \label{eqn-gauge-invariance-lattice}
    \left(h(\alpha)\right)_{\alpha \in \mathsf{P}(\BB{G})} \stackrel{\mcl L}{\equiv} \left(f(\ol \alpha)^{-1} h(\alpha) f(\ul \alpha)\right)_{\alpha \in \mathsf{P}(\BB{G})}.
\end{equation}

The main result in two-dimensional lattice Yang-Mills theory is that the measure $\mu_{\text{YM}}^{\BB{G}}$ is invariant under subdivision, in the following sense. If we define a graph $\BB{G}'$ from $\BB{G}$ by subdividing and adding some number of edges, then 
each edge $e$ in $\BB{G}$ is a path $\alpha(e)$ in $\BB{G}'$. The measure $\mu_{\text{YM}}^{\BB{G}}$ associates a random element of $G^{\BB{E}^+}$ to the collection of edges $e \in \BB{E}^+$, and the measure
$\mu_{\text{YM}}^{\BB{G}'}$ associates a random element of $G^{\BB{E}^+}$ to the corresponding collection of 
paths $\alpha(e)$ in $\BB{G}'$. The result states that these two random elements of $G^{\BB{E}^+}$ have the same law.

This invariance by subdivision property makes it possible to take a continuum limit of the discrete measures 
$\mu_{\text{YM}}^{\BB{G}}$. We do not justify taking this limit in the present work, referring to~\cite{sengupta1997, Levy2010} for the details. The result of this work is the following theorem (see also~\cite[Theorem 4.1]{Levy2011a}).

\begin{lem}[The Yang-Mills holonomy process]\label{lem-holonomy} 
We can associate a $G$-valued random variable $H(\alpha)$ to every path $\alpha$ in $\cM$, such that the following is true.
\begin{itemize}
 \item For every graph $\BB{G}$, the joint law of the random variables associated to the edges of $\BB{G}$ is given by $\mu_{\text{YM}}^{\BB{G}}$.
 \item
 If $\{\alpha_n\}_{n \in \BB{N}}$ is a sequence of paths started at a given point, and the paths converge uniformly when parametrized at unit speed to a path $\alpha$, then $H(\alpha_n) \to H(\alpha)$ in probability.
\end{itemize}
Moreover, these two properties uniquely characterize the joint law of the random variables associated with all paths in $\cM$.
\end{lem}

We call the collection of $G$-valued random variables described in Theorem~\ref{lem-holonomy} the \emph{Yang-Mills holonomy process}, and we denote their joint law by $\mu_{\text{YM}}$. We call the random variable associated with 
a path (or reduced path) $\alpha$ the \emph{holonomy} of $\alpha$. From the gauge invariance property~\eqref{eqn-gauge-invariance-lattice} of the lattice Yang-Mills measure on $\BB{G}$, it is possible to derive the following gauge invariance property of the Yang-Mills holonomy process. For any mapping $f:M \to G$,
\begin{equation} \label{eqn-gauge-invariance-holonomy}
    \left(H(\alpha)\right)_{\alpha \in \mathsf{P}(\BB{G})} \stackrel{\mcl L}{\equiv} \left(f(\ol \alpha)^{-1} H(\alpha) f(\ul \alpha)\right)_{\alpha \in \mathsf{P}(\BB{G})}.
\end{equation}
This means that if $\alpha$ is not a loop, $H(\alpha)$ is sampled uniformly from the Haar measure on $G$.

Therefore, to describe the Yang-Mills holonomy process, we analyze the holonomies of loops on $\cM$. To describe the holonomies of loops on $\cM$, we first express the set of loops on $\BB G$ in algebraic terms. Given a rooted graph $\BB{G}$ on $\cM$ with root vertex $v$, we consider the set of loops $\alpha$ in $\BB{G}$ based at $v$. We note that the concatenation of two loops based at $v$ is a loop based at $v$. To give this set of loops the structure of a group with the concatenation operation, we need to view $\alpha \alpha^{-1}$ as the constant loop. To accomplish this, we consider two loops \emph{equivalent} if we can construct one from the other by a finite sequence of insertions and erasures of sub-paths of the form $\alpha \alpha^{-1}$. We observe that the set of equivalence classes of loops has the structure of a group. We can represent each equivalence class by its unique \emph{reduced} loop, where we define a reduced loop as a loop with no sub-path of the form $\alpha \alpha^{-1}$. The group structure of the set of 
equivalence classes of loops induces a group structure on the set of reduced loops based at $v$. We denote this group as $\mcl L_v^{\mathsf{red}}(\BB{G})$.

One can show that the group $\mcl L_v^{\mathsf{red}}(\BB{G})$ is a free group. To describe a basis for this group, we define a type of loop on $\BB{G}$ called a \emph{lasso}. See Figure~\ref{fig-eg-lasso-g0} (C) for an example of lassos.

\begin{defn}\label{defn-lasso}
 We define a \textbf{lasso} $\lambda_F$ associated to $F \in \BB{F}^b$ as a loop of the form $\alpha \cdot \partial F \cdot \alpha^{-1}$ for some path $\alpha$ from the root vertex $v$ to a boundary vertex of $F$.
\end{defn}

\begin{figure}[ht!] \centering
    \includegraphics[width=\textwidth]{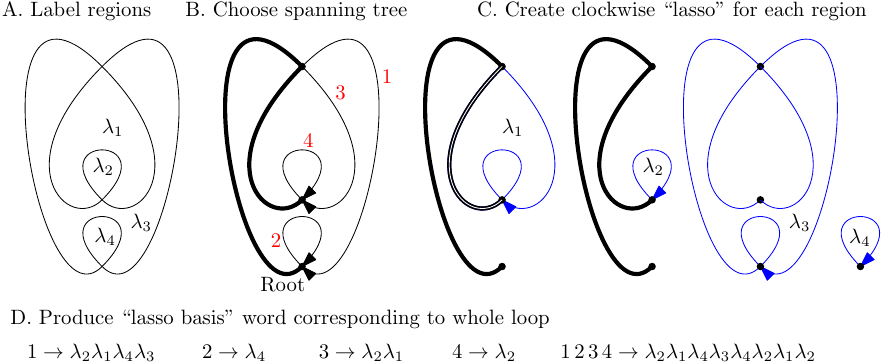}
    \caption{\label{fig-eg-lasso-g0}
    \textbf{(A)} A loop is drawn in the plane with 4 bounded faces. This loop defines the plane graph $\BB G$ with 3 vertices, 6 edges, and 4 faces. \textbf{(B)} A choice of root (the bottom vertex) and spanning tree (thick edges) is indicated. Each non-tree edge is oriented clockwise. \textbf{(C)} A lasso is constructed for each bounded face. Note that the ``loop'' part of the lasso is a (possibly self-intersecting but not self-crossing) loop that traces the full boundary of the (open and simply-connected) $\lambda_i$ region clockwise, starting and ending at the blue arrowhead.  \textbf{(D)} Each edge (labeled by red numbers) can be represented by a word of lassos. Concatenating all edges in order, we obtain a lasso representation of the loop in (A).
    }
\end{figure}

\begin{remark}\label{remark-tree}
If we specify a spanning tree $T$ of $\BB{G}$, then we can canonically define a collection $(\lambda_F)_{\BB{F}^b}$ of lassos associated with the bounded faces of $\BB{G}$. Suppose that $T$ is a fixed spanning tree of $\BB{G}$. We consider the dual graph $\wh{\BB{G}}$ of $\BB{G}$, whose vertices correspond to faces of $\BB{G}$. The spanning tree $T$ determines a spanning tree $\wh{T}$ of $\wh{ \BB{G}}$. This dual tree $\wh T$ associates a distinguished directed edge $e$ of $\BB{G}$ to each bounded face $F \in \BB{F}^b$; namely, if we consider the unique path $\alpha$ in $\wh T$ from $F$ to the unbounded face of $\BB{G}$, then the edge $e$ is the first edge of $\BB{G}$ crossed by $\alpha$, oriented so that it crosses $\alpha$ from right to left. For $F \in \BB{F}^b$ with associated distinguished directed edge $e$, we define the loop $\lambda_F$ as the loop that traverses
\begin{itemize}
\item first the path in $T$ from $v$ to the initial point of $e$,
\item then the boundary of $F$ starting with $e$, 
\item and finally the path in $T$ from the initial point of $e$ back to $v$.
\end{itemize}
\end{remark}

Any set of lassos that we associate to the bounded faces of $\BB{G}$ form a basis for the group of reduced loops based at $v$. For example, Figure~\ref{fig-eg-lasso-g0}\ (A-C) illustrates how we construct lassos from a choice of root vertex and spanning tree for a given loop, and (D) gives a representation of this loop in terms of the lassos we constructed.

\begin{prop}
For each $F \in \BB{F}^b$, let $\lambda_F$ be a lasso associated to $F$. Then the collection $(\lambda_F)_{F \in \BB{F}^b}$ of lassos that we obtain is a basis for the free group $\mcl L_v^{\mathsf{red}}(\BB{G})$. (We call this collection of lassos a \emph{lasso basis} of $\BB{G}$.)
\end{prop}

In other words, 
if $(\lambda_F)_{F \in \BB{F}^b}$ is a lasso basis of $\BB{G}$, then every loop $\Gamma$ in $\BB{G}$ rooted at $v$ can be represented in $\mcl L_v^{\mathsf{red}}(\BB{G})$ as a concatenation of lassos $\lambda_1,\ldots,\lambda_L$; i.e., as
\begin{equation}
    \label{eqn-lasso-rep}
    \lambda_{c(1)}^{\ep(1)} \cdots \lambda_{c(M)}^{\ep(M)}
\end{equation}
for some \emph{coloring map} $c:[M] \to [L]$ and $\ep:[M] \to \{-1,1\}$.

\begin{defn}\label{defn-lasso-rep}
    Let $\Gamma$ be a loop in $\cM$ that can be represented as an element of $\mcl L_v^{\mathsf{red}}(\BB{G})$ for some rooted graph $(\BB{G}, v)$. Given a lasso basis $(\lambda_F)_{F \in \BB{F}^b}$ of $\BB{G}$, we call the representation~\eqref{eqn-lasso-rep} of $\Gamma$ its \textbf{lasso representation} in $\BB{G}$.
\end{defn}

We may now characterize the joint law of the holonomies of loops in $\mcl L_v^{\mathsf{red}}(\BB{G})$. The following proposition may be derived from Definition~\ref{def-lym}. (We consider just the case in which $\cM$ is the plane, since this is the surface that we study in this work.)

\begin{prop}\label{prop-hol-bm}
Suppose that $\cM$ is the plane. The holonomies of the lassos of $\BB{G}$ are jointly independent, and the law of the holonomy of each lasso $\lambda_F$ is given by $p_{|F|}(g) dg$. 
\end{prop}

We can equivalently describe the law of the holonomy of a lasso in terms of \emph{Brownian motion} on $G$.
We recall that Brownian motion $U_t$ on a general Riemannian manifold started from a point $w$ is defined as the Markov process started from $w$ with generator one-half the Laplace-Beltrami operator associated to the manifold.\footnote{For Lie groups, the definition here does not involve the choice of specifying left or right independent increment, while in the SDE description~\eqref{bm-sde} below we will specify it to be left L\'{e}vy process and have right independent increment. However, it has been argued in~\cite[Lemma 1.4]{Levy2011a} that $(U_t)_{t\ge 0}$ and $(U_t^*)_{t\ge 0}$ have the same distribution and essentially the main results shall not depend on this choice.} This definition of Brownian motion directly implies that the holonomy of a lasso has the law of Brownian motion on $G$, started from the identity of $G$, at time $|F|$. 

We now define one of the most important scalar observables of the Yang-Mills measure.

\begin{defn}[Wilson loop expectation]\label{defn-wilson-loop}
We define the \textbf{Wilson loop} corresponding to a rectifiable loop $\Gamma$ in $\cM$ as the trace of 
its holonomy. For a collection $\bm \Gamma$ of rectifiable loops in $\cM$, we define their \textbf{Wilson loop expectation} (or the \textbf{Yang-Mills holonomy field}) as the $\mu_{\text{YM}}$-expectation of the product of the corresponding Wilson loops, denoted by $\Phi_N^G(\bm \Gamma)$.
\end{defn}

The Wilson loop expectations play the role of $n$-point functions in the theory, and one can show (from the compactness of $G$) that the Wilson loop expectations characterize the joint distribution of the conjugacy classes of the holonomies of all loops in $\mcl L_v^{\mathsf{red}}(\BB{G})$.\footnote{In fact, for many classes of gauge group $G$, such as those we consider in this paper, the Wilson loop expectations determine the joint law of the holonomies of the loops in $\mcl L_v^{\mathsf{red}}(\BB{G})$ up to \emph{simultaneous conjugation} by an element of $G$. See~\cite[Section 1.5]{Levy2019a} for a more detailed explanation.}

\section{Overview of main results with an example}\label{sec-overview}
Let $G=\UN$, and consider a loop $\Gamma_0$ that we illustrated in Figure~\ref{fig-eg-lasso-g0} (A), based at its root vertex (the bottom intersection). This loop is embedded on the plane with smooth curves, and it divides the plane into four finite faces. We can describe these four faces as the regions enclosed by the clockwise-directed loops $\lambda_1,\ldots,\lambda_4$ based at the root with a choice of spanning tree, called \emph{lassos} (Definition~\ref{defn-lasso}).

In the case of $\Gamma_0$, we can write its \emph{lasso representation} (Definition~\ref{defn-lasso-rep}) as
\begin{equation} \label{eqn-lasso-rep-g0}
\Gamma_0 \equiv \lr{2}\lr{1}\lr{4}\lr{3}\lr{4}\lr{2}\lr{1}\lr{2}
\end{equation}
which was demonstrated in Figure~\ref{fig-eg-lasso-g0}\ (D).

We denote by $\abs{\lambda_i}$ the Lebesgue area enclosed by $\lambda_i$. By a variational method and It\^o's formula, the Wilson loop expectation $\Phi_N(\Gamma_0)=\Phi_N^{\UN}(\Gamma_0)$ of $\Gamma_0$ can be computed as
\begin{align}
 \Phi_N(\Gamma_0) =& Ne^{-\frac{2\ls{1}+3\ls{2}+\ls{3}+2\ls{4}}{2}}\biggl(-\frac{N^2-1}{3}\cosh\frac{\ls{1}}{N}\nonumber \\
 &+\frac{N^2+2}{3}\cosh\frac{3\ls{2}+\ls{1}}{N}
    -N\sinh\frac{3\ls{2}+\ls{1}}{N}\biggr) \biggl(\cosh\frac{\ls{4}}{N}-N\sinh\frac{\ls{4}}{N}\biggr).
    \label{eqn-eg-g0}
\end{align}
For example, see~\cite[\S 2.4]{Levy2019a} which explains how such calculations can be done with simpler examples and discusses possible extensions to more complicated loops. In Section~\ref{sec-permutations}, we provide an algorithmic method (Theorem~\ref{thm::walk-on-permutations}) for computing Wilson loop expectations, producing a chart in Section~\ref{sec-chart}.

Our main result, Theorem~\ref{thm-main-simple}, allows us to obtain~\eqref{eqn-eg-g0} for the Wilson loop expectation of $\Gamma_0$ by taking a probability-centered approach rather than using techniques from representation theory. Moreover, our theorem expresses the Wilson loop expectation of $\Gamma_0$ as a sum over surfaces ``spanning'' the loop $\Gamma_0$. To state Theorem~\ref{thm-main-simple} in more precise terms, we state it in two parts. 
\begin{itemize}
 \item 
We first express the Wilson loop expectation of a loop $\Gamma$ for general gauge groups $G$ in terms of a Poisson point process on a space defined by the loop.
 \item
We then interpret the first result as a sum over surfaces for the classical gauge groups $\UN$, $\SON$, $\SUN$, and $\SphN$.
\end{itemize}
We restrict to the case in Theorem~\ref{thm-main-simple} of a single loop, deferring our treatment of the general case of multiple loops to Section~\ref{sec-multiple}.\footnote{The case of multiple loops is treated essentially the same way, with only very minor adjustments, but the statements and proofs are much more cumbersome because of the extra notation needed. Note that in~\cite{Levy2011a}, the product of traces of holonomies obtained from multiple loops is called a ``Wilson skein'' instead of a Wilson loop.}

\subsection{Statements of the main results}\label{sec-main-results}

In our first lemma, we take the gauge group $G$ to be an arbitrary compact connected\footnote{We need to assume $G$ is connected in order to approximate Brownian motion on $G$ by a random walk on the associated Lie algebra $\frk g$. See Proposition~\ref{prop-rw-convergence}.} Lie group. We consider a loop
$\Gamma$ in the plane with lasso representation~\eqref{eqn-lasso-rep}, and we express its Wilson loop expectation in terms of a Poisson point process $\Sigma$ in the space
\begin{equation}
 \label{eqn-def-mcl-D-1}
\mcl D := \bigsqcup_{\stackrel{m<m^*}{c(m) = c(m^*)}} D_{(m,m^*)},
\end{equation}
where $D_{(m,m^*)}$ on the plane is isomorphic to the region bounded by $\lambda_{c(m)} \equiv \lambda_{c(m^*)}$. We call $\mcl D$ the \emph{space of matching-color lasso pairs}; see Definition~\ref{defn-space-match} below.

To state the lemma, we define a \emph{pairing} $\pi \in \frk S_{2n}$ (for $n$ a non-negative integer) as a permutation whose cycles all have size $2$. Also, for a pairing $\pi$ and a collection of random variables $Z_1,\ldots,Z_{2n}$, we define $\left\langle Z_1, \ldots, Z_{2n} \right\rangle_\pi$ as in Lemma~\ref{lem-wicks}.

\begin{lem}\label{lem-main-poisson}
Let $G$ be a compact connected Lie group with associated Lie algebra $\frk g$, and let $W$ be the standard Brownian motion on $\frk g$ at time $1$, as defined in~\eqref{eqn-bm-lie-algebra}.
Consider the loop $\Gamma$ with lasso representation~\eqref{eqn-lasso-rep}. 
Let $\Sigma$ be a Poisson point process in the space $\mcl D$ of matching-color lasso pairs (Definition~\ref{defn-space-match}) whose intensity is the Lebesgue measure, and define the sign $\ep(\Sigma) \in \{-1,1\}$ and pairing $\pi(\Sigma) \in \frk S_{2|\Sigma|}$ associated to $\Sigma$ as in Definitions~\ref{def-epsilon-sigma} and~\ref{defn-pi}. 
Then the Wilson loop expectation of $\Gamma$ in the gauge group $G$ is given by
\begin{equation}\label{eqn-constant-term}
\exp\left(\frac{\frk c_{\frk g}}{2} \sum_{m=1}^M \abs{\lambda_{c(m)}} + \sum_{\stackrel{m<m^*}{c(m) = c(m^*)}} |\lambda_{c(m)}| \right)
\end{equation}
(with the constant $\frk c_{\frk g}$ defined in~\eqref{bm-sde}) times the expected value of $\ep(\Sigma)$ times the sum of
\begin{equation}
 \label{eqn-exp-single}
 \avg{W_{a_1a_2}, \ldots, W_{a_{2|\Sigma|} a_{2|\Sigma|+1}}}_{\pi(\Sigma)}
\end{equation}
over all $a_1,\ldots,a_{2|\Sigma|+1} \in [N]$ with $a_{2|\Sigma|+1} = a_1$.
\end{lem}

We can represent the Poisson point process $\Sigma$ in Lemma~\ref{lem-main-poisson} as follows. We can represent the $M$ lassos in the lasso representation~\eqref{eqn-lasso-rep} of $\Gamma$ by a ``rose graph'' with $M$ petals, where the petals represent the lassos in~\eqref{eqn-lasso-rep} in cyclic order around the graph. We have drawn this rose graph in Figure~\ref{fig-eg-surface-g0}\ (G) for the example of the loop $\Gamma_0$, where there are $M=8$ petals. We can represent each point in $\Sigma$ by a pair of points in two matching petals, as illustrated in Figure~\ref{fig-eg-surface-g0}\ (H), where there are $\abs{\Sigma} = 5$ pairs of points. Thus, the points in $\Sigma$ determine $2|\Sigma|$ points in the rose graph.

\begin{figure}[ht!] \centering
	\includegraphics[width=\textwidth]{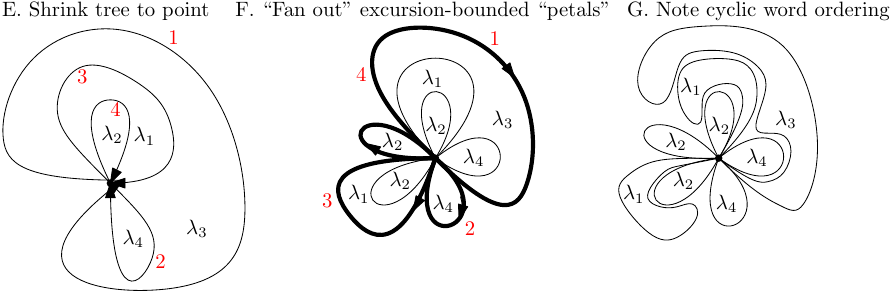}
	\includegraphics[width=0.9\textwidth]{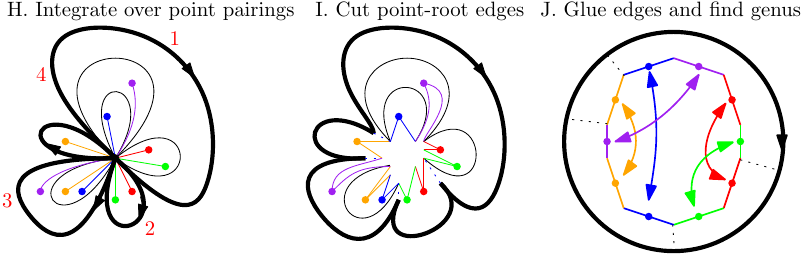}
	\caption{\label{fig-eg-surface-g0}
     {\bf Construction of a rose graph:} (E-G) illustrate how to construct a rose graph associated with the example in Figure~\ref{fig-eg-lasso-g0}. The number of petals is equal to the length of the lasso representation. In fact, this construction of a rose graph suggests an alternative way to obtain the lasso representation of a given loop with the specified root and spanning tree, by reading petals in the clockwise order.
    {\bf Sum over surfaces:} The integral (H) is over all compatible pairings with finitely many pairs, with weight depending on the genus of the surface constructed in (I) and (J).
    }
\end{figure}

Now, suppose that we draw a slit from each of these $2|\Sigma|$ points to the center of the rose graph, and then we cut along these $2|\Sigma|$ slits, so that each slit becomes an edge of a petal. See Figure~\ref{fig-eg-surface-g0}\ (I). 
These $2|\Sigma|$ new edges come in pairs, since each point in $\Sigma$ gives rise to two of these edges. By gluing or contracting each of these $|\Sigma|$ pairs of edges, we can construct a surface that ``spans'' the loop $\Gamma$, in the sense that the boundary of the new surface is a covering space of the original loop, which is explained further in Section~\ref{sec-spanning}.

If the gauge group $G$ is one of the classical Lie groups, then we can express the expectation~\eqref{eqn-exp-single} as a sum over these ``spanning'' surfaces.
In stating this result precisely, we will diverge slightly from the intuition we just described, by essentially ignoring the boundaries of the original petals and focusing on how we glue or contract the new edges formed by the $2|\Sigma|$ slits. We start with a region $H$ whose boundary is a directed cycle graph with $2|\Sigma|$ (oriented) edges with vertices labeled $1,\ldots,2|\Sigma|$ in cyclic order. 
We interpret the pairing $\pi(\Sigma)$ as a way to partition the edges of $H$ into $|\Sigma|$ pairs; namely, We consider two edges paired if their vertex labels modulo $2|\Sigma|$ are $p,p+1$ and $q,q+1$, where $p$ and $q$ are transposed by $\pi(\Sigma)$.\footnote{In terms of our intuitive explanation above, we can think of the boundary of $H$ as the edges formed by cutting along the slits in Figure~\ref{fig-eg-surface-g0}\ (H), and the pairing $\pi(\Sigma)$ identifies pairs of edges formed from the same point in $\Sigma$.}
We construct surfaces from $H$ by gluing or contracting pairs of edges paired by $\pi(\Sigma)$ as described in Table~\ref{table-gluing}.

\begin{table}[htbp]
\centering
\begin{tabular}{ll}
\toprule
Lie group $G$ 
& \\ 
\midrule
$\UN $ 
& We glue the pair of edges with the opposite orientation.\\
$\SON$ 
& We glue the pair of edges with the same or the opposite orientation.\\
$\SUN$ 
& We either contract both edges or glue them with the opposite orientation.\\
$\SphN$ 
& We glue the pair of edges with the same or the opposite orientation.\\
\bottomrule
\end{tabular}
\caption{The gluing/contracting operations on each pair of edges matched by $\pi(\Sigma)$.}\label{table-gluing}
\end{table}
We then interpret~\eqref{eqn-exp-single} as a sum indexed by the set of possible surfaces we can construct. The simplest case is $G = U(N)$, since for this gauge group, we have only one choice of operation, and so the sum over surfaces reduces to a single term.

\begin{lem}\label{lem-un}
 Let $G = \UN$. Then~\eqref{eqn-exp-single} is equal to ${(-1)}^{|\Sigma|} N^{-1+\chi}$, where $\chi$ is the Euler characteristic of the surface formed from $H$ by gluing pairs of edges matched by $\pi(\Sigma)$ with the opposite orientation.
\end{lem}

For each of the other three gauge groups, we have $2^{|\Sigma|}$ ways to construct a surface, and we express~\eqref{eqn-exp-single} as a sum over these $2^{|\Sigma|}$ surfaces, where the summand is a function of the following four observables.
\begin{itemize}
 \item the Euler characteristic $\chi$ of the surface,
 \item the number $\alpha$ of pairs of edges glued with the opposite orientation,
 \item the number $\gamma$ of contracted edges (which is always even), and
 \item the non-orientable genus $\mu$ of the surface (which is zero if the surface is orientable).
\end{itemize}

\begin{lem}\label{lem-other}
    Let $G = \SON$, $\SUN$ or $\SphN$. Let $H$ be defined as above, and consider the $2^{|\Sigma|}$ ways to construct a surface from $H$ by gluing or contracting each of the $|\Sigma|$ pairs of edges matched by $\pi(\Sigma)$ according to one of the two operations listed in Table~\ref{table-gluing}. Then we can express~\eqref{eqn-exp-single} as the sum, over this set of $2^{|\Sigma|}$ surfaces, of
    \begin{equation}
    \label{eqn-cov-N-1}
        \begin{cases}
        {(-1)}^{\alpha} N^{-1+\chi}, & \text{if $G=\SON$} \\
            {(-1)}^{\alpha} N^{-1- \gamma +\chi}, & \text{if $G=\SUN$} \\
            {(-1)}^{\alpha+\mu} N^{-1+\chi}, & \text{if $G=\SphN$} 
        \end{cases}
    \end{equation}
\end{lem}

To demonstrate how we can apply Lemma~\ref{lem-main-poisson} to compute the Wilson loop expectation of any loop $\Gamma$, we present another example of the eight-shaped loop $\Gamma_1$ depicted in Figure~\ref{fig-eg-surface-g1}\ (A) with $G = \UN$. In this case, the lasso representation of $\Gamma_1$ corresponding to the choice of root vertex and spanning tree described in Figure~\ref{fig-eg-surface-g1} is
\begin{equation}\label{eqn-lasso-rep-ex}
\Gamma_1 \equiv \lambda_2^{-1} \lambda_3^{-1} \lambda_1^{-1} \lambda_2 \lambda_4 \lambda_1.
\end{equation}
This is slightly more general than the previous example of $\Gamma_0$ in that the lasso representation may contain the inverses of lassos.

 \begin{figure}[ht!] \centering
	\includegraphics[width=\textwidth]{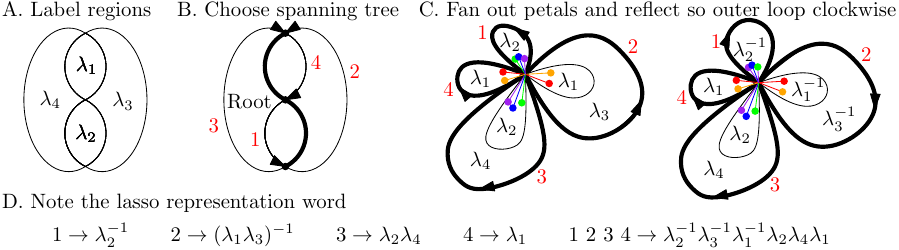}
	\includegraphics[width=0.95\textwidth]{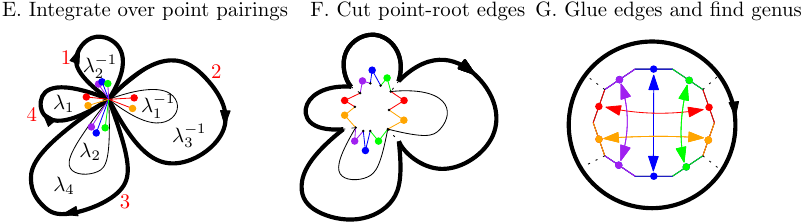}
	\caption{\label{fig-eg-surface-g1}
    \textbf{(A-D)} Construction of the lassos, the lasso representation, and the rose graph for the loop $\Gamma_1$ with the specified root and spanning tree. \textbf{(E-G)}
    An example of the surface $H$ in Lemma~\ref{lem-un} for the loop $\Gamma_1$ with $G = \UN$.
    }
\end{figure}

Recalling the definitions of $c$ and $\ep$ from~\eqref{eqn-lasso-rep}, this means that the space
of matching-color lasso pairs is given by
\begin{equation} \label{eqn-space-example-g0}
D_{(1,4)} \oplus D_{(3,6)},
\end{equation}
where $D_{(1,4)}$ and $D_{(3,6)}$ are isomorphic to the regions bounded by the lassos $\lambda_2$ and $\lambda_1$, respectively. 
Let $\Sigma$ be a Poisson point process in~\eqref{eqn-space-example-g0} with $k$ points in $D_{(3,6)}$ and $m$ points in $D_{(1,4)}$. We have $\ep(\Sigma) = (-1)^{k+m}$. Following Lemma~\ref{lem-un}, we consider a surface $H$ whose boundary is a directed cycle graph with $2 |\Sigma| = 2(k+m)$ edges. See Figure~\ref{fig-eg-surface-g1}\ (E) for an illustration of $H$ with $k=2$ and $m=3$. (The edges corresponding to the points in $D_{(3,6)}$ are red; the edges corresponding to the points in $D_{(1,4)}$ are in blue.) The pairing $\pi(\Sigma)$ corresponds to matching the $2(k+m)$ edges exactly as illustrated in Figure~\ref{fig-eg-surface-g1}\ (G). By Lemma~\ref{lem-un},~\eqref{eqn-exp-single} is equal to $(-1)^{|\Sigma|} N^{-1+\chi}$, where $\chi$ is the Euler characteristic of the surface formed by gluing pairs of matching edges with the opposite orientation. We can observe directly from Figure~\ref{fig-eg-surface-g1}\ (G) that the surface we obtain from this gluing is a torus if $k$ and $m$ are both nonzero, and a sphere if either $k$ or $m$ is zero. Therefore, by Lemma~\ref{lem-main-poisson}, the Wilson loop expectation of $\Gamma_0$ is equal to
\[
e^{-\frac{2\abs{\lambda_1}+2\abs{\lambda_2}+\abs{\lambda_3}+\abs{\lambda_4}}{2}}
\left(N\left(1+\sum_{k=1}^\infty\frac{\abs{\lambda_1}^k}{k!} + \sum_{m=1}^\infty \frac{\abs{\lambda_2}^m}{m!}\right) + \frac{1}{N}\sum_{k,m=1}^\infty \frac{\abs{\lambda_1}^k\abs{\lambda_2}^m}{k!m!}\right) \] or \[ N e^{-\frac{2\abs{\lambda_1}+2\abs{\lambda_2}+\abs{\lambda_3}+\abs{\lambda_4}}{2}}
\left(e^{\abs{\lambda_1}}+e^{\abs{\lambda_2}}-1+\frac{1}{N^2}(e^{\abs{\lambda_1}}-1)(e^{\abs{\lambda_2}}-1)\right).
\]
This agrees with the Wilson loop expectation of $\Gamma_1$ computed in~\cite[\S 2.4]{Levy2019a}. In principle, the expression~\eqref{eqn-eg-g0} for $\Phi_N(\Gamma_0)$ can be obtained for the previous loop $\Gamma_0$ in the same way, while computing the genus for each surface term is not straightforward compared to the case of $\Gamma_1$. Indeed, we provide a simple algorithm in Section~\ref{sec-permutations} that allowed us to compute~\eqref{eqn-eg-g0}.

\subsection{An outline of the proofs}\label{sec-overview-poisson}

We now outline the steps in our proofs of Lemmas~\ref{lem-main-poisson} and~\ref{lem-un}, and we highlight the key ideas and intuition for each step while glossing over some technical details. (The proof of Lemma~\ref{lem-other} follows the same main idea as that of Lemma~\ref{lem-un}, with some additional details.) We present the full proof of Lemma~\ref{lem-main-poisson} in Section~\ref{sec-poisson-sum}, and the proofs of Lemmas~\ref{lem-un} and~\ref{lem-other} in Section~\ref{sec-surface-story}.

\medskip
\noindent
\textit{Step 1: Write the holonomy as a product of Brownian motions on $G$.}
\medskip

\noindent
As we explain further in Section~\ref{sec-background}, it is well-known that the holonomy of a \emph{simple} loop $\lambda$, which is homeomorphic to a circle, has the same law as the \emph{standard Brownian motion} on $G$ run for time $\abs{\lambda}$. By identifying $G$ as a subgroup of $\GLNC$, we may view the Brownian motion on $G$ as a stochastic process on the space of $N\times N$ complex matrices. 
Moreover, the holonomies of any collection of face-disjoint simple loops are \emph{independent}. (See Proposition~\ref{prop-hol-bm}.) This means that 
the holonomy of each lasso $\lambda_i$ is a Brownian motion $U^{(i)} = U^{(i)}_t$ run for time $|\lambda_i|$, and the Brownian motions $U^{(1)},\ldots,U^{(4)}$ are independent. Moreover, the holonomy of $\lambda_i^{-1}$ has the same law as $\left( {U^{(i)}_{|\lambda_i|}} \right)^{-1}$, and the holonomy of a loop represented as a product of lassos $\lambda_i$ is given by the associated product of the variables $U^{(i)}_{|\lambda_i|}$.
This allows us to express the Wilson loop expectation of the 8-shaped loop $\Gamma_1$ in Figure~\ref{fig-eg-surface-g1} by
\begin{equation}\label{eqn-holonomy-example}
{U^{(2)}_{\abs{\lambda_2}}}^{-1}{U^{(3)}_{\abs{\lambda_3}}}^{-1}{U^{(1)}_{\abs{\lambda_1}}}^{-1}{U^{(2)}_{\abs{\lambda_2}}}{U^{(4)}_{\abs{\lambda_4}}}{U^{(1)}_{\abs{\lambda_1}}}.
\end{equation}

\medskip
\noindent
\textit{Step 2: Approximate the Brownian motions by random walks on the Lie algebra.}
\medskip

\noindent
The Wilson loop expectation is the expected trace of the holonomy---in our example, the expectation of the trace of~\eqref{eqn-holonomy-example}. 
To compute this expectation, we approximate each Brownian motion $U^{(i)}_{\abs{\lambda_i}}$ by a \emph{random walk} on $\GLNC$. (See Definition~\ref{defn-hol-approx} and Proposition~\ref{prop-conv-prob-approx}.) We derive this first-order approximation from the expression of Brownian motion on $G$ as an SDE that `wraps' the linear Brownian motion on its Lie algebra $\uN$. The approximation we obtain is
\begin{equation}\label{eqn-rw-approx-heuristic}
	U^{(i)}_{\abs{\lambda_i}}\approx \prod_{j=1}^J \left((1- \frac{\abs{\lambda_i}}{2})I_N + \frac{W_{\abs{\lambda_i}}^{(i, j)}}{\sqrt{J}}\right)
\end{equation}
where the processes $W^{(i,j)}_t$ for $i=1,2,3,4$ and $1\le j\le J$ are independent standard Brownian motions on $\uN$. The $\approx$ symbol in~\eqref{eqn-rw-approx-heuristic} means that the random walk approximation in~\eqref{eqn-rw-approx-heuristic} converges to $U^{(i)}_{\abs{\lambda_i}}$ in probability as $J\to \infty$. To obtain the random walk approximation of the inverse of $U^{(i)}_{\abs{\lambda_i}}$, we simply reverse the order of the product in~\eqref{eqn-rw-approx-heuristic} and switch the sign of the second term in each factor in the product. In other words,
\begin{equation}\label{eqn-rw-approx-heuristic-inv}
	{\left({U^{(i)}_{\abs{\lambda_i}}}\right)}^{-1}\approx	\prod_{j=J}^1 \left((1-\frac{\abs{\lambda_i}}{2})I_N - \frac{W_{\abs{\lambda_i}}^{(i, j)}}{\sqrt{J}}\right)
\end{equation}
If we replace each Brownian motion in the holonomy~\eqref{eqn-holonomy-example} by its random walk approximation, we obtain a large product of $6J$ terms
where each term of this product has the form
\begin{equation}\label{eqn-term-form-outline}
\underbrace{\left(1-\frac{\abs{\lambda_i}}{2J}\right)}_{\mytag{A}{termA-outline}} I_N \pm
\underbrace{\frac{W_{\abs{\lambda_i}}^{(i, j)}}{\sqrt{J}}}_{\mytag{B}{termB-outline}}.
\end{equation}
In the example of the 8-shaped loop in Figure~\ref{fig-eg-surface-g1}, the holonomy approximation is given by
\begin{equation}
 	\label{eqn-large-product-outline}
\begin{split}
	\prod_{j=J}^1 \left((1-\frac{\abs{\lambda_2}}{2})I_N - \frac{W_{\abs{\lambda_2}}^{(2, j)}}{\sqrt{J}}\right)
	\prod_{j=J}^1 \left((1-\frac{\abs{\lambda_3}}{2})I_N - \frac{W_{\abs{\lambda_3}}^{(3, j)}}{\sqrt{J}}\right)\\
	\prod_{j=J}^1 \left((1-\frac{\abs{\lambda_1}}{2})I_N - \frac{W_{\abs{\lambda_1}}^{(1, j)}}{\sqrt{J}}\right)
	\prod_{j=1}^J \left((1-\frac{\abs{\lambda_2}}{2})I_N + \frac{W_{\abs{\lambda_2}}^{(2, j)}}{\sqrt{J}}\right)\\
	\prod_{j=1}^J \left((1-\frac{\abs{\lambda_4}}{2})I_N + \frac{W_{\abs{\lambda_4}}^{(4, j)}}{\sqrt{J}}\right)
	\prod_{j=1}^J \left((1-\frac{\abs{\lambda_1}}{2})I_N + \frac{W_{\abs{\lambda_1}}^{(1, j)}}{\sqrt{J}}\right).
\end{split}
\end{equation}

If we expand this product to a sum of $2^{6J}$ terms, each term of the sum is a product 
\begin{equation}\label{eqn-product-WW-outline}
\prod_{k=1}^K \ep(m_k) \frac{W_{\abs{\lambda_{c(m_k)}}}^{(c(m_k), j_k)}}{\sqrt{J}}
\end{equation}
of some number $K$ of matrices of the form~\eqref{termB-outline}, multiplied by $6J-K$ scalars of the form~\eqref{termA-outline}. We show by a series of arguments that, in computing the expected trace of the holonomy, we can fix $K$ and take $J \to \infty$
for each term of the sum separately. If we take the product of $6J-K$ scalars of the form~\eqref{termA-outline}, and we send $J \to \infty$ limit while keeping $K$ fixed, then we obtain the same limit as the limit
\[
\lim_{J \to \infty} \prod_{m=1}^6 \left(1-\frac{\abs{\lambda_{c(m)}}}{2J}\right)^J = \exp\left(-\frac{2\abs{\lambda_{1}}+2\abs{\lambda_{2}}+\abs{\lambda_{3}}+\abs{\lambda_{4}}}{2} \right),
\] of all $6J$ scalars of the form~\eqref{termA-outline}.

\begin{figure}[ht!] \centering
	\begin{tabular}{cc} 
	\includegraphics[width=0.45\textwidth]{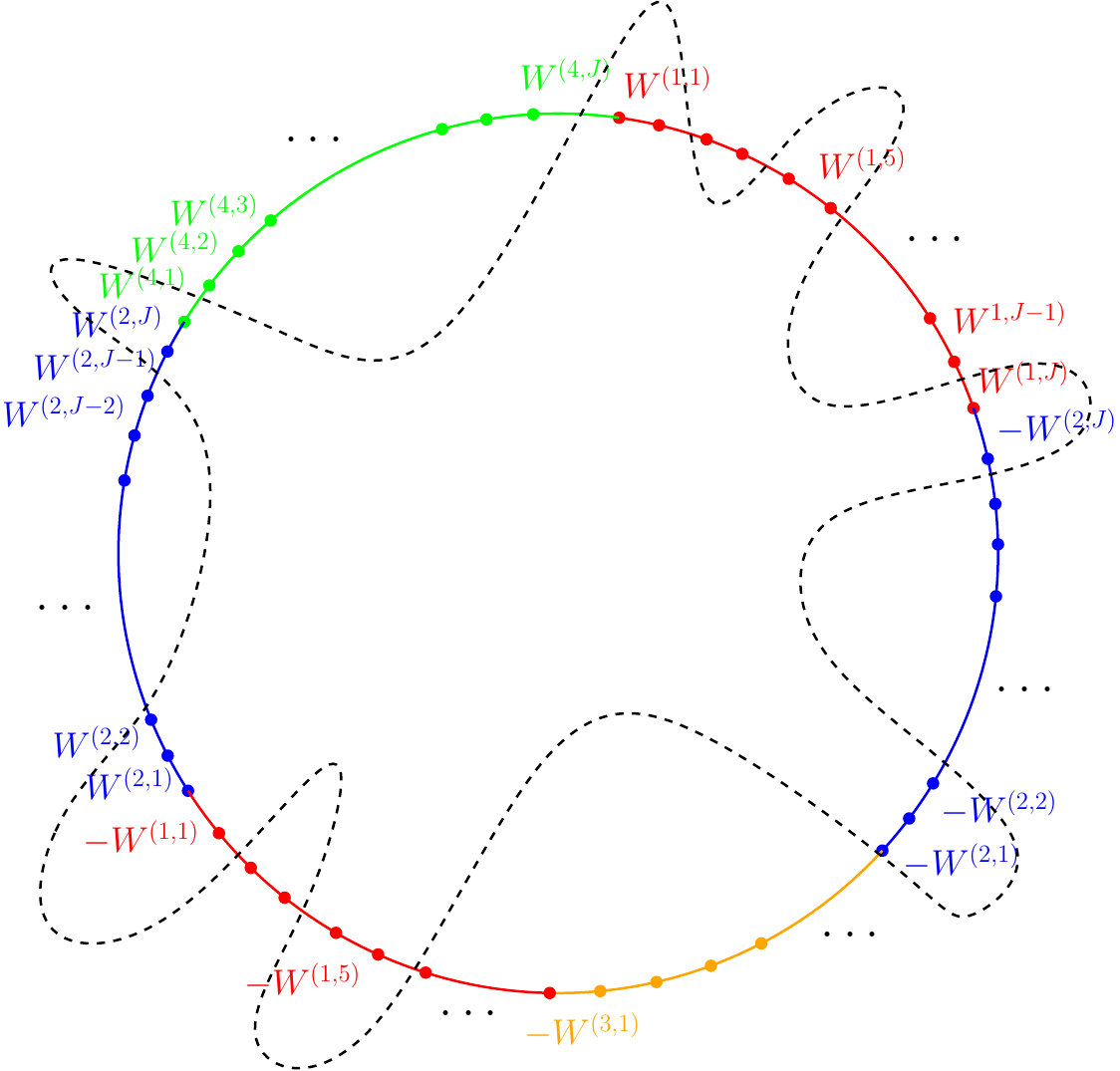}
	&
	\includegraphics[width=0.45\textwidth]{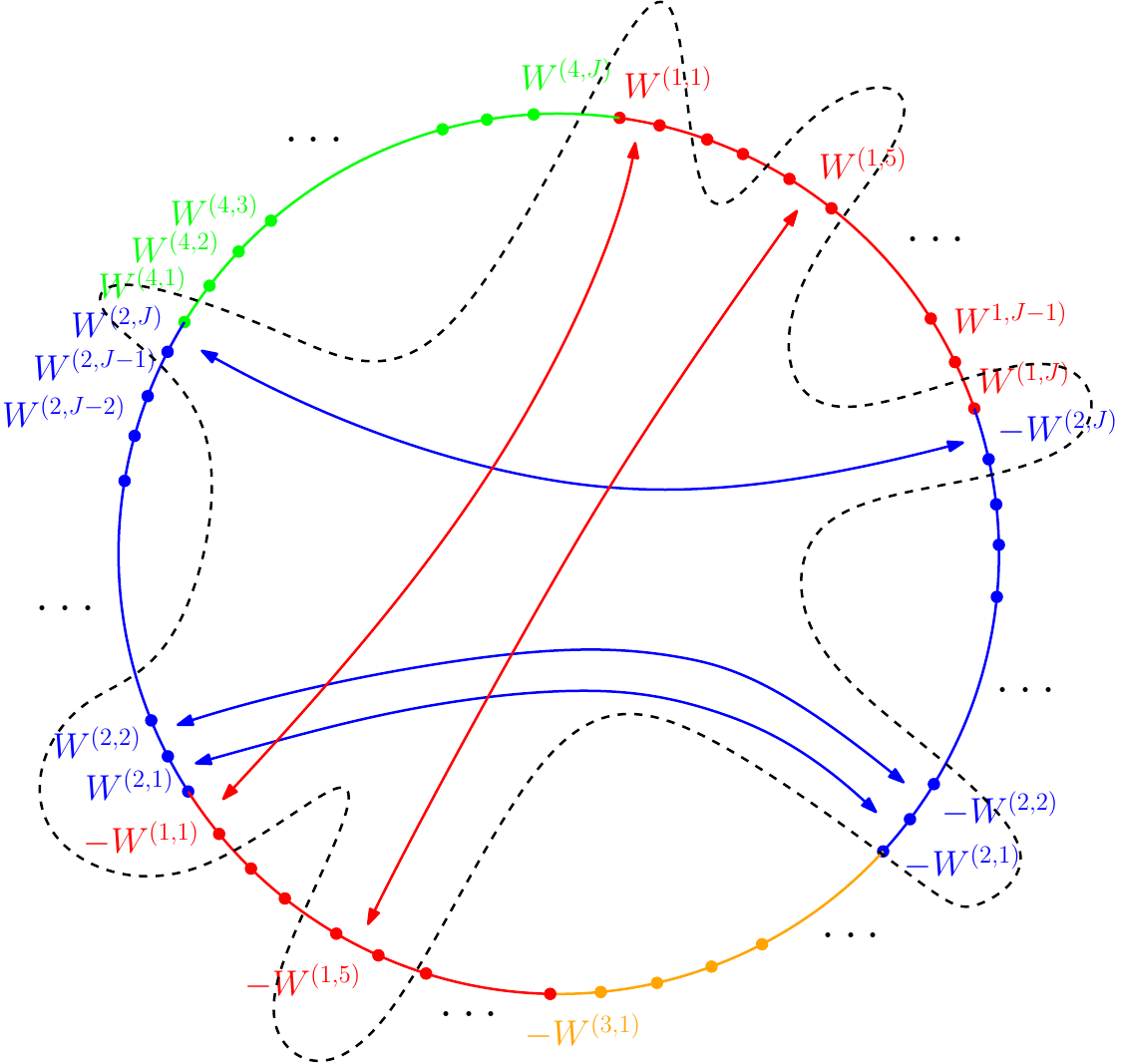}
	\end{tabular}
	\caption{A schematic diagram illustrating a term~\eqref{eqn-product-WW-outline} in the expanded product of random walk approximations~\eqref{eqn-term-form-outline}, for the example loop $\Gamma_1$ in Figure~\ref{fig-eg-surface-g1}. Each colored arc represents one of the six lassos in the lasso representation~\eqref{eqn-lasso-rep-ex} of $\Gamma_1$. The two arcs corresponding to $\lambda_1^{\pm 1}$ are colored red, the two arcs corresponding to $\lambda_2^{\pm 1}$ are colored blue, and the arcs corresponding to $\lambda_3^{-1}$ and $\lambda_4$ are green and orange. Every arc is divided by $J$ segments, and the location index $j=1,2\cdots, J$ is assigned in the clockwise or counterclockwise order depending on whether the corresponding lasso appears as its inverse in the lasso representation. \textbf{Left:} Each segment corresponds to a term~\eqref{eqn-term-form-outline} in the large product~\eqref{eqn-large-product-outline}. Each term of the expansion of~\eqref{eqn-large-product-outline} corresponds to a \emph{subset} of these $6J$ segments. (We have circled an example of a subset with the dotted line.) \textbf{Right:} Given a term of the expansion, its expected trace is non-zero only if the segments in the corresponding subset can be ``paired properly'' according to Wick's formula. That is, if two segments are paired, their color $i$ and location index $j$ should both match. In the example we have illustrated in the figure, there is only \emph{one} such pairing of the $10$ segments in our subset, which we have depicted with arrows.
	}\label{fig-expansion-diagram}
\end{figure}

\medskip
\noindent
\textit{Step 3: Simplify by Wick's formula}
\medskip

\noindent
We now focus on the product of matrices~\eqref{eqn-product-WW-outline}. The trace of this matrix is a sum of product of Gaussians; therefore, to compute its expectation, we apply \emph{Wick's formula} (also known as \emph{Isserlis’ theorem}). To state Wick's formula, we define the notion of a \emph{pairing}.

\begin{defn}
Given a totally ordered set $X$ of size $K$, we define a \emph{pairing} $\pi$ of $X$ as a permutation of $X$ with all cycles of size 2. (Note that $X$ has no pairings if $|X|$ is odd. If $K = 0$, then we consider the empty permutation a pairing.) We say that two elements $x,x^*$ of $X$ are \emph{paired} by $\pi$ if $\pi(x) = x^*$. We view $\pi$ as an element of the symmetric group $\frk S_K$ by identifying $X$ with $\{1,\ldots,K\}$ via the order-preserving map.
\end{defn}

\begin{lem}[Wick's formula]\label{lem-wicks}
Let $(Z_1,\ldots,Z_K)$ be a centered Gaussian vector. For each pairing $\pi \in \frk S_K$, let 
\begin{equation}
 \label{eqn-defn-paired-exp}
\left\langle Z_1, \ldots, Z_K \right\rangle_\pi := \prod_{\stackrel{i<j}{\pi(i) = j}} \left\langle Z_i Z_j \right\rangle
\end{equation}
be the product of the expectations $\left\langle Z_i Z_j \right\rangle$ for all $i,j$ paired by $\pi$. (If $K=0$ or $\pi$ is the empty pairing, we set~\eqref{eqn-defn-paired-exp} equal to $1$.)
Then the expectation of the product of random variables $Z_1 \cdots Z_K$ is the sum of the expectations~\eqref{eqn-defn-paired-exp} over all pairings $\pi \in \frk S_K$.
\end{lem}

By Wick's formula (and Brownian scaling), the expected trace of~\eqref{eqn-product-WW-outline} is a large sum of terms of the form
\begin{equation}\label{eqn-under-o}
\ep(m_1) \cdots \ep(m_K) \cdot
\sqrt{\frac{\abs{\lambda_{c(m_1)}}}{J}} \cdots\sqrt{\frac{\abs{\lambda_{c(m_K)}}}{J}} \cdot 
 \avg{W^{(c(m_1), j_1)}_{a_1a_2} \cdots W^{(c(m_K), j_K)}_{a_K a_{K+1}}}_\pi.
\end{equation}
for some indices $a_1,\ldots,a_{K+1} \in [N]$ with $a_1 = a_{K+1}$.

\medskip
\noindent
\textit{Step 4: Represent the Wilson loop expectation by a Poisson process.}
\medskip

\noindent
We defined
\begin{equation}
 \label{eqn-paired-expec-ov}
 \avg{W^{(c(m_1), j_1)}_{a_1a_2} \cdots W^{(c(m_K), j_K)}_{a_K a_{K+1}}}_\pi
\end{equation}
in~\eqref{eqn-under-o} as the product of expression of the form
\begin{equation}\label{eqn-each-pair}
 \left\langle W^{(c(m), j)}_{a,b} W^{(c(m^*), j^*)}_{a^*,b^*} \right\rangle
\end{equation}
for some $(m,j)$ and $(m^*,j^*)$ in $[6] \times [J]$ and indices $a,b,a^*,b^* \in [N]$. 

We observe that, if~\eqref{eqn-paired-expec-ov} is nonzero, then we must have $c(m) = c(m^*)$ and $j = j^*$. This means that, if $J$ is very large, we can view 
both indices $(m,j)$ and $(m^*,j^*)$ as corresponding roughly to a single point in the region $D_{(m,m^*)}$ bounded by $\lambda_{c(m)} \equiv \lambda_{c(m^*)}$. Since $J$ is very large, $[J]$ closely approximates the continuous interval $[0,1]$; so, roughly speaking, we can view $j = j^*$ as an element of $[0,1]$ that records the location of our point within the region $D_{(m,m^*)}$.

\begin{defn}\label{defn-space-match}
Let $\mcl C$ be the \emph{set of matching-color lasso pairs}
\[\mcl C = \{ (m,m^*) : \text{$m < m^*$ and $c(m) = c(m^*)$}\}.\]
We define the \emph{space of matching-color lasso pairs} as
\begin{equation}
\mcl D := \bigsqcup_{(m,m^*) \in \mcl C} D_{(m,m^*)},\label{eqn-space-mcp}
\end{equation}
where $D_{(m,m^*)}$ is isomorphic to the region bounded by $\lambda_{c(m)}$. 
We define a parametrizing bijection $\eta: \mcl C \times [0,1] \to \mcl D$ by defining $\eta$ on each cylinder set $(m,m^*) \times [0,1]$ as a space-filling curve in $D_{(m,m^*)}$. (We only use the curve $\eta$ to record the locations of points in $\mcl D$.)
\end{defn}

By representing every $(m,j)$ and $(m^*,j^*)$ paired by $\pi$ in~\eqref{eqn-paired-expec-ov} as a point in $\mcl D$, we can represent the expression
\begin{equation}
 \label{eqn-under-detail}
\underbrace{\ep(m_1) \cdots \ep(m_K)}_{\mytag{A}{termA1}} \cdot
\underbrace{\sqrt{\frac{\abs{\lambda_{c(m_1)}}}{J}} \cdots\sqrt{\frac{\abs{\lambda_{c(m_K)}}}{J}} }_{\mytag{B}{termB1}} \cdot 
\underbrace{\avg{W^{(c(m_1), j_1)}_{a_1a_2} \cdots W^{(c(m_K), j_K)}_{a_K a_{K+1}}}_\pi}_{\mytag{C}{termC1}},
\end{equation}
we derived in~\eqref{eqn-under-o} as a function of a collection of $K/2$ points in $\mcl D$. Therefore, we can view the \emph{sum} of the expressions~\eqref{eqn-under-detail} over all possible choices of indices as an expectation with respect to a \emph{random} collection of points $\Sigma$ in $\mcl D$, where~\eqref{termB1} is the probability measure of $\Sigma$. Roughly speaking, if we split each region comprising $\mcl D$ into $J$ equal-area blocks, then~\eqref{termB1} gives the probability that $|\Sigma| = K/2$ and that these $K/2$ points are contained in a given subset of $K/2$ blocks. Since the probability~\eqref{termB1} is proportional to the product of the areas of the blocks, we deduce that, in the $J \to \infty$ limit, the points $\Sigma$ are uniformly distributed from the set of ordered tuples in $\bigsqcup_{k = 0}^\infty \mcl D^k$. (Here, the space $\bigsqcup_{k = 0}^\infty \mcl D^k$ is endowed with the product measure induced from Lebesgue measure on $\mcl D$, and we define the ordering on this space in terms of $\eta$.) In other words, $\Sigma$ is a \emph{Poisson point process} on $\mcl D$. 

Finally, we show that we can interpret both the signs $\ep(\cdot)$ in~\eqref{termA1} and the pairing $\pi$ in~\eqref{termC1} as the following functions of $\Sigma$. 

\begin{defn}\label{def-epsilon-sigma}
Let $\Sigma$ be a collection of points in $\mcl D$. We define $\ep(\Sigma)$ as 
\[
 \prod_{((m,m^*), t) \in \eta^{-1}(\Sigma)} \ep(m) \ep(m^*).
\]
\end{defn}

\begin{defn}\label{defn-pi}
Let $\Sigma$ be a collection of points in $\mcl D$. We define $\pi(\Sigma) \in \frk S_{2|\Sigma|}$ as the pairing of the lexicographically ordered set
\begin{equation}
 \label{eqn-lex-a}
 \bigcup_{((m,m^*), t) \, \in \, \eta^{-1}(\Sigma)} \{(m,t),(m^*,t)\}
 \end{equation}
that pairs $(m,t)$ and $(m^*,t)$ for each $((m,m^*), t) \in \eta^{-1}(\Sigma)$.
\end{defn}

When we express the sum of terms~\eqref{eqn-under-detail} as an expectation with respect to the Poisson point process $\Sigma$, the term~\eqref{termA1} contributes to the constant~\eqref{eqn-constant-term}, the term~\eqref{termB1} corresponds to the sign $\ep(\Sigma)$, and the term~\eqref{termC1} corresponds to the term~\eqref{eqn-exp-single}. This completes the outline of the proof of Lemma~\ref{lem-main-poisson}.

\medskip
\noindent
\textit{Step 5: Represent~\eqref{termC1} as gluing edges to form a CW complex}
\medskip

The constant~\eqref{eqn-constant-term} is fairly explicit, as is the sign $\ep(\Sigma)$, but the sum of~\eqref{eqn-exp-single} over all possible indices $a_1,\ldots,a_{K+1}$ is more enigmatic for general Lie groups. However, it has a very nice geometric interpretation for certain choices of Lie groups. Here, we focus on the case $G = \UN$, which is treated in Lemma~\ref{lem-un}. Suppose we fix an instance of the point process $\Sigma$ and consider the sum of~\eqref{eqn-exp-single} over indices $a_1,\ldots,a_{K+1} \in [N]$ with $a_{K+1} = a_1$. This corresponds to summing a term of the form~\eqref{termC1} over all possible $a_1,\ldots,a_{K+1}$. 
As in the schematic diagram in Figure~\ref{fig-expansion-diagram}, we can express the $K$ matrices that appear in~\eqref{termC1} as $K/2$ matching pairs of edges of a cycle graph, with the pairing described by $\pi(\Sigma)$. We now remove all the other edges in the diagram, so we are left with a cycle graph with exactly $K$ edges, divided into $K/2$ pairs. We view the cycle graph as forming the boundary of a region $H$. We have drawn this graph in Figure~\ref{fig-expansion-diagram}, with the arrows representing the pairing of edges. The edges correspond, in cyclic order, to the terms $W_{a_pa_{p+1}}$ in~\eqref{termC1}.\footnote{We can ignore the superscript $(c(m_p),j_p)$, since we have already taken it into account when choosing the pairing $\pi(\Sigma)$. We explain this more precisely in the proof itself.} We can represent the indices $a_1,\ldots,a_K$ as labels of the \emph{vertices} of the graph. 
Now, in analyzing the sum of~\eqref{termC1} over all possible choices of indices $a_1,\ldots,a_{K+1}$, we observe that for a fixed choice of $a_1,\ldots,a_{K+1}$, the term~\eqref{termC1} is the product of $K/2$ covariances of the form~\eqref{eqn-each-pair}. Each of these $K/2$ covariances equals either $0$ or $-1/N$ (see~\eqref{eqn-casimir} below), so the product of covariances is equal to $(-N)^{-K/2}$ or zero. The key challenge, therefore, is to determine \emph{which} choices of $a_1,\ldots,a_{K+1}$ correspond to a nonzero value of~\eqref{termC1}.

The key idea is that, even though we have $K$ indices 
$a_1,\ldots,a_{K}$ (recall that $a_{K+1} = a_1$), we do not have $K$ degrees of freedom in choosing these indices---meaning that we cannot choose values for all $K$ indices independently if we want~\eqref{termC1} to be nonzero. If two edges of the graph are paired by $\pi(\Sigma)$, then in order for the corresponding covariance~\eqref{eqn-each-pair} to be nonzero, the vertex labels of these two edges must satisfy some relation. The specific relation that the labels must satisfy depends on the choice of Lie group. In the case $G = \UN$, we can interpret the relationship as follows: if we glue the two edges with the opposite orientation (where the edges are oriented in the same direction around the face), then the vertex labels of the glued edges must match. We can glue all $K/2$ pairs of edges this way to form a closed surface, which we view as a CW complex with a single face and $K/2$ edges. Having formed this CW complex, we can view the number $V$ of vertices of the CW complex as representing the numbers of \emph{degrees of freedom} we have in choosing the indices $a_1,\ldots,a_{K}$. To conclude, the sum of~\eqref{termC1} over all possible choices of indices $a_1,\ldots,a_{K+1}$ contains $N^V$ nonzero terms, with each nonzero term equal to $(-1/N)^{-K/2}$. Therefore, the sum equals $(-1)^{|\Sigma|} N^{-1+\chi}$, where $\chi$ is the Euler characteristic of the surface. This completes the outline of the proof of Lemma~\ref{lem-un}.

In Figure~\ref{fig-expansion-diagram}, we visualized a term~\eqref{eqn-product-WW-outline} in the expanded product of random walk approximations~\eqref{eqn-term-form-outline} as a collection of $K$ color-coded segments, which we circled with a dotted line. In this example, we have $K=10$. We remove all the segments not included in the subset to obtain a cycle graph with $10$ edges. We can express the expected trace of~\eqref{eqn-product-WW-outline} as a sum over pairings of these $10$ edges and labelings of the vertices of the graph by elements of $[N]$. For a term of the sum to be nonzero, both the pairing and the labeling must be ``compatible'' with the underlying data in a manner that we explain in the text. As we noted in the caption to Figure~\ref{fig-expansion-diagram}, the only compatible pairing in this example is the one that we have depicted with arrows.

\medskip
\noindent
\textit{Step 6: Interpret the CW complex in the previous step as a spanning surface}
\medskip

\begin{figure}[ht!] \centering
	\includegraphics[width=\textwidth]{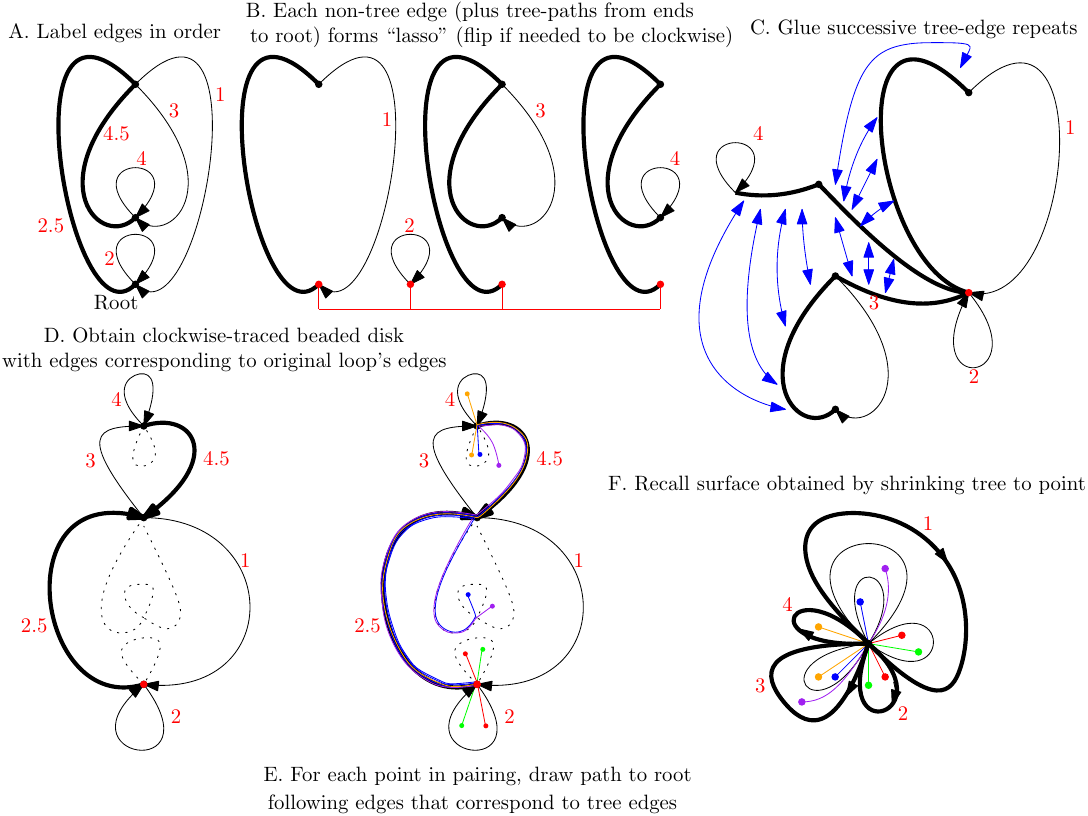}
	\caption{\textbf{Reference and spanning surfaces:} (A-D) illustrates a construction of the reference covering of the loop $\Gamma_0$. Beaded disk (D) is the same as Figure~\ref{fig-eg-surface-g0}\ (E) except the tree is \emph{not} shrunk to a point. (E) represents an example of surfaces spanning $\Gamma_0$, analogous to the previous procedure resulting in (F). The same construction with a more complicated loop is then shown in Figure~\ref{fig-petals-complicated}.
 }\label{fig-eg-petals2-g0}
\end{figure}

Up to this point, we expressed the Wilson loop expectation of a certain loop in terms of the sum of CW complexes constructed from the combinatorial structure of a lasso representation of the loop.
To establish a type of gauge-string duality story, we want to interpret these CW complexes further as topological spaces geometrically constructed from the loop, which span the loop as their boundaries. For this purpose, we define the \emph{reference covering} (Definition~\ref{defn-reference-covering}) and \emph{spanning surfaces} (Definition~\ref{defn-spanning}) of a loop. Spanning surfaces can be considered generalized ramified surfaces, which are different from the usual ramified surfaces unless the loop is special, e.g.\ in the case where the loop winds a simple loop several times, as studied in~\cite{Levy2008}. Figure~\ref{fig-eg-petals2-g0} illustrates the reference covering and an example of spanning surface for the loop $\Gamma_0$ in Figure~\ref{fig-eg-lasso-g0}. Roughly speaking, instead of shrinking the spanning tree to a point (as in Figure~\ref{fig-eg-surface-g0}), we keep the lassos with tree edges attached to obtain the reference covering. Analogous to the procedure described in Figure~\ref{fig-eg-surface-g0}, we can consider pairs of points that play the role of ``ramification points'' to construct a spanning surface. A more detailed description can be found in Section~\ref{sec-spanning}.

f\section{Proof for the Poisson sum of a single loop}\label{sec-poisson-sum}
In this section, we prove Lemma~\ref{lem-main-poisson}. 
We are given a loop $\Gamma$ in the plane with lasso representation
$\lambda_{c(1)}^{\ep(1)} \cdots \lambda_{c(M)}^{\ep(M)}$
for some collection of lassos ${(\lambda_\ell)}_{\ell \in [L]}$ and some coloring map $c:[M] \to [L]$ and $\ep:[M] \to \{-1,1\}$, as in Definition~\ref{defn-lasso-rep}.
By Proposition~\ref{prop-hol-bm}, the holonomy of each lasso $\lambda_\ell$ is a Brownian motion $U^{(\ell)} = U^{(\ell)}_t$ run for time $|\lambda_\ell|$, and the Brownian motions $U^{(1)},\ldots,U^{(L)}$ are independent. The holonomy of $\Gamma$ is
\begin{equation}\label{eqn-holonomy-general}
\left(U^{c(1)}_{\abs{\lambda_{c(1)}}}\right)^{\ep(1)} \cdots \left(U^{c(M)}_{\abs{\lambda_{c(M)}}}\right)^{\ep(M)}.
\end{equation}
The Wilson loop expectation of $\Gamma$ is the expected trace of~\eqref{eqn-holonomy-general}. Our method is to compute this expected trace by approximating the Brownian motions $U^{(i)}$ by random walks on the corresponding Lie algebra. We can expand this product of random walks as a sum of products of Gaussian matrices. We calculate the expected trace of this sum by applying Wick's formula. Finally, we interpret the result we obtain as an expectation with respect to a Poisson point process.

\subsection{Approximating Brownian motion on \texorpdfstring{$G$}{G} by a random walk}

To approximate Brownian motion on $G$ by a random walk, we first describe Brownian motion on $G$ as an It\^{o} SDE driven by a Brownian motion on the associated Lie algebra. (See, e.g.,~\cite[Chapter 2]{Liao2004} for a more detailed exposition of this formulation of Brownian motion on Lie groups.) We assume that $G$ is equipped with a bi-invariant metric that is defined in terms of an inner product $\inp{\cdot}{\cdot}_{\fg}$ on the associated Lie algebra $\fg$. We define Brownian motion on the Lie algebra $\fg$ as the process
\begin{equation} \label{eqn-bm-lie-algebra}
 W_t = \sum_{j=1}^{\dim \fg} B_t^{(j)} X_j,
\end{equation}
where $\{X_j\}_{j=1}^{\dim \fg}$ is an orthonormal basis of $\fg$ and $\{B_t^{(j)}\}_{j=1}^{\dim \fg}$ is a set of independent standard real Brownian motions. 
We can express Brownian motion on the Lie group $G$ started from $U_0 \in G$ in terms of $W_t$ by the Stratonovitch SDE
\[ dU_t = U_t \circ dW_t;\]
or, equivalently, by the 
It\^{o} SDE
\begin{equation}\label{bm-sde}
dU_t = U_t dW_t + \frac{\fc_\fg}{2} U_t dt.
\end{equation}
where the constant $\fc_\fg$ is defined by the expression
\begin{equation}\label{eq-constant}
 \fc_\fg I = \sum_{j=1}^{\dim \fg} X_j^2
\end{equation}
for $I$ the identity in $\frk g$.
We define \emph{standard} Brownian motion on $G$ as Brownian motion on $G$ started from the identity of $G$. Also, with $U_t$ defined by~\eqref{bm-sde}, we say that the Brownian motion $U_t$ is \emph{generated} by $W_t$.

For the classical Lie groups $G$ that we consider in this paper---namely, $\UN, \SON, \SUN$ and $\SphN$---we define the inner product on $\fg$ in terms of the Frobenius inner product \begin{equation}\label{eq:f-ip}
\inp{X}{Y}_{F} := \Tr(X^*Y)
\end{equation}
on $\BB{C}^{N \times N}$ as
\[
\inp{X}{Y}_{\fg} = \frac{\beta N}{2}\inp{X}{Y}_{F},
\]
where $\beta$ is 1 for $G=\SON$, 2 for $G=\UN$ or $G=\SUN$, and 4 for $G=\SphN$.
We note that $G=\ON$ has the same Lie algebra as $\SON$, but is not connected. The values of $\fc_\fg$ in these examples are computed in~\cite[Lemma 1.2]{Levy2011a} and summarized in Table~\ref{table-lie}. 

\begin{table}[htbp]
\centering
\begin{tabular}{llc}
\toprule
Lie group $G$ 
& Lie algebra $\frk g$ & $\fc_\fg$\\ 
\midrule
$\ON = \{O\in \GLNR: O^t O=I_N\}$ 
& $\oN = \{ X\in M_N(\mathbb R): X^t+X = 0 \}$ & \\
$\SON = \{O\in \ON: \det(O)=1\}$ 
& $\soN = \oN$ & $-1+\frac{1}{N}$\\
$\UN = \{U\in \GLNC: U^*U=I_N\}$ 
& $\uN = \{ X\in M_N(\mathbb C): X^*+X = 0 \}$ & $-1$ \\
$\SUN = \{U \in \UN : \det(U)=1\}$ 
& $\suN = \{ X\in \uN: \Tr(X) = 0 \}$ & $-1+\frac{1}{N^2}$\\
$\SphN = \{ S\in \mathrm{U}(N): S^t J S =J\}$ & $\sphN = \{ X\in \mathfrak{u}(N): X^t J+JX = 0 \}$ & $-1-\frac{1}{N}$\\
\bottomrule
\end{tabular}
\caption{A table of the definitions of the four classical Lie groups we consider in this paper: namely, the special orthogonal group $\SON$, the unitary group $\UN$, the special unitary group $\SUN$, and the unitary symplectic group $\SphN$ (for $N$ even). The table also describes the Lie algebras corresponding to these four Lie groups. In the table, $J$ denotes the matrix $\begin{pmatrix}
 0 & I_{N/2} \\
 -I_{N/2} & 0
\end{pmatrix}$.}\label{table-lie}
\end{table}

The SDE formulation~\eqref{bm-sde} of Brownian motion on a compact Lie group allows us to explicitly describe a random walk approximation of Brownian motion that will be integral to our proof of Lemma~\ref{lem-main-poisson}. Let $G$ be a compact Lie group. By the Peter-Weyl theorem, we may assume that $G$ is a closed subgroup of the general linear group $\GLNC$ for some $N$~\cite[Corollary IV.4.22]{Knapp2002}. We can approximate Brownian motion on $G$ by a random walk in $\BB{C}^{N \times N}$ using the \emph{Euler-Maruyama method}, which is the stochastic version of the forward Euler method in ODE theory. Specifically, for fixed $T>0$ and $J \in \BB{N}$, we define a random walk $\wt U_t = \wt U_t^J$ on the interval $[0,T]$ with step size $\Delta t:=T/J$ as follows. For $j=1,2,\dots, J$, we define $\widetilde{U}_{j\Delta t}$ recursively as
\begin{align}
\widetilde{U}_{j\Delta t} &:= \widetilde{U}_{(j-1)\Delta t} + \widetilde{U}_{(j-1)\Delta t}(W_{j\Delta t}-W_{(j-1)\Delta t}) + \frac{\fc_\fg}{2}\widetilde{U}_{(j-1)\Delta t}\Delta t\nonumber \\
&= \widetilde{U}_{(j-1)\Delta t}\left((1+\frac{\fc_\fg}{2}\Delta t)I_N+(W_{j\Delta t}-W_{(j-1)\Delta t})\right),\label{eq:walk-recursion}
\end{align}
where $W_t$ is 
as defined in~\eqref{eqn-bm-lie-algebra}. We then define $\widetilde{U}_t$ for other values of $t\in [0,T]$ by linear interpolation. The random walk $\widetilde{U}_t$ is measurable with respect to $W_t$, but is not contained in $G$ for finite $J$. The following result, whose proof is outlined in~\cite[Proposition 7]{Said2012}, asserts that this random walk converges to Brownian motion on $G$ locally uniformly in probability as $J \to \infty$ in the Frobenius norm (the norm defined by the inner product~\eqref{eq:f-ip} above).

\begin{prop}[Random-walk convergence in probability]\label{prop-rw-convergence}
Let $U_t$ be the standard Brownian motion on a connected
compact Lie group $G$. Then its random-walk approximation $\widetilde{U}_t$ defined by~\eqref{eq:walk-recursion} converges to $U_t$ locally uniformly in probability as $J\to\infty$, that is for any $T>0$ and fixed $\ep>0$ we have
\begin{equation}
 \adjustlimits\lim_{J\to\infty}\sup_{0\le t\le T}\Pr*{\nrm{\widetilde{U}_t-U_t}_F>\ep} = 0.
\end{equation}
\end{prop}

Proposition~\ref{prop-rw-convergence} implies that we can approximate the holonomy 	of $\Gamma$ in terms of Brownian motions on the Lie algebra.

\begin{defn}\label{defn-hol-approx}
Given a loop in the plane, we define its \emph{holonomy approximation} as the expression we obtain by replacing each $U^{(i)}_{\abs{\lambda_i}}$ and $\left({U^{(i)}_{\abs{\lambda_i}}}\right)^{-1}$ in~\eqref{eqn-holonomy-general} by
\[\wt{U}^{(i)}_{+}=\prod_{j=1}^J \left(\left(1+ \frac{\frk c_{\frk g} \abs{\lambda_i}}{2}\right)I_N + \frac{W_{\abs{\lambda_i}}^{(i, j)}}{\sqrt{J}}\right) \quad\text{and}\quad \wt{U}^{(i)}_{-}=\prod_{j=J}^1 \left(\left(1+\frac{\frk c_{\frk g} \abs{\lambda_i}}{2}\right)I_N - \frac{W_{\abs{\lambda_i}}^{(i, j)}}{\sqrt{J}}\right),\]
respectively.
\end{defn}

\begin{prop}\label{prop-conv-prob-approx}
Given a collection of loops in the plane, their holonomy approximations converge jointly in probability as $J \to \infty$ to their holonomies.
\end{prop}

\begin{proof}
This is stated in Proposition~\ref{prop-rw-convergence} for the case with a single lasso, and can be similarly derived for its inverse by applying the SDE for the inverse of Brownian motion on $G$. The proposition follows since convergence in probability respects
multiplications.
\end{proof}

Therefore, to compute the expected trace of the holonomy directly, we fix a value of $J$, and we consider the expected trace of the holonomy approximation. Later, we will combine Proposition~\ref{prop-conv-prob-approx} with additional estimate to show that the expected trace of the holonomy approximation converges as $J \to \infty$ to the expected trace of the holonomy.

The holonomy approximation of a loop is a large product $\prod_{m=1}^M \wt{U}^{(c(m))}_{\ep(m)}$ of $MJ$ terms
where $\wt{U}^{(i)}_{\pm}$ is a product of $J$ terms as defined in Definition~\ref{defn-hol-approx},
with each term of the form
\begin{equation}\label{eqn-term-form}
\underbrace{\left(1+\frac{\frk c_{\frk g} \abs{\lambda_i}}{2J}\right)I_N}_{\mytag{A}{termA}} \pm
\underbrace{\frac{W_{\abs{\lambda_i}}^{(i, j)}}{\sqrt{J}}}_{\mytag{B}{termB}}.
\end{equation}
We expand this large product of $MJ$ terms to obtain a sum of $2^{MJ}$ terms. 
When expanding the product, each term of the expansion corresponds to choosing either the matrix~\eqref{termA} or the matrix~\eqref{termB} for each term~\eqref{eqn-term-form} in the product, and then multiplying the chosen matrices together.

For example, if we choose the matrix~\eqref{termA} for every term~\eqref{eqn-term-form} in the holonomy approximation, and we multiply these matrices together, we get the identity matrix $I_N$ times the constant
\begin{equation}
 \label{eqn-zj}
Z_J := \prod_{m=1}^M \left(1+\frac{\frk c_{\frk g} \abs{\lambda_{c(m)}}}{2J}\right)^J,
\end{equation}
which converges as $J \to \infty$ to
\begin{equation}
 \label{eqn-z}
Z = \exp\left(\frac{\frk c_{\frk g}}{2} \sum_{m=1}^M \abs{\lambda_{c(m)}} \right).
\end{equation}
In general, each term of the sum is a scalar multiplied by a product of matrices of the form
\begin{equation}
 \label{eqn-product-W}
\prod_{k=1}^K \ep(m_k) \frac{W_{\abs{\lambda_{c(m_k)}}}^{(c(m_k), j_k)}}{\sqrt{J}}
\end{equation}
for some non-negative integer $K$ between $0$ and $MJ$ and some collection of indices $\{(m_k,j_k)\}_{k=1}^K$ from $[M] \times [J]$. The scalar is the product of all but $K$ of the terms in~\eqref{eqn-zj}. This means that the scalar is bounded from above by $Z_J$; moreover, if we fix $K$ and take the $J \to \infty$ limit, the scalar converges to $Z$.
In summary, we have the following.\begin{prop}\label{prop-hol-approx}
The holonomy approximation of $\Gamma$ is equal to the sum, over all non-negative integers $0 \leq K \leq MJ$ and all lexicographically ordered tuples $((m_k,j_k))_{k=1}^K$ of distinct elements of $[M] \times [J]$, of the product of matrices~\eqref{eqn-product-W}
times a scalar in $(0,Z_J]$ that converges to $Z$ as $J \to \infty$. (See~\eqref{eqn-zj} and~\eqref{eqn-z} for the definitions of $Z_J$ and $Z$.)
\end{prop}


\subsection{Simplifying by Wick's formula}\label{sec-step-3}

We now describe how we compute the expectation of the trace of~\eqref{eqn-product-W} for a \emph{fixed} choice of $K$ and lexicographically ordered tuple of indices $((m_k,j_k))_{k=1}^K$.
First, we derive an expression for the law of trace of~\eqref{eqn-product-W} in terms of the entries of the matrices $W_1^{(c(m_k), j_k)}$.

\begin{prop}\label{prop-trace-hol-approx}
The trace of~\eqref{eqn-product-W} is equal in law to the sum of
\begin{equation}\label{eqn-single-term}
\ep(m_1) \cdots \ep(m_K) \cdot \sqrt{\frac{\abs{\lambda_{c(m_1)}}}{J}} \cdots \sqrt{\frac{\abs{\lambda_{c(m_K)}}}{J}} \cdot
(W_{1}^{(c(m_1), j_1)})_{a_1a_2} \cdots (W_{1}^{(c(m_K), j_K)})_{a_K a_{K+1}}
\end{equation}
over all $a_1,\ldots,a_{K+1} \in [N]$ with $a_1 = a_{K+1}$.
\end{prop}

\begin{proof}
By
Brownian scaling, the joint law of the matrices $W_{\abs{\lambda_i}}^{(i, j)}/\sqrt{J}$ for any collection of values of $i$ and $j$, is equal to the joint law of the matrices $\sqrt{\abs{\lambda_i}/J}\, W_{1}^{(i, j)}$. Therefore, after taking the product of matrices~\eqref{eqn-product-W}, the $\ell$-th diagonal element of the resulting matrix is the sum of the terms~\eqref{eqn-single-term}
over all $a_1,\ldots,a_{k+1}$ between $1$ and $N$ with $a_1=a_{k+1} = \ell$. We obtain the trace of~\eqref{eqn-product-W} by summing over all $\ell$. 
\end{proof}

To compute the expectation of this trace, we apply \emph{Wick's formula} (Lemma~\ref{lem-wicks}).
The term~\eqref{eqn-single-term} contains a product of entries of the matrices $W_{1}^{(i, j)}$---these entries are jointly centered (complex) Gaussian random variables. Therefore, by Wick's formula, we can express the expectation of any product of Gaussian random variables in terms of expected products of \emph{pairs} of the variables.

\begin{prop}\label{prop-latest}
The expected trace of~\eqref{eqn-product-W} is given by the sum of the terms
\begin{equation}\label{eqn-omega-a}
\ep(m_1) \cdots \ep(m_K) \cdot
\sqrt{\frac{\abs{\lambda_{c(m_1)}}}{J}} \cdots\sqrt{\frac{\abs{\lambda_{c(m_K)}}}{J}} \cdot 
\avg{(W_{1}^{(c(m_1), j_1)})_{a_1a_2} \cdots (W_{1}^{(c(m_K), j_K)})_{a_K a_{K+1}}}_\pi
\end{equation}
over all pairings $\pi$ of the lexicographically ordered tuple $((m_k, j_k))_{k=1}^K$, and all $a_1,\ldots,a_{K+1} \in [N]$ with $a_1 = a_{K+1}$.
\end{prop}
\begin{proof}
The proposition follows from applying Wick's formula (Lemma~\ref{lem-wicks}) to the expression for the trace of~\eqref{eqn-product-W} in Proposition~\ref{prop-trace-hol-approx}.
\end{proof}

Combining Propositions~\ref{prop-hol-approx} and~\ref{prop-latest}, we have the following expression for the expected trace of the holonomy approximation.

\begin{prop}\label{prop-hol-approx-update}
The expected trace of the holonomy approximation of $\Gamma$ is equal to the sum, over all
\begin{itemize}
 \item non-negative integers $K$ between $0$ and $MJ$,
 \item lexicographically ordered tuples $((m_k,j_k))_{k=1}^K$ of distinct elements of $[M] \times [J]$,
 \item pairings $\pi \in \frk S_K$, and
 \item indices $a_1,\ldots,a_{K+1} \in [N]$ with $a_1 = a_{K+1}$;
\end{itemize}
of~\eqref{eqn-omega-a}
times a scalar in $(0,Z_J]$ that converges to $Z$ as $J \to \infty$. (We defined $Z_J$ and $Z$ in~\eqref{eqn-zj} and~\eqref{eqn-z}, and $I_N$ is the $N \times N$ identity matrix.)
\end{prop}


\subsection{Analyzing the \texorpdfstring{$J \to \infty$}{J→∞} limit}

Now that we have derived an expression for the expected trace of the holonomy approximation, we turn to the task of proving that this expected trace converges to the Wilson loop expectation in the $J \to \infty$ limit. This is the most technically involved step in the proof of Lemma~\ref{lem-main-poisson}. Roughly speaking, we will need to show that the expected trace of the holonomy approximation is uniformly integrable in $J$, and we will prove this assertion by bounding its $L^2$ norm.

To prepare for this argument, we first
observe that, while we are summing over a very large number of terms in Proposition~\ref{prop-hol-approx-update}, most of the terms of the sum are zero. If we examine the summand~\eqref{eqn-omega-a}, we observe that the summand is nonzero \emph{only if} the tuple of indices $((m_k,j_k))_{k=1}^K$ and the pairing $\pi$ are ``compatible'' in the sense that we now define.

\begin{defn}
Let $((m_k,j_k))_{k=1}^K$ be a lexicographically ordered tuple of distinct elements of $[M] \times [J]$, and let $\pi \in \frk S_K$ be a pairing. We say that $((m_k,j_k))_{k=1}^K$ and $\pi$ are \emph{compatible} if each pair of indices $(m,j)$, $(m^*,j^*)$ paired by $\pi$ satisfies $c(m) = c(m^*)$ and $j=j^*$.
\end{defn}

To bound the number of compatible tuples and pairings, we express the set of compatible tuples and pairings in the following alternative form.
 
 \begin{prop}\label{prop-matching-bijection}
Let $\mcl C$ be the \emph{set of matching-color lasso pairs}
\begin{equation*}\mcl C = \{ (m,m^*) : \text{$m < m^*$ and $c(m) = c(m^*)$}\}\end{equation*}
that we defined in Definition~\ref{defn-space-match}.
Then, for every even integer $K$ between $0$ and $MJ$, we have a bijection between compatible tuples $((m_k,j_k))_{k=1}^K$ and pairings $\pi$ on the one hand, and collections of $K/2$ elements of $\mcl C \times [J]$ on the other.
\end{prop}

\begin{proof}
Every compatible $((m_k,j_k))_{k=1}^K$ and $\pi$ determines a collection of $K/2$ elements of $\mcl C \times [J]$. Conversely, any collection of $K/2$ elements of $\mcl C \times [J]$ determines the tuple $((m_k,j_k))_{k=1}^K$. Though the collection of elements of $\mcl C \times [J]$ is unordered, we stipulated that $((m_k,j_k))_{k=1}^K$ is lexicographically ordered; this ordering determines the pairing $\pi$ as well. By definition, the tuple and pairing are compatible. It is easy to check that this is a bijection.
\end{proof}

We are now ready to prove that the expected trace of the holonomy approximation converges to the Wilson loop expectation.

\begin{prop}\label{prop-limit-legal}
The expected trace of the holonomy approximation converges as $J \to \infty$ to the expected trace of the holonomy.
\end{prop}

We note that the proof of this proposition is the most tedious of the arguments in the proof of Lemma~\ref{lem-main-poisson}, and the reader may skip it and return to it later without missing any of the core ideas of our approach.

\begin{proof}
We already know that the trace of the holonomy approximation converges in probability to the trace of the holonomy (Proposition~\ref{prop-conv-prob-approx}). Therefore, to show that the trace of the holonomy approximation converges in mean as $J \to \infty$, it suffices to show that the trace of the holonomy approximation is bounded in $L^2$ uniformly in $J$. 

By propositions~\ref{prop-hol-approx} and~\ref{prop-trace-hol-approx}, the law of the trace of the holonomy approximation is given by the sum of 
a scalar in $(0,Z_J]$ times~\eqref{eqn-single-term} over all $K \in [MJ]$, indices $a_1,\ldots,a_{K+1} \in [N]$ with $a_1 = a_{K+1}$, and lexicographically ordered tuples $((m_k,j_k))_{k=1}^K$ of distinct elements of $[M] \times [J]$. Since $Z_J \to Z$ as $J \to \infty$, we can bound each of the scalars and each of the terms $|\lambda_{c(\cdot)}|$ by a constant $C>0$ that does not depend on $J$. Therefore, if we square the expression for the law of the trace and apply Wick's formula, we find that the expectation of the squared trace is bounded from above by the sum of
\begin{equation}
 \label{eqn-sq-t}
C^{K+1} J^{-K}
\left| \left\langle
(W_{1}^{(c(m_1), j_1)})_{a_1a_2} \cdots (W_{1}^{(c(m_K), j_K)})_{a_K a_{K+1}}
(W_{1}^{(c(m_1^*), j_1^*)})_{a_1a_2} \cdots (W_{1}^{(c(m_K^*), j_K^*)})_{a_K a_{K+1}}
\right\rangle_\pi
\right|
\end{equation}
over all indices $a_1,\ldots,a_{K+1},a_1^*,\ldots,a_{K+1}^*$ in $[N]$ with $a_1 = a_{K+1}$ and $a_1^* = a_{K+1}^*$, all pairings $\pi \in \frk S_{2K}$, and all lexicographically ordered tuples $((m_k,j_k))_{k=1}^K$ and $((m_k^*,j_k^*))_{k=1}^K$ of distinct elements of $[M] \times [J]$. By replacing each $m_k^*$ by $m_k^*+K$, we can represent the pair of tuples $((m_k,j_k))_{k=1}^K$ and $((m_k^*,j_k^*))_{k=1}^K$ by a single lexicographically ordered tuple $((m_k,j_k))_{k=1}^{2K}$ of distinct elements of $[2M] \times [J]$. The corresponding term~\eqref{eqn-sq-t} is nonzero only if $((m_k,j_k))_{k=1}^{2K}$ and $\pi$ are compatible. 
By proposition~\ref{prop-matching-bijection}, the number of compatible tuples $((m_k,j_k))_{k=1}^{2K}$ and pairings $\pi \in \frk S_{2K}$ is at most ${(2M)^{2} \binom{J}{K}}$. Therefore, the number of nonzero terms~\eqref{eqn-sq-t} in our expression for the expectation of the squared trace is at most \[N^{2K} {(2M)^{2} \binom{J}{K}}
\leq N^{2K} (2M)^{2K} J^K / K!.\] Moreover, after making $C$ larger if necessarily, we can bound each term~\eqref{eqn-sq-t} by $C^{2K+1} J^{-K}$. We conclude that the expectation of the squared trace is at most
\[
\sum_{K=0}^{MJ} N^{2K} \frac{(2M)^{2K} J^K}{K!} C^{2K+1} J^{-K} = \sum_{K=0}^{MJ} C \frac{(2CMN)^{2K}}{K!},\]
which is bounded uniformly in $J$ by $C e^{2CMN}$.
\end{proof}

By restricting the sum in proposition~\ref{prop-hol-approx-update} to compatible $((m_k,j_k))_{k=1}^K$ and $\pi$, we can also simplify the expression for the summand. We observe that, if $((m_k,j_k))_{k=1}^K$ and $\pi$ are compatible, we have
\[
\avg{(W_{1}^{(c(m_1), j_1)})_{a_1a_2} \cdots (W_{1}^{(c(m_K), j_K)})_{a_K a_{K+1}}}_\pi = 
\avg{W_{a_1a_2} \cdots W_{a_K a_{K+1}}}_\pi.
\]
This observation, plus an additional simplification, allows us to express the expected trace of the holonomy as the following $J \to \infty$ limit.

\begin{prop}\label{prop-new-wilson}
Let $W$ be the Brownian motion on $\frk g$ at time $1$.
We can express the expected trace of the holonomy as the constant $Z$ (defined in~\eqref{eqn-z}) times the $J \to \infty$ limit of the sum of
\begin{equation}\label{eqn-under}
\underbrace{\ep(m_1) \cdots \ep(m_K)}_{\mytag{A}{term1}} \cdot
\underbrace{\sqrt{\frac{\abs{\lambda_{c(m_1)}}}{J}} \cdots\sqrt{\frac{\abs{\lambda_{c(m_K)}}}{J}} }_{\mytag{B}{term2}} \cdot 
\underbrace{\avg{W_{a_1a_2} \cdots W_{a_K a_{K+1}}}_\pi}_{\mytag{C}{term3}}
\end{equation}
over the set of compatible $((m_k,j_k))_{k=1}^K$ and $\pi$ and indices $a_1,\ldots,a_{K+1} \in [N]$ with $a_1 = a_{K+1}$.
\end{prop}

\begin{proof}
By proposition~\ref{prop-hol-approx-update}, we can express the Wilson loop expectation as the $J \to \infty$ limit of the sum, over the set of tuples $((m_k,j_k))_{k=1}^K$, pairings $\pi$, and $a_1,\ldots,a_{K+1} \in [N]$ with $a_1 = a_{K+1}$, of~\eqref{eqn-omega-a} times a constant in $(0,Z_J]$ that tends to $Z$ as $J \to \infty$.
As we explained above, we can restrict the sum to \emph{compatible} tuples $((m_k,j_k))_{k=1}^K$ and pairings $\pi$. Moreover, if $((m_k,j_k))_{k=1}^K$ and $\pi$ are compatible, we have
\[
\avg{(W_{1}^{(c(m_1), j_1)})_{a_1a_2} \cdots (W_{1}^{(c(m_K), j_K)})_{a_K a_{K+1}}}_\pi = 
\avg{W_{a_1a_2} \cdots W_{a_K a_{K+1}}}_\pi.
\]
To complete the proof of the proposition, it suffices to show that we can take the limit inside the summation, so that we can extract the multiplicative scalar in each term and replace it by its limit $Z$. We will justify take the limit inside the summation by applying the dominated convergence theorem. Each term in the summation is bounded by $C^{K/2} J^{-K/2}$. Moreover, by Proposition~\ref{prop-matching-bijection}, the number of such terms at most ${M^2 \binom{J}{K/2}} \leq M^{K} J^{K/2}/(K/2)!$. Since the sum of $(C^{K/2} J^{-K/2})(M^{K} J^{K/2}/(K/2)!)$ over all even non-negative integers $K$ is finite, the dominated convergence theorem applies.
\end{proof}

\subsection{Representing the Wilson loop expectation via a Poisson point process}

To obtain a probabilistic expression for the $J \to \infty$ limit in Proposition~\ref{prop-new-wilson}, we 
first describe a compatible tuple and pairing as a function of a collection of points in the space $\mcl D$ that we defined in Definition~\ref{defn-space-match}. This formulation has the advantage that the definition of $\mcl D$ does not depend on $J$.

Suppose that $\Sigma$ is a collection of points in $\mcl D$. Each point $x \in \Sigma$ lies in $\eta((m,m^*) \times [\frac{j-1}{J},\frac{j}{J}])$ for some $(m,m^*) \times j$ in $\mcl C \times [J]$. 
Let $\mcl E = \mcl E(\Sigma,J)$ be the event that the elements $(m,m^*) \times j$ of $\mcl C \times [J]$ are distinct for different $x \in \Sigma$. Note that, for a fixed collection $\Sigma$, the event $\mcl E(\Sigma,J)$ holds for all large enough $J$. On this event, the elements of $\Sigma$ determine $|\Sigma|$ distinct elements of $\mcl C \times [J]$.
By applying the bijection in the proof of Proposition~\ref{prop-matching-bijection}, we can express these elements of $\mcl C \times [J]$ as a compatible tuple $((m_k,j_k))_{k=1}^{2|\Sigma|}$ and pairing $\pi$. 

\begin{defn}
We call the tuple $((m_k,j_k))_{k=1}^{2|\Sigma|}$ and pairing $\pi$ just defined the compatible tuple and pairing \emph{associated} to $\Sigma$. (If $\mcl E$ does not hold, then $\Sigma$ does not have an associated tuple and pairing.)
\end{defn}

We have defined both the tuple and pairing associated to $\Sigma$ for a fixed value of $J$. We now show that, on $\mcl E$, the pairing $\pi$ associated to $\Sigma$ does not depend on $J$. 
\begin{prop}\label{prop-pi}
Let $\Sigma$ be a collection of points in $\mcl D$. Then, for each $J$ for which $\mcl E$ holds, the pairing associated to $\Sigma$ is equal to the pairing $\pi(\Sigma)$ defined in Definition~\ref{defn-pi}.
\end{prop}

\begin{proof}
For $t \in [0,1]$, let $t_J = \lfloor J t \rfloor + 1$. Then the tuple associated to $\Sigma$ is the lexicographical ordering of the set
\begin{equation}
 \label{eqn-lex-b}
 \bigcup_{((m,m^*), t) \, \in \, \eta^{-1}(\Sigma)} \{(m,t_J),(m^*,t_J)\}
 \end{equation}
This means that the pairing associated to $\Sigma$ is the pairing of the lexicographically ordered set~\eqref{eqn-lex-b}
 that pairs $(m,t_J)$ and $(m^*,t_J)$ for each $((m,m^*), t) \in \eta^{-1}(\Sigma)$.
Since $s < t$ implies that $s_J \leq t_J$, the positions of $(m,t_J)$ and $(m^*,t_J)$ in the lexicographical ordering of~\eqref{eqn-lex-b} is the same as the positions of $(m,t)$ and $(m^*,t)$ in the lexicographical ordering of~\eqref{eqn-lex-a}. This means that the pairing associated to $\Sigma$ is the pairing that pairs each $(m,t)$ and $(m^*,t)$ in~\eqref{eqn-lex-a}.
\end{proof}

By Proposition~\ref{prop-pi}, on $\mcl E$, the pairing associated to $\Sigma$ does not depend on $J$. This means that, if we express the term~\eqref{term3} in~\eqref{eqn-under} as a function of $J$ and $\Sigma$ by taking the associate compatible tuple and pairing, then the term~\eqref{term3} does not depend on $J$. The same is clearly true for the term~\eqref{term1}.

\begin{prop}
Let $\Sigma$ be a collection of points in $\mcl D$. Then, if $\mcl E$ holds, the term~\eqref{term1} for the associated compatible tuple and pairing is given by the sign $\ep(\Sigma)$ defined in Definition~\ref{def-epsilon-sigma}.
\end{prop}

\begin{proof}
The proposition follows immediately from the definition of $\ep(\Sigma)$ in Definition~\ref{def-epsilon-sigma}.
\end{proof}

Now, suppose that $\Sigma$ is a Poisson point process on $\mcl D$ with intensity given by the Lebesgue measure. Then the probability that $\Sigma$ is associated to a particular compatible tuple $((m_k,j_k))_{k=1}^K$ and pairing $\pi$ is given by~\eqref{term2}. Hence, we can interpret the sum in Proposition~\ref{prop-new-wilson} as the expected value of some function of $\Sigma$.

\begin{prop}\label{prop-final-J}
Let $\Sigma$ be a Poisson point process on $\mcl D$ with intensity given by the Lebesgue measure. We can express the
sum in Proposition~\ref{prop-new-wilson} as the constant $Z$ defined in~\eqref{eqn-z} times the expected value of the sum of
\begin{equation}\label{eqn-before-limit}
 \exp\left(\sum_{(m,m^*) \in \mcl C} |\lambda_{c(m)}| \right)
 \mathbf{1}_{\mcl E(\Sigma,J)} \ep(\Sigma) \avg{W_{a_1a_2} \cdots W_{a_{2|\Sigma|} a_{2|\Sigma|+1}}}_{\pi(\Sigma)} 
\end{equation}
over all $a_1,\ldots,a_{2|\Sigma|+1}$ in $[N]$ with $a_1 = a_{2|\Sigma|+1}$.
\end{prop}

\begin{proof}[Proof of Lemma~\ref{lem-main-poisson}]
Since $\BB{P}(\mcl E(\Sigma,J)) \to 1$ as $J \to \infty$, the theorem follows from taking the $J \to \infty$ limit of the sum in Proposition~\ref{prop-final-J} and applying the dominated convergence theorem.
\end{proof}

\section{Proofs for the surface sum of multiple loops}\label{sec-surface-story}
In this section, we prove Lemmas~\ref{lem-un} and~\ref{lem-other} as well as their multiple-loop versions.  We first recall that, in Lemma~\ref{lem-main-poisson}, we considered an arbitrary compact connected Lie group $G$, and we asserted that the Wilson loop expectation of a collection of loops with gauge group $G$ can be written as an expectation with respect to a Poisson point process $\Sigma$ on the \emph{space of matching-color lasso pairs}.  We first defined a sign $\ep(\Sigma)$  and a pairing $\pi(\Sigma)$  as functions of $\Sigma$.  (See Definitions~\ref{def-epsilon-sigma} and~\ref{defn-pi}.) We then showed that the Wilson loop expectation is given by a constant~\eqref{eqn-constant-term} times the expected value of $\ep(\Sigma)$ times the sum
\begin{equation}
    \label{eqn-exp-single-2}
 \avg{W_{a_1a_2}, \ldots, W_{a_{2|\Sigma|} a_{2|\Sigma|+1}}}_{\pi(\Sigma)}
\end{equation}
over all $a_1,\ldots,a_{2|\Sigma|+1} \in [N]$ with $a_{2|\Sigma|+1} = a_1$. We rewrite the product of expectations~\eqref{eqn-exp-single-2} as follows.    

\begin{prop}
Let $K = |\Sigma|$.\footnote{We note that this differs from the definition of $K$ in the previous section by a factor of $2$.}
If we write $\pi(\Sigma)$ canonically as
\[
\pi(\Sigma) = (p_1 \; q_1) \cdots (p_{K} \; q_{K}), \qquad p_j < q_j, \quad p_1 < \cdots < p_{K},
\]
then we can express~\eqref{eqn-exp-single-2} as 
\begin{equation} \label{eq-weight-term}
\prod_{k=1}^{K} \avg{W_{a_{p_k}a_{s(p_k)}} W_{a_{q_k}a_{s(q_k)}}},
\end{equation}
where $s$ is the cyclic permutation $(1 \cdots 2K)$.
\end{prop}
\begin{proof}
    This follows from the definition of $\avg{\cdot}_{\pi}$ for a pairing $\pi$.
\end{proof}

Now, for fixed $p,q \in [2K]$, the covariance 
\begin{equation}\label{eq-covariance}
	\avg{W_{a_p a_{s(p)}} W_{a_q a_{s(q)}}}
\end{equation}
is nonzero if and only if the four values $a_p, a_{s(p)}, a_q, a_{s(q)}$ satisfy some relation.  The relation that these four values must satisfy depends on the covariance structure of $W$, which in turn depends on the particular Lie algebra $\frk g$. For the four Lie algebras that we are considering---namely, $\uN$, $\soN$, $\suN$, and $\SphN$---there are four possible relations to consider.  We will label each these relations, so that we can easily refer to them throughout the section.

\begin{defn}\label{defn-relation}
We define four types of \emph{relations} between $a_p, a_{s(p)}, a_q, a_{s(q)}$:
\begin{enumerate}[I.]
    \item $a_p = a_{s(q)}$ and $a_{s(p)} = a_q$;
    \item  $a_p = a_{q}$ and $a_{s(p)} = a_{s(q)}$;
    \item $a_p = a_{s(p)}$ and $a_{q} = a_{s(q)}$;
    \item  $a_p\equiv a_{q}+N\pmod {2N}$ and $a_{s(p)}\equiv a_{s(q)}+N\pmod {2N}$.
\end{enumerate}
We note that the relations of type \rom{1}, \rom{2}, and \rom{3} are satisfied simultaneously when $a_p=a_{s(p)} = a_q = a_{s(q)}$, and the relations of type \rom{1} and \rom{4} are satisfied simultaneously when $a_p=a_{s(p)}+N = a_q+N = a_{s(q)} \pmod {2N}$.
\end{defn}

\begin{prop}\label{prop-compatible-one}
    The covariance~\eqref{eq-covariance} is nonzero if and only if $a_p, a_{s(p)}, a_q, a_{s(q)}$ satisfy a relation of type $n$ for some $n \in \mcl R^G$, where 
    \begin{equation}
    \label{eqn-defn-R}
    \mathcal R^{\UN} = \{\romnum{1}\}, \qquad \mathcal R^{\SON} = \{\romnum{1}, \romnum{2}\}, \qquad \mathcal R^{\SUN} = \{\romnum{1},\romnum{3}\}, \qquad \mathcal R^{\SphN} = \{\romnum{1},\romnum{4}\},
\end{equation}
Consequently,~\eqref{eq-weight-term} is nonzero if and only if, for each $k$, the indices $a_{p_k}, a_{s(p_k)}, a_{q_k}, a_{s(q_k)}$ satisfy a relation of type $n$ for some $n \in \mcl R^G$.
\end{prop}

\begin{proof}
    The proposition follows from the covariance structure of $W$, which is given, e.g., in~\cite[Lemma 1.1]{Levy2011a} or~\cite{Dahlqvist2017} as
            \begin{equation}
        \label{eqn-casimir} \avg{W_{ab} W_{cd}} = 
    \begin{cases}
    \left(-\frac{1}{N}\right){\1_{a=d} \1_{b=c}} & \text{if $G=\UN$} \\
		\left(-\frac{1}{N}\right){\1_{a=d} \1_{b=c}}
		+\left(\frac{1}{N}\right){\1_{a=c} \1_{b=d}}, & \text{if $G=\SON$} \\
		\left(-\frac{1}{N}\right){\1_{a=d} \1_{b=c}}+
		\left(\frac{1}{N^2}\right){\1_{a=b} \1_{c=d}}, & \text{if $G=\SUN$} \\
		\left(-\frac{1}{N}\right){\1_{a=d} \1_{b=c}}
		+\eta(a)\eta(b)\left(\frac{1}{N}\right){\1_{a \equiv c + N/2} \1_{b \equiv d + N/2}}, & \text{if $G=\SphN$} 
    \end{cases}
    \end{equation}
    where, in the fourth expression, the equivalence $\equiv$  is modulo $N$, and $\eta(i) = \1_{\{i \le N/2\}} - \1_{\{i > N/2\}}$.
   \end{proof}

By Proposition~\ref{prop-compatible-one}, the indices $a_1,\ldots,a_{2K}$ must satisfy some set of relations in order for~\eqref{eq-covariance} to be nonzero.  We can encode these relations geometrically as follows:   we view the indices $a_1,\ldots,a_{2K}$  as labelings of the vertices of a surface, with edges joining the vertices labeled $a_p$ and $a_{s(p)}$ for each $p$; and we view the relations as  ways of gluing or pinching edges. 

We now describe this geometric construction in detail.  We  define a $2$-dimensional CW complex given by an oriented surface $H$  whose boundary is a directed cycle graph with $2|\Sigma| = 2K$ edges.  (If $K = 0$, then $H$ has the sphere topology; otherwise, $H$ has the disk topology.)
We label the vertices of the boundaries $\partial H$ by $v_1,\ldots,v_{2K}$ in cyclic  order around the boundary. By definition of $s$, the edges of the boundaries  $\partial H^{(i)}$ are given by $\overrightarrow{v_p v_{s(p)}}$ for $p \in [2K]$. If we label each vertex $v_j$ by $a_j$, then we can associate each type of relation  between the $4$-tuple $a_p,a_{s(p)},a_q,a_{s(q)}$ to a way to \emph{glue} or \emph{contract} the pair of edges $\overrightarrow{v_p v_{s(p)}}$ and $\overrightarrow{v_q v_{s(q)}}$.  We choose the gluing/contraction operation so that the relation compares  the labels of the two pairs of vertices identified by the operation.
\begin{itemize}
\item
We associate relation \rom{1}  to gluing the pair of edges $\overrightarrow{v_{p} v_{s(p)}}$ and $\overrightarrow{v_{q} v_{s(q)}}$ so that  the vertices  $v_p,v_{q}$ are identified and the vertices  $v_{s(p)},v_{s(q)}$ are identified.
\item
We associate relations \rom{2} and \rom{4}  to gluing the pair of edges $\overrightarrow{v_{p} v_{s(p)}}$ and $\overrightarrow{v_{q} v_{s(q)}}$ so that the vertices  $v_p,v_{s(q)}$ are identified and the vertices  $v_q,v_{s(p)}$ are identified.
\item
We associate relation \rom{3} to contracting the pair of edges $\overrightarrow{v_p v_{s(p)}}$ and $\overrightarrow{v_q v_{s(q)}}$. This operation identifies the vertices  $v_p,v_{s(p)}$ and the vertices $v_q,v_{s(q)}$.
\end{itemize}

Given a sequence $\mathbf{r} = (r_k)_{k \in [K]}$ in $\{\romnum{1}, \romnum{2}, \romnum{3}, \romnum{4}\}^{K}$, we construct a closed surface by gluing/contracting each pair of edges $\overrightarrow{v_{p_k} v_{s(p_k)}}$ and $\overrightarrow{v_{q_k} v_{s(q_k)}}$ according to the relation $r_k$:

\begin{defn}[The surface $H(\pi, \bm r)$]\label{defn-H}
Let $\mathbf{r} = (r_k)_{k \in [K]}$ in $\{\romnum{1}, \romnum{2}, \romnum{3}, \romnum{4}\}^{K}$, and let $\sim$ be the equivalence relation that describes  applying to each pair of edges $\overrightarrow{v_{p_k} v_{s(p_k)}}$ and $\overrightarrow{v_{q_k} v_{s(q_k)}}$ on $\bigsqcup_{i=1}^n \partial H^{(i)}$ the gluing/contraction operation associated to the relation $r_k$. We define $H(\pi,\mathbf{r})$ as the CW complex given by the disjoint union of closed surfaces $\left( \bigsqcup_{i=1}^n H^{(i)}\right)/\sim$, with edges and vertices given by the embedded graph without crossings $\left(\bigsqcup_{i=1}^n \partial H^{(i)}\right)/\sim$.
\end{defn}

The following lemma asserts that we can view the set of labels $a_1,\ldots,a_{2K}$ such that each $a_{p_k}, a_{s(p_k)}, a_{q_k}, a_{s(q_k)}$ satisfies the relation $r_k$ as a labeling of the vertices of $H(\pi,\mathbf{r})$.

\begin{prop}\label{lem-label-bijection}
Given a sequence $\mathbf{r} = (r_k)_{k \in [K]}$ in $\{\romnum{1}, \romnum{2}, \romnum{3}, \romnum{4}\}^{K}$, there exists a bijection between the following two sets.
\begin{itemize}
\item The set $\mcl L_1(\pi,\mathbf{r})$ of labelings $a_1,\ldots,a_{2K}$ of the vertices of $\bigsqcup_{i=1}^n \partial H^{(i)}$ (with $v_j$ labeled by $a_j:= a(j)$) such that, for each $k$, the labels $a_{p_k}, a_{s(p_k)}, a_{q_k}, a_{s(q_k)}$ satisfy relation $r_k$.  
    \item The set $\mcl L_2(\pi,\mathbf{r})$ of labelings of the vertices of  $H(\pi,\mathbf{r})$ by the elements of $[N]$.
\end{itemize}
\end{prop}

\begin{proof}
We define a bi-directed graph $\BB{G}(\pi,\mathbf{r})$ on the vertices $\{v_1,\ldots,v_{2K}\}$ as the graph with the following two bi-directed edges for each $k \in [K]$.
\begin{itemize}
    \item We join by edges the two pairs of vertices in $\{v_{p_k},v_{s(p_k)},v_{q_k}v_{s(q_k)}\}$ identified by the gluing/contraction operation that corresponds to the relation $r_k$.
    \item We assign each of these two edges of this graph an orientation at both ends, such that the edges point outward toward the vertices $v_{p_k}, v_{q_k}$ and inward away from the vertices $v_{s(p_k)}, v_{s(q_k)}$.
\end{itemize}

Following standard terminology, we call an edge of the graph \emph{directed} if it is oriented outward toward one of its vertices and inward away from the other, and \emph{introverted} (resp. \emph{extroverted}) if it is oriented inward away from (resp.\ outward toward) both of its vertices.

  If $G = \SON, \UN,$ or $\SUN$, then a labeling in $\mcl L_1(\pi,\mathbf{r})$ is a labeling of each vertex in $\bigsqcup_{i=1}^n \partial H^{(i)}$ by an element of $[N]$  such that all the vertices in each connected component of $\BB{G}(\pi,\mathbf{r})$ have the same label.  This corresponds to a labeling of each connected component of $\BB{G}(\pi,\mathbf{r})$; or, equivalently, to a labeling of the vertices of $H(\pi,\mathbf{r})$ by $[N]$.
  
  If $G = \SphN$, then a labeling in $\mcl L_1(\pi,\mathbf{r})$ is a labeling of each vertex in $\bigsqcup_{i=1}^n \partial H^{(i)}$ by an element of $[N]$  such that any two vertices connected by a directed edge have the same label, and such that the labels of any two vertices connected by an introverted or extroverted edge differ by $N/2$.  Since each connected component of $\BB{G}(\pi,\mathbf{r})$ has an even number of introverted or extroverted edges, the set of such labelings is in bijection with the set of labelings of the connected components of $\BB{G}(\pi,\mathbf{r})$---or, equivalently, with the set of labelings of the vertices of $H(\pi,\mathbf{r})$ by $[N]$.
 \end{proof}

We are now ready to prove Lemmas~\ref{lem-un} and~\ref{lem-other}.  We state these two lemmas together, in slightly different terminology, as the following proposition.

\begin{prop}\label{prop-both-lemmas}
	Let $G$ be one of the classical Lie groups $\UN$, $\SON$, $\SUN$, $\SphN$. (For the last case, we assume $N$ is even.)  
	For $\mathbf{r} \in (\cR^G)^K$, we let
	$H(\pi,\bm r)$ be the CW complex defined in Definition~\ref{defn-H}.  Then~\eqref{eq-weight-term} is equal to the sum, over all $\mathbf{r} \in (\cR^G)^K$, of
	  \begin{equation}
        \label{eqn-omega-term-1}
    \begin{cases}
    (-1)^{-\#_{\romnum{1}}(\mathbf{r})} N^{-1+\chi}, & \text{if $G=\UN$} \\
	(-1)^{-\#_{\romnum{1}}(\mathbf{r})} N^{-1+\chi} , & \text{if $G=\SON$} \\
	 (-1)^{-\#_{\romnum{1}}(\mathbf{r})} N^{-2\#_{\romnum{3}}(\mathbf{r})} N^{-1+\chi}, & \text{if $G=\SUN$} \\
	 (-1)^{-\#_{\romnum{1}}(\mathbf{r})} (-1)^{\mu} N^{-1+\chi}, & \text{if $G=\SphN$} 
    \end{cases}
    \end{equation}
    where 	$\#_{n}(\mathbf{r})$ is the number of copies of $n$ in $\mathbf{r}$, and $\chi = \chi(H(\pi,\bm r))$ and $\mu = \mu(H(\pi,\bm r))$ are the Euler characteristic and non-orientable genus, respectively, of  $H(\pi,\bm r)$.
\end{prop}

To prove the proposition, we need the following result to handle the $\eta$ terms that appear in the covariance~\eqref{eqn-casimir} of $W$ in the case $G = \SphN$.

\begin{prop}\label{prop-sgn}
Let $\nu: [N]\to \{\pm 1\}$ be the mapping defined as
 	\begin{equation}
 	    \label{eqn-eta-def}
		\nu(i) = \1_{\{i \le N/2\}} - \1_{\{i > N/2\}}.
	 	\end{equation}
	so that $\nu(p_k)=-\nu(q_k)$ if $r_k=\romnum{4}$.
Suppose that $\mathbf{r} = (r_k)_{k \in [K]} \in \mathcal (R^{\SphN})^K$ is a sequence of relations.  
Then, if $a_1,\ldots,a_{2K}$ are such that $a_{p_k}, a_{s(p_k)}, a_{q_k}, a_{s(q_k)}$ satisfies the relation $r_k$ for each $k$, then \begin{equation}
\prod_{\{k : r_k = \romnum{4}\}} 	\nu(a_{p_k})\nu(a_{s(p_k)})\label{eqn-defn-sgn}
\end{equation}
is equal to $(-1)$ raised to the power of the non-orientable genus of $H(\pi,\bm r)$ (which is zero if the surface is orientable).
\end{prop}

We first prove Proposition~\ref{prop-both-lemmas} assuming Proposition~\ref{prop-sgn}. Then, in the next section, we prove Proposition~\ref{prop-sgn}.

\begin{proof}[Proof of Proposition~\ref{prop-both-lemmas} assuming Proposition~\ref{prop-sgn}]
    By~\eqref{eqn-casimir}, for a fixed $a: [2K]\to [N]$,~\eqref{eq-covariance} equals 
    \begin{equation}
        \label{eqn-cov-N-single}
    \begin{cases}
    \left(-\frac{1}{N}\right){\1_{\romnum{1}}(k)} & \text{if $G=\UN$} \\
		\left(-\frac{1}{N}\right){\1_{\romnum{1}}(k)}
		+\left(\frac{1}{N}\right){\1_{\romnum{2}}(k)}, & \text{if $G=\SON$} \\
		\left(-\frac{1}{N}\right){\1_{\romnum{1}}(k)}+
		\left(\frac{1}{N^2}\right){\1_{\romnum{3}}(k)}, & \text{if $G=\SUN$} \\
		\left(-\frac{1}{N}\right){\1_{\romnum{1}}(k)}
		+\nu(a_{p_k})\nu(a_{s(p_k)})\left(\frac{1}{N}\right){\1_{\romnum{4}}(k)}, & \text{if $G=\SphN$} 
    \end{cases}
    \end{equation}
     where $\1_{n}(k)$ is the indicator that $a_{p}, a_{s(p)}, a_{q}, a_{s(q)}$ satisfy the relation of type $n$.
    Substituting each term of the product~\eqref{eq-weight-term} by  the corresponding expression in~\eqref{eqn-cov-N-single}, we deduce that~\eqref{eq-weight-term} equals the sum of
     \begin{equation}
        \label{eqn-omega-term-single}
    \begin{cases}
    (-N)^{-\#_{\romnum{1}}(\mathbf{r})}, & \text{if $G=\UN$} \\
	(-N)^{-\#_{\romnum{1}}(\mathbf{r})} N^{-\#_{\romnum{2}}(\mathbf{r})}, & \text{if $G=\SON$} \\
	 (-N)^{-\#_{\romnum{1}}(\mathbf{r})} N^{-2\#_{\romnum{3}}(\mathbf{r})}, & \text{if $G=\SUN$} \\
	 (-N)^{-\#_{\romnum{1}}(\mathbf{r})} (-1)^{\mu(H(\pi,\bm r))} N^{-\#_{\romnum{4}}(\mathbf{r})}, & \text{if $G=\SphN$} 
    \end{cases}
    \end{equation}
    over all $\mathbf{r} \in (\mathcal{R}^G)^K $ such that $a$ is compatible with $\mathbf{r}$. (We applied Proposition~\ref{prop-sgn} to obtain the expression for $\SphN$.) Therefore, the sum~\eqref{eqn-exp-single-2} equals the sum of~\eqref{eqn-omega-term-single} over all $\mathbf{r} \in (\mathcal{R}^G)^K $ and all $a$ compatible with $\mathbf{r}$.  If we fix $\mathbf{r}$, then by Lemma~\ref{lem-label-bijection}, we can replace the sum over labelings $a$ compatible with $\mathbf{r}$ by a sum over all labelings of the set $V(H(\pi,\mathbf{r}))$ of vertices of $H(\pi,\mathbf{r})$.  Therefore,~\eqref{eq-weight-term} is equal to $N^{|V(H(\pi,\mathbf{r}))|}$ times the expression~\eqref{eqn-omega-term-single}.  Now, since $H(\pi,\mathbf{r})$ has $\#_{\romnum{1}}(\mathbf{r}) + \#_{\romnum{2}}(\mathbf{r})+ \#_{\romnum{4}}(\mathbf{r})$ edges and a single face, we have
    \[
    |V(H(\pi,\mathbf{r}))| = \chi(H(\pi,\bm r)) - 1 + \#_{\romnum{1}}(\mathbf{r}) + \#_{\romnum{2}}(\mathbf{r})+ \#_{\romnum{4}}(\mathbf{r}).
    \]
    Therefore, $N^{|V(H(\pi,\mathbf{r}))|}$ times~\eqref{eqn-omega-term-single} is equal to the summand in~\eqref{eqn-cov-N-single}.
\end{proof}


\subsection{The case of the symplectic group}

We now turn to proving Proposition~\ref{prop-sgn}. We note that Definition~\ref{defn-operations} and Proposition~\ref{prop-equiv-C} below appear elsewhere (in slightly different terms), where they are used to prove the so-called classification theorem for compact surfaces; see~\cite[Sections 6.1-6.2]{Gallier} for a comprehensive treatment.

We first formulate the data $H$, $\pi$ and $\bm r$ slightly differently. 
We recall that, to construct $H(\pi, r)$, we started with a CW complex $H$ given by an oriented surface (or \emph{face}) with cycle graph boundary.  We view $\pi$ as a matching of the edges of $H$, and $\bm r$ as a set of constraints  on the labels of the vertices of pairs of matching edges.  
We can encode $\pi$ and $\bm r$ as follows. Let $E$ be the set of edges of $H$, with $|E| = 2K$. We consider some set of labels $L$ with $|L| = K$, and we let ${L}^{-1}$ be the set of formal inverses of elements of $H$.  We define a map from $E$ to $L \cup {L}^{-1}$ that labels each edge of $H$ by an element of 
$L \cup {L}^{-1}$ as follows.
For each element $x \in L$, the preimage of $\{x,x^{-1}\}$ under the edge labeling map is a pair of edges $e,e'$ matched by $\pi$.  If the pair $e,e'$ are associated to the relation $\romnum{1}$, then we label the edges $x$ and $x^{-1}$; if the pair is associated to the relation $\romnum{4}$, then we label both edges $x$ (or both edges $x^{-1}$).   
This labeling of the edges of $H$ determines both the matching $\pi$ and the relations $\bm r$.

To compute~\eqref{eqn-defn-sgn} for a labeling $\bm a$ of the vertices of $H$, we will ``simplify'' $H$ by a sequence of elementary transformations of its faces and edge labels.  These transformations are defined so that they leave~\eqref{eqn-defn-sgn} unchanged. 

In other words, we express $H$ as an element of the following set. 

\begin{defn}
We let $\mcl C$ be the set of all collections $C$ of faces such that
\begin{itemize}
    \item  the boundaries of the faces are cycle graphs;
    \item each edge is labeled by an element of $L$ or its formal inverse $L^{-1}$, for some set $L$; and
    \item
    for each $x \in L$, exactly two edges are labeled by an element of the set  $\{x,x^{-1}\}$.
\end{itemize}
\end{defn}

We constructed a surface $H(\pi,\bm r)$ by gluing pairs of edges of $H$ with matching edges labels. Similarly, each element $C$ of $\mcl C$ determines a collection of surfaces constructed as follows. 

\begin{defn}
Let $C \in \mcl C$. We orient each edge is oriented counterclockwise around its face, and glue matching pairs of edges of $C$ as follows.
\begin{itemize}
    \item 
If a pair of edges $e$ and $e'$ are labeled $x$ and $x^{-1}$, respectively, for some $x$, then we glue the edges $e$ and $e'$ so that the initial vertex of $e$ is identified with the terminal vertex of $e'$, and the terminal vertex of $e$ is identified with the initial vertex of $e'$.
    \item
If a pair of edges $e$ and $e'$ are both labeled $x$ or both labeled $x^{-1}$ for some $x$, then we glue the edges $e$ and $e'$ so that the initial vertices of $e$ and $e'$ are identified and the terminal vertices of $e$ and $e'$ are identified.
\end{itemize}
We call the resulting collection of surfaces the \emph{surfaces constructed from $C$}.
\end{defn}

We say that a path of edges in the boundary of a face has label $x_1\cdots x_n$ if the edges in the path are labeled  $x_1,\ldots,x_n$ in counterclockwise order around the face. If the path of edges is the entire boundary, then we call  $x_1\cdots x_n$ a label of the boundary.  Note that the label of the boundary is unique only up to a cyclic permutation: if $x_1\cdots x_n$ is a  boundary label, then so is $x_{p+1} \cdots x_n x_1 \cdots x_p$ for any $p$.

As we explained above, the edge labels of $H$ describe relations between matching pairs of edges.  We recall that these relations impose constraints on the labels of the \emph{vertices} of $H$.  Generalizing to $C \in \mcl C$, we consider labelings of the vertices of $C$ by elements of $[N]$, and we view the edge labels of $C$ as imposed relations on these vertex labels.  We call a vertex labeling  \emph{compatible} if it obeys these constraints.

\begin{defn}[Vertex labelings of $C$]
Let $C \in \mcl C$, and let $\bm a$ be a labeling of the vertices of $C$ by elements of $[N]$. For an edge $e$ in $C$, we define $\iota(e)$ and $\tau(e)$ as the vertex labels of the initial and terminal vertices of $e$, respectively, with the edge oriented counterclockwise around its face.  We say that $\bm a$ is a \emph{compatible} labeling of $C$  if
the following is true.
\begin{itemize}
    \item 
If two edges $e,e'$ are labeled $x$ and $x^{-1}$ for some $x$, then their respective vertex labels $\iota(e),\tau(e)$ and $\iota(e'),\tau(e')$ satisfy relation  $\romnum{1}$; i.e., $\iota(e)=\tau(e')$ and $\iota(e')=\tau(e)$.  
\item
If both edges are labeled $x$ or both edges are labeled $x^{-1}$ (for some $x$), then their respective vertex labels
 $\iota(e),\tau(e)$ and $\iota(e'),\tau(e')$ satisfy relation  $\romnum{4}$; i.e., $\iota(e) \equiv \iota(e')+N/2$ and $\tau(e) \equiv \tau(e')+N/2$ modulo $N$.
 \end{itemize}
\end{defn}

We now define the product~\eqref{eqn-eta-def} for general $C \in \mcl C$ with compatible labeling $\bm a$. 

\begin{defn}\label{defn-eta-C}
We let  $\nu(C, \bm a)$ be the product of
\[
\nu(\iota(e)) \nu(\tau(e))
\]
over all pairs of edges $e,e'$ of $K$ with the same label (i.e., both $x$ or both $x^{-1}$ for some $x$). 
\end{defn}

We illustrate Definition~\ref{defn-eta-C} with two examples.

\begin{eg}\label{ex-C}

\begin{enumerate}
    \item 
Suppose that $C$ has a single face with boundary label  
\begin{equation}
    \label{eqn-et-1}
x_1 y_1 x_1^{-1} y_1^{-1} \cdots x_m y_m x_m^{-1} y_m^{-1}
\end{equation}
for some distinct labels $x_1,\ldots,x_m,y_1,\ldots,y_m$ in $L$.
Then $\nu(C,\bm a) = 1$ for \emph{any} labeling $\bm a$ of the vertices of $C$, since no two edges have the same edge label.
\item
Suppose that $C$ has a single face with boundary label  
\begin{equation}
    \label{eqn-et-2}
x_1 x_1 x_2 x_2 \cdots x_m x_m
\end{equation}
for some distinct labels $x_1,\ldots,x_m$ in $L$. Let $e_j$ and $e_j'$ be the edges 
labeled $x_j$ and $x_j^{-1}$, respectively. Then $\tau(e_j) = \iota(e_j')$, and  $\iota(e_j) \equiv \iota(e_j') + N/2$ modulo $N$. This means that $\nu(\iota(e_j))\nu(\tau(e_j)) = -1$.  We deduce that $\nu(C,\bm a) = (-1)^m$ for \emph{any} labeling $\bm a$ of the vertices of $C$.
\end{enumerate}
\end{eg}

The values of $\nu(C,\bm a)$ we computed in Example~\ref{ex-C} have a natural geometric interpretation in terms of the surfaces constructed from $C$.

\begin{prop}\label{prop-eachface}
    Let $C \in \mcl C$.  If each face of $C$ has one of the forms~\eqref{eqn-et-1} or~\eqref{eqn-et-2}, then for each vertex labeling $\bm a$ of $C$, the sign $\nu(C,\bm a)$ is equal to the sum of the non-orientable genuses of the surfaces constructed from $C$.
\end{prop}

\begin{proof}
    The surface constructed from a face of the form~\eqref{eqn-et-1} is orientable, hence has zero non-orientable genus.    The surface constructed from a face of the form~\eqref{eqn-et-2}  has  non-orientable genus $m$ (see, e.g.,~\cite[Section 6.4]{Gallier}). Both these numbers agree with the value of $\nu(C,\bm a)$ we computed in the above example.  This implies the result for a collection of faces of the forms~\eqref{eqn-et-1} and~\eqref{eqn-et-2}.
\end{proof}

To apply Proposition~\ref{prop-eachface} to prove Proposition~\ref{prop-sgn}, we now show that we can reduce \emph{any} collection of faces $C \in \mcl C$ to a collection of faces of the form~\eqref{eqn-et-1} or~\eqref{eqn-et-2} by a sequence of \emph{elementary transformations} that leave $\nu(C,\bm a)$ unchanged.

\begin{defn}\label{defn-operations}
Given two elements $C,C' \in \mcl C$ with labelings $\bm a, \bm a'$ of its vertices, we say that $(C',\bm a')$ is an \emph{elementary transformation} of $(C,\bm a)$  if we can obtain $(C',\bm a')$  from $(C,\bm a)$  by one of the following four operations.
\begin{enumerate}[(I)]
\item\label{p4}
We contract a path of edges with label $xx^{-1}$.  We do not change the vertex labels of any edges.
\item\label{p1}
For a pair of edges labeled $x,x^{-1}$, we replace the edge labeled $x$ by a path labeled $yz$, and  $x^{-1}$ by $z^{-1} y^{-1}$, for some $y,z$ that do not appear in $C$.  If $x$ (resp.\ $x^{-1}$) has vertex labels $a,b$, then we give $y$ and $z$ (resp.\ $z^{-1}$ and $y^{-1}$)  vertex labels $a,c$ and $c,b$ for some $c$.  
\item\label{p2}
We consider a face with some boundary label $x_1\cdots x_p x_{p+1} \cdots x_m$, and we replace this face by one with boundary label  $x_1\cdots x_p w$ and one with boundary label $w^{-1} x_{p+1}\cdots x_m$, for some edge label $w$ that does not appear in $K$.  We do not change any vertex labels of the edges $x_1,\ldots, x_m$; this determines the vertex labels of the edges $w,w^{-1}$.
\item\label{p3}
We consider a face $F$ with boundary label  $x_1 \cdots x_m$, and we replace it by one with boundary label $x_m^{-1}\cdots x_1^{-1}$.  If an edge $e \in F$  had vertex labels $\iota(e) = a$ and $\tau(e) = b$, then we set the vertex labels of the corresponding edge $e'$ in the new face to be $\iota(e') \equiv b+N/2$ and $\tau(e') \equiv a+N/2$ modulo $N$.
\end{enumerate}
We define an equivalence relation on $\mcl C$ as  the  equivalence relation generated by the elementary subdivision relation.
\end{defn}

    \begin{figure}[ht]
          \centering
\includegraphics[width=\linewidth]{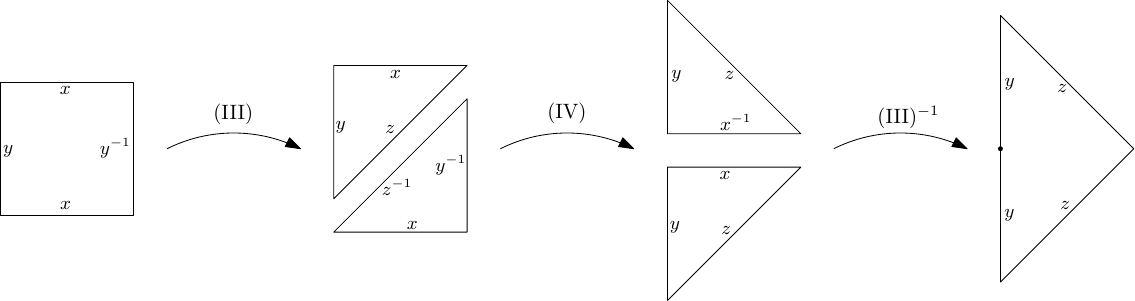} \caption{An example of replacing an element of $\mcl C$ by one of the form~\eqref{eqn-et-2} by a sequence of elementary transformations.}
        \label{fig-trans}
    \end{figure}

It is shown, e.g., in~\cite[Lemma 6.1]{Gallier} that we can reduce an arbitrary $C \in \mcl C$ to one with faces of the forms~\eqref{eqn-et-1} and~\eqref{eqn-et-2}.

\begin{prop}\label{prop-equiv-C}
Each $C \in \mcl C$ is equivalent, under the equivalence relation defined in Definition~\ref{defn-operations}, to an element of $\mcl C$ with each face of the form~\eqref{eqn-et-1} or~\eqref{eqn-et-2}.
\end{prop}

\begin{proof}
    This result appears in~\cite[Lemma 6.1]{Gallier}, though the definitions and terminology differ slightly.
\end{proof}

Moreover, the four operations in Definition~\ref{defn-operations} do \emph{not} affect the value of the sign $\nu(C,\bm a)$, as the following proposition asserts.

\begin{prop}\label{prop-eta-inv}
    Suppose that $(C',\bm a')$ is an elementary transformation of $(C,\bm a)$.  If $\bm a$ is a compatible labeling of $C$, then $\bm a'$ is a compatible labeling of $C'$, and $\nu(C,\bm a) = \nu(C',\bm a')$.
\end{prop}

\begin{proof}
    The fact that $\bm a'$ is a compatible labeling of $K'$ is easy to check.  We now explain why $\nu(K,\bm a) = \nu(C',\bm a')$.
If we obtain $(C',\bm a')$  from $(C,\bm a)$ by either of the   operations~\eqref{p1} or~\eqref{p2}, then $\nu(C,\bm a) = \nu(C',\bm a')$ since we have not changed or added any pair of edges with the same label.  Now suppose that  we  obtain $(C',\bm a')$  from $(C,\bm a)$ by the  operation~\eqref{p3}.  Suppose we replaced the face $F$ in $C$ with a face $F'$. 

Suppose that $e$ and $e'$ are a matching pair of edges in $C$.  These two edges correspond to a matching pair of edges in $C'$; we will denote this matching pair of edges in $C'$ by $e$ and $e'$ as well.  If neither $e$ nor $e'$ is in $F$, or if both $e$ and $e'$ are in $F$, then the operation~\eqref{p3} does not change the quantity $\nu(\iota(e)) \nu(\tau(e))$.  Thus, we can express the quotient $\nu(C')/\nu(C)$ as the product of $\nu(\iota(e)) \nu(\tau(e))$ over all pairs of matching edges $e,e'$ such that $e \in F$ and $e' \notin F$. Since \[\nu(\iota(e)) \nu(\tau(e)) \nu(\iota(e')) \nu(\tau(e')) = 1\]
for any pair of matching edges $e,e'$, this means that $\nu(C')/\nu(C)$  is the product of $\nu(\iota(e)) \nu(\tau(e))$ over all edges $e \in F$.  This product is equal to $1$ since each vertex is the terminal vertex of one edge and the initial vertex of another edge. 
\end{proof}

\begin{proof}[Proof of Proposition~\ref{prop-sgn}]
    The result follows from Propositions~\ref{prop-eachface},~\ref{prop-equiv-C}, and Proposition~\ref{prop-eta-inv}.
\end{proof}

\begin{proof}[Proof of Lemmas~\ref{lem-un} and~\ref{lem-other}]
 Both lemmas are stated as Proposition~\ref{prop-both-lemmas}.
\end{proof}

\subsection{Generalizing to multiple loops}\label{sec-multiple}

We now describe the generalizations of Lemmas~\ref{lem-main-poisson},~\ref{lem-un} and~\ref{lem-other} to any \emph{finite collection} of loops in the plane.\footnote{The multiple-loop version of Wilson loops are also referred as ``Wilson skeins'' in~\cite{Levy2011a}.} Let $\Gamma^{(1)},\ldots,\Gamma^{(n)}$ be loops in the plane, and let $(\lambda_\ell)_{\ell \in [L]}$ be a collection of lassos such that each $\Gamma^{(i)}$ has the lasso representation
\begin{equation}
    \label{eqn-lasso-i}
\lambda_{c^{(i)}(1)}^{\ep^{(i)}(1)} \cdots \lambda_{c^{(i)}(M^{(i)})}^{\ep^{(i)}(M^{(i)})}
\end{equation}
for some $c^{(i)}:[M^{(i)}] \to [L]$ and $\ep^{(i)}:[M^{(i)}] \to \{-1,1\}$.  We let $M = \sum_{i=1}^n M^{(i)}$, and we define the ``concatenated'' mappings $c:[M] \to [L]$ and $\ep:[M] \to \{-1,1\}$  such that the concatenation of loops $\Gamma^{(1)} \cdots \Gamma^{(n)}$ has lasso representation~\eqref{eqn-lasso-rep}.  We define the space matching-color pairs $\mcl D$ in terms of this concatenated mapping $c$ as
\begin{equation}
    \label{eqn-def-mcl-D-2}
\mcl D := \bigsqcup_{\stackrel{m<m^*}{c(m) = c(m^*)}} D_{(m,m^*)},
\end{equation}
where $D_{(m,m^*)}$ is isomorphic to the region bounded by $\lambda_{c(m)} \equiv \lambda_{c(m^*)}$. 
If $\Sigma$ is a collection of points in $\mcl D$, then we can define the associated sign $\ep(\Sigma)$ and pairing $\pi(\Sigma)$ as in Lemma~\ref{lem-main-poisson}.

If we were considering the Wilson loop expectation of the concatenated loop $\Gamma^{(1)} \cdots \Gamma^{(n)}$, then the Wilson loop expectation would be given precisely by the expression in Lemma~\ref{lem-main-poisson}.  However, since we are considering the $n$ loops separately, we obtain a slightly different expression for the term~\eqref{eqn-exp-single-2}.  In the case of a single loop, the number $2|\Sigma|$   appears in~\eqref{eqn-exp-single-2} because the points $\Sigma$ correspond to $2|\Sigma|$ points in the space $\bigsqcup_{m = 1}^M D_m$,
where $D_m$ is the region bounded by $\lambda_{c(m)}$.
In the case of multiple loops, the points $\Sigma$ corresponds to $2|\Sigma|$ points in   the space
\begin{equation}\label{eqn-Dim}
\bigsqcup_{i=1}^n\bigsqcup_{m = 1}^{M^{(i)}} D^{(i)}_m, 
\end{equation}
where $D^{(i)}_m$ is isomorphic to the region bounded by $\lambda_{c^{(i)}(m)}$.  In other words,  $\Sigma$ determines $K_i$ points in
\[\bigsqcup_{m = 1}^{M^{(i)}} D^{(i)}_m\]
for each $i$, where $\sum_{i=1}^n K_i = 2|\Sigma|$.  
The result is that, in the case of multiple loops, we replace the term~\eqref{eqn-exp-single-2} by a term of the form
\begin{equation}
    \label{eqn-exp-multiple}
    \left\langle W_{a^{(1)}_1 a^{(1)}_2}, \ldots, W_{a^{(1)}_{K^{(1)}} a^{(1)}_{K^{(1)}+1}}, \ldots,
W_{a^{(n)}_1 a^{(n)}_2}, \ldots, W_{a^{(n)}_{K^{(n)}} a^{(n)}_{K^{(n)}+1}} \right\rangle_{\pi(\Sigma)}.
\end{equation}
In summary, we obtain the following analogue of Lemma~\ref{lem-main-poisson}.

\begin{lem}\label{lem-main-poisson-gen}
Let $G$, $\frk g$, and $W$ be as in Lemma~\ref{lem-main-poisson}, and consider the collection of loops $\Gamma^{(1)},\ldots,\Gamma^{(n)}$ with lasso representations~\eqref{eqn-lasso-i}.  Let $\Sigma$ be a Poisson point process in the space $\mcl D$ of matching-color lasso pairs, and let $\ep(\Sigma)$, $\pi(\Sigma)$, and the numbers $K_i$ with $\sum_i K_i = 2|\Sigma|$ be defined as above.  Then the Wilson loop expectation of the collection of loops $\Gamma^{(1)},\ldots,\Gamma^{(n)}$ in the gauge group $G$ is given by the constant
\begin{equation}\label{eqn-lem-const}
\exp\left(\frac{\frk c_{\frk g}}{2} \sum_{m=1}^M  \abs{\lambda_{c(m)}} + \sum_{\stackrel{m<m^*}{c(m) = c(m^*)}} |\lambda_{c(m)}| \right)
\end{equation}
times the expected value of $\ep(\Sigma)$ times the sum of
\begin{equation}
    \label{eqn-exp-multiple-2}
    \left\langle W_{a^{(1)}_1 a^{(1)}_2}, \ldots, W_{a^{(1)}_{K^{(1)}} a^{(1)}_{K^{(1)}+1}}, \ldots,
W_{a^{(n)}_1 a^{(n)}_2}, \ldots, W_{a^{(n)}_{K^{(n)}} a^{(n)}_{K^{(n)}+1}} \right\rangle_{\pi(\Sigma)}
\end{equation}
over all 
 $a^{(i)}_1,\ldots,a^{(i)}_{K^{(i)}+1} \in [N]$ with $a^{(i)}_{K^{(i)}+1} = a^{(i)}_1$, for each $i$.
\end{lem}

\begin{proof}
The proof of Lemma~\ref{lem-main-poisson-gen} is essentially the same as that of Lemma~\ref{lem-main-poisson}, with only minor adjustments.  We omit the details.
\end{proof}

Next, we generalize Lemmas~\ref{lem-un} and~\ref{lem-other}.  In proving these lemmas, we wrote  $\pi(\Sigma)$ canonically as
\[
\pi(\Sigma) = (p_1 \; q_1) \cdots (p_K \; q_K), \qquad p_j < q_j, \quad p_1 < \cdots < p_K,
\]
where $K = |\Sigma|$,
and we expressed~\eqref{eqn-exp-single-2} as 
$\prod_{k=1}^{|\Sigma|} \avg{W_{a_{p_k}a_{s(p_k)}} W_{a_{q_k}a_{s(q_k)}}}$,
where $s$ is the cyclic permutation $(1 \cdots 2|\Sigma|)$.  In the case of $n$ loops, we have the same expression, with $s$ replaced by a permutation with $n$ cycles.  Specifically, if $J_k$ is the permutation matrix of the cyclic permutation $(1 \cdots k)$, then $s$ is the permutation
\begin{equation}\label{eqn-pmatrix}
\begin{pmatrix}
J_{K^{(1)}} & 0 & \cdots & 0 \\
0 & J_{K^{(2)}} & \cdots & 0 \\
0 & 0 & \ddots & 0 \\
0 & 0 & \cdots & J_{K^{(n)}}
\end{pmatrix},
\end{equation}
with $K^{(1)},\ldots,K^{(n)}$ the functions of $\Sigma$ defined above.  We can apply the same arguments in Section~\ref{sec-surface-story}, where instead of starting from a single region $H$ with directed cycle graph boundary of length $2|\Sigma|$, we consider $n$ regions $H^{(1)},\ldots,H^{(n)}$ with directed cycle graph boundaries of lengths $K^{(1)},\ldots,K^{(n)}$. (If $K_i = 0$, then $H^{(i)}$ has the sphere topology; if $K_i > 0$, then $H^{(i)}$ has the disk topology.) As in the single loop case, we label the vertices of the boundaries $\partial H^{(i)}$ by $1,\ldots,{2|\Sigma|}$ in order from $i=1$ to $n$ and in cyclic order around each boundary.\footnote{In other words, the vertices of $\partial H^{(i)}$ are labeled by ${\sum_{j=1}^{i-1} K_i+1}, \ldots, {\sum_{j=1}^{i} K_j}$ in cyclic counterclockwise order. } By definition of $s$, the edges of the boundaries  $\partial H^{(i)}$ have vertex labels $p$ and $s(p)$ for $p$ ranging in $[2K]$.  We view $\pi(\Sigma)$ as matching the edge with vertex labels $p_k$ and $s(p_k)$ to the edge with vertex labels $q_k$ and $s(q_k)$, for each $k \in [K]$. By repeating all the arguments in Section~\ref{sec-surface-story}, we obtain the following generalizations of Lemmas~\ref{lem-un} and~\ref{lem-other}.\footnote{We intentionally formulated the proof of Proposition~\ref{eqn-defn-sgn} so that the proof still applies if we start from a collection of faces $H^{(1)},\dots, H^{(n)}$.}

\begin{lem}\label{lem-un-gen}
 Let $G = U(N)$.  Then~\eqref{eqn-exp-multiple-2} is equal to $(-1)^{|\Sigma|} N^{-n+ \chi}$, where $\chi$ is the sum of the Euler characteristics of the surfaces formed from  $H^{(1)},\ldots,H^{(n)}$ by gluing pairs of edges matched by $\pi(\Sigma)$ with the opposite orientation.
\end{lem}

\begin{lem}\label{lem-other-gen}
Let $G = \SON$, $\SUN$ or $\SphN$. Let $H^{(1)},\ldots,H^{(n)}$ be defined as above, and consider the $2^{|\Sigma|}$ ways to construct a collection of surfaces from these $n$ disjoint regions by gluing or contracting each of the $|\Sigma|$ pairs of edges matched by $\pi(\Sigma)$ according to one of the two operations listed in Table~\ref{table-gluing}.  Then we can express~\eqref{eqn-exp-multiple-2} as the sum, over this set of $2^{|\Sigma|}$ collections of surfaces, of 
 \begin{equation}
        \label{eqn-cov-N-multiple}
    \begin{cases}
    	(-1)^{\alpha}  N^{-n+\chi}, & \text{if $G=\SON$} \\
		(-1)^{\alpha}  N^{-n- \gamma +\chi}, & \text{if $G=\SUN$} \\
		(-1)^{\alpha+\mu}  N^{-n+\chi}, & \text{if $G=\SphN$} 
    \end{cases}
    \end{equation}
    where
    \begin{itemize}
    \item $\chi$ is the sum of the Euler characteristics of the surfaces,
    \item $\alpha$ is  the number of pairs of edges glued with the opposite orientation,
    \item $\gamma$ is the number of contracted edges (which is always even), and
    \item $\mu$ is sum of the non-orientable genuses of the surfaces (which is zero if the surface is orientable).
\end{itemize}
\end{lem}
\begin{proof}[Proofs of Lemmas~\ref{lem-un-gen} and~\ref{lem-other-gen}]
 Lemmas~\ref{lem-un-gen} and~\ref{lem-other-gen} are proved by following the same arguments as those used to prove Lemmas~\ref{lem-un} and~\ref{lem-other}, with $H$ replaced by $H^{(1)},\ldots,H^{(n)}$.  We omit the details.
\end{proof}

\subsection{Surfaces spanning loops}\label{sec-spanning}

In this section, we construct surfaces topologically equivalent to the CW complexes in Definition~\ref{defn-H} so that they ``span'' a given loop $\Gamma$. Combined with Lemma~\ref{lem-other-gen} where Wilson loop expectations are expressed in terms of sums over CW complexes, we eventually establish a complete story of gauge-string duality for the 2D Yang-Mills theory, as stated in Theorem~\ref{thm-main-simple}.

Let $\Gamma$ be a loop in the plane. After choosing a root and a spanning tree, we construct the lassos, e.g.\ as described in Figure~\ref{fig-eg-lasso-g0}. However, instead of shrinking the edges of the spanning tree to a point, we keep the copies of non-tree edges that consist of $\Gamma$ and the tree edges attached to it. After ``fan out'' these copies in the clockwise order around the root, we may glue tree edges to obtain the beaded disk for $\Gamma$, called the \emph{reference covering} of $\Gamma$ with the specified root and spanning tree. See Figure~\ref{fig-eg-petals2-g0} (A-C) for an illustration of this construction. Figure~\ref{fig-petals-complicated}\ (A) is another example of a reference covering for a more complicated loop.
We note that a similar notion has been introduced in~\cite{Shor1992,Mukherjee2011}.

\begin{figure}[ht!] \centering
	\includegraphics[width=\textwidth]{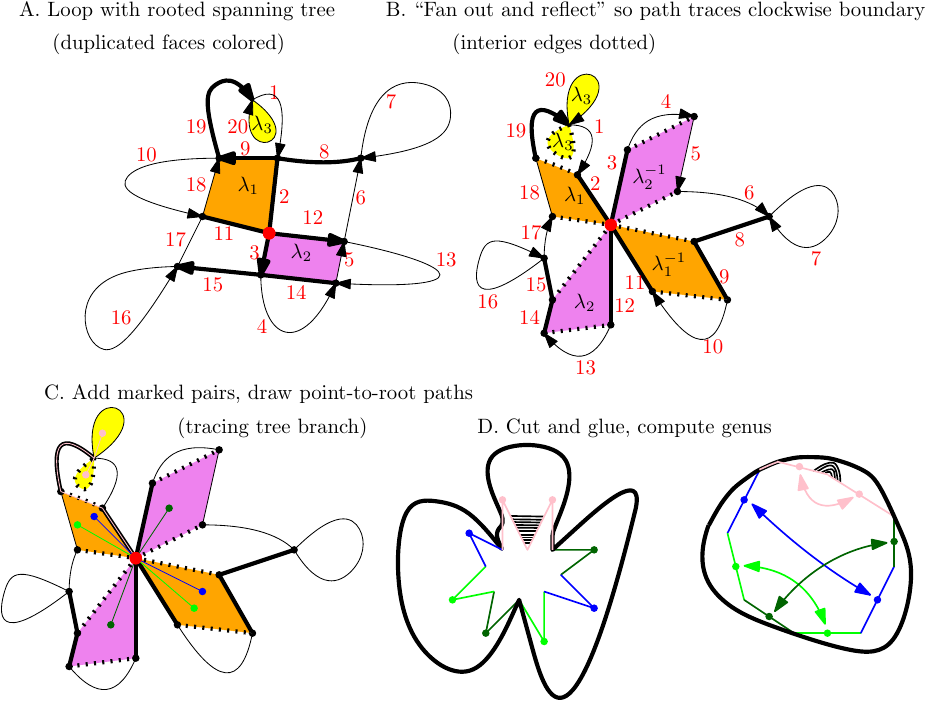}
	\caption{\textbf{(A)} A complicated loop drawn in the plane. \textbf{(B)} Fan out and reflect to produce a beaded disk. (See Figure~\ref{fig-eg-petals2-g0} for a simpler example.) \textbf{(C)} Some point-to-root paths are drawn on the beaded disk. \textbf{(D)}
    The point-to-root paths may intersect boundary and/or overlap each other, but this turns out not to affect the genus computation, see Figure~\ref{fig-gluing-trees} for further explanation.
 }\label{fig-petals-complicated}
\end{figure}

\begin{defn}
    \label{defn-reference-covering}
    We call a map $\phi: D\to \BB R^2$ of $(\BB R^2, \Gamma)$ a \textbf{reference covering} if $D$ is a beaded disk and it satisfies the following.
    \begin{itemize}
    \item When you restrict to the circle that forms the boundary of the disk, you get a parameterization of the loop $\Gamma$.
    \item If $A$ is a bounded region in the complement of $\Gamma$ then $\phi^{-1}(A)$ has some finite number $k$ of components, and $\phi$ maps each component homeomorphically onto $A$.
    \end{itemize}
    In addition, we call the reference covering of $\Gamma$ constructed from a choice of root and spanning tree as described in the previous paragraph the \textbf{rooted reference covering with spanning tree} of the loop $\Gamma$.
\end{defn}

From the rooted reference covering with spanning tree of $\Gamma$, we can construct a surface analogous to the procedure described in Figure~\ref{fig-eg-surface-g0}. Suppose that pairs of points are given on the faces of the reference covering. These points are referred as the pairs of \emph{ramification points}. Then we construct a path to each ramification point so that when we shrink the spanning to a point, the paths should be identical to the slits we constructed in Figure~\ref{fig-eg-surface-g0}. Some paths may share the same edge of the reference covering. See Figure~\ref{fig-eg-petals2-g0}\ (E) for the pahts to ramification points in the case of $\Gamma_0$, and compare this with Figure~\ref{fig-eg-petals2-g0}\ (F) which describes the corresponding slits described in Section~\ref{sec-surface-story}. We now define a \emph{spanning surface} by cutting and gluing these paths to ramification points. 

\begin{defn}
    \label{defn-spanning}
    We say a surface \textbf{spans the loop} $\Gamma$ as its boundary if the surface is obtained by cutting and gluing path-to-root trees constructed from ramification points.
\end{defn}

See Figure~\ref{fig-petals-complicated} for an example of surfaces spanning a complicated loop. Note that the cutting and gluing still work even though the slit paths share some edges to form a tree. See Figure~\ref{fig-gluing-trees} to see how this works. In particular, it is not difficult to see that the resulting surface obtained by gluing path-to-root trees is topologically equivalent to the CW complex we constructed in Section~\ref{sec-surface-story}.

\begin{figure}[ht!] \centering
	\includegraphics[width=\textwidth]{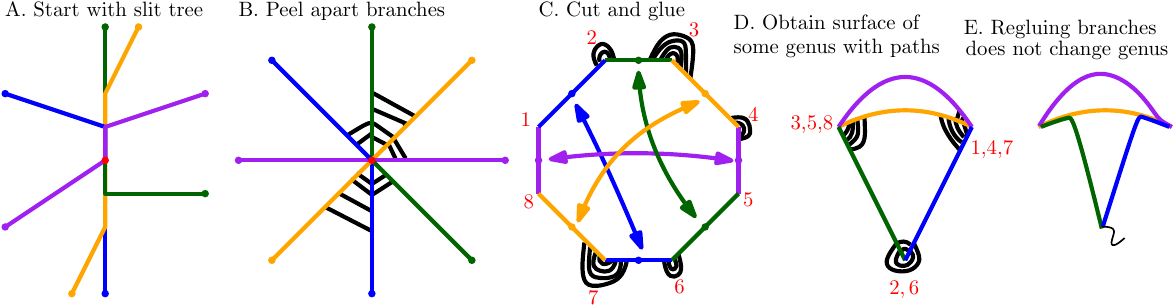}
	\caption{The cutting/gluing construction at the end of Figure~\ref{fig-petals-complicated} still works if the slit paths form a tree (instead of being disjoint) as shown above; we can ignore the extra ``tail'' in (E). The genus depends only on the ordering of edges within the polygon (C). This remains true if the slit paths also intersect or trace the boundary, as in Figure~\ref{fig-petals-complicated} (E). In that scenario, the construction proceeds as follows: we glue an extra disk to the outside of the loop in Figure~\ref{fig-petals-complicated} (E) to produce a topological sphere; we then proceed with the gluing/cutting within that sphere as above to create a new surface without boundary; finally, we can the remove the extra disk at the end to obtain a surface with a boundary loop.
 }\label{fig-gluing-trees}
\end{figure}

\begin{figure}[ht!] \centering
	\includegraphics[width=\textwidth]{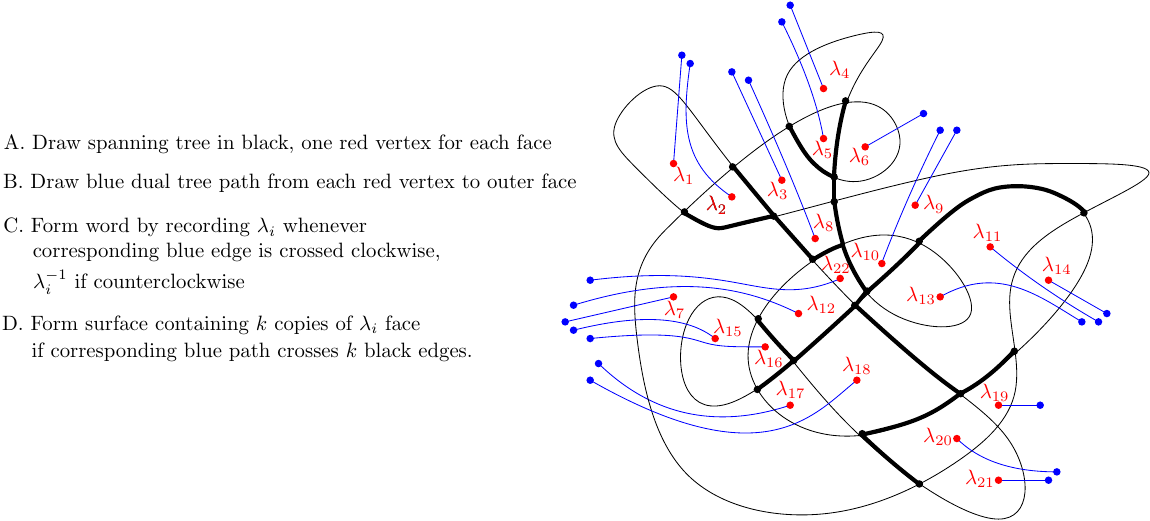}
	\caption{{\bf Another perspective on word and reference surface formation:} Suppose that $b_1,b_2, \ldots, b_M$ are the locations where the black curve intersects a blue segment, presented in the order in which they are traversed by the loop, starting at the root. Every $b_i$ corresponds to one copy of the corresponding face (centered at the red dot at the end of the blue segment). Suppose a black edge $e$ in the loop separates two faces, which together occur $2k+1$ times. Then the $2k+1$ faces will have to be paired up---with $k$ pairs of faces being glued together along $e$ and one unpaired face having $e$ as its boundary. 
 If $e$ is \emph{not} part of the (thick black) spanning tree, then one of the corresponding blue paths must hit $k$ edges (not including $e$) and the other must hit $k+1$ edges (including $e$). In this case it is easy to see how the faces are paired with each other: namely faces are paired when their corresponding $b_i$ occur along the same black edge; the face left unpaired corresponds to the $b_i$ that occurs along $e$. If $e$ \emph{is} an edge of the (thick black) spanning tree, then the first of the $b_i$ hit after the loop traverses $e$ (if $e$ is directed away from the root) or the last of the $b_i$ intersection hit before the loop traverses $e$ (if $e$ is directed toward the root) corresponds to the boundary face; and the remaining edges are paired together according to cyclic ordering. (That is, the first two blue/black intersections reached after the ``boundary face'' intersection are paired with each other, as are the next two, etc.) A pair could include two copies of the same face (with opposite orientations) or one copy of each type of face (with the same orientation).
 }\label{fig-dualtree}
\end{figure}

We end this section with another perspective on the rooted reference covering with spanning tree. The caption of Figure~\ref{fig-dualtree} roughly sketches an alternative way to understand the construction of a reference surface once one is given a loop and a spanning tree.

\section{Applications}\label{sec-applications}
\subsection{Properties of Wilson loop expectations}\label{sec-corollaries}
Our main results (Lemma~\ref{lem-main-poisson-gen},~\ref{lem-un-gen}, and~\ref{lem-other-gen}) have several corollaries. Although none of the statements in this section is original, we give new proofs based on our framework which are probabilistic and relatively simpler than known methods.

We denote by $\Phi_N^{G}(\bm{\Gamma})$ the Wilson loop expectation of the collection of loops $\bm{\Gamma}$ in the gauge group $G$.
Throughout this section, we assume the setting of Lemma~\ref{lem-main-poisson-gen},~\ref{lem-un-gen}, and~\ref{lem-other-gen}. Particularly, recall the definitions of $\mcl D$, $\Sigma$, and $\ep(\Sigma)$ in the first paragraph of Section~\ref{sec-multiple}. In addition, for $C>0$, denote by $\Sigma^C$ a Poisson point process on $\mcl D$ whose intensity is $C$ times the Lebesgue measure. Thus, $\Sigma$ and $\Sigma^1$ has the same distribution. We first observe a simple fact which is straightforward from the definition of $\Sigma^C$:
\begin{prop}\label{prop-sigma-shift}
	For $A, C>0$, let $f$ be some measurable function with respect to $\Sigma^C$. Then the expectation of $A^{\abs{\Sigma^C}} f(\Sigma^C)$ equals to $Z_{A,C}:=\exp\left(C(A-1)\sum_{\stackrel{m<m^*}{c(m) = c(m^*)}} \abs{\lambda_{c(m)}}\right)$ times the expectation of $f(\Sigma^{AC})$.
\end{prop}

The first corollary establishes a connection between the Wilson loop expectations in $\UN$ and $\SUN$.
\begin{cor}[{\cite[Proposition 1.5]{Levy2011a}}]\label{cor-sun}
	\[\Phi_N^{\SUN}(\bm{\Gamma}) = 
 \Ex{\ep(\Sigma^{1/N^2})}
	\exp\left( \sum_{\ell=1}^L \frac{\abs{c^{-1}(\ell)}^2 \abs{\lambda_{\ell}}}{2N^2} \right) \Phi_N^{\UN}(\bm{\Gamma}).\]
\end{cor}
\begin{proof}
 
 In light of Lemma~\ref{lem-main-poisson-gen}, $\Phi_N^{\SUN}(\bm{\Gamma})$ and $\Phi_N^{\UN}(\bm{\Gamma})$ can be represented as the constant~\eqref{eqn-lem-const} times the expectation of~\eqref{eqn-exp-multiple-2} multiplied by $\ep(\Sigma)$. The former constant differs by the factor
	\begin{equation}\label{eqn-const-diff}
		\exp\left({\sum_{m=1}^M\frac{|\lambda_{c(m)}|}{2N^2}}\right)
	\end{equation}
	because $\fc_{\suN} = \fc_{\uN} + 1/N^2$. Thus, it remains to analyze the latter expectation for each case.

 Suppose $G=\SUN$. For each point in $\Sigma$, we flip a fair coin to filter $\Sigma$, which becomes another Poisson point process, denoted by $\Sigma_2$. In particular, $\Sigma_1:=\Sigma\setminus \Sigma_2$ and $\Sigma_2$ have the same law as $\Sigma^{1/2}$. In the context of Lemma~\ref{lem-other-gen}, we may interpret $\Sigma_2$ as points corresponding to \emph{contracted edges} among points in $\Sigma$. This implies that the expectation of~\eqref{eqn-exp-multiple-2} multiplied by $\ep(\Sigma)$ equals to the expectation of
 \begin{equation}\label{eqn-sun-alt}
 \ep(\Sigma_1)\ep(\Sigma_2) 2^{\abs{\Sigma_1}+\abs{\Sigma_2}} N^{-n-2\abs{\Sigma_2}+\chi(\Sigma_1)}
 \end{equation}
 by Lemma~\ref{lem-other-gen}, where $\chi(\Sigma_1)$ is defined as $\chi$ in Lemma~\ref{lem-un-gen} by replacing $\Sigma$ to $\Sigma_1$. Since $\Sigma_1$ and $\Sigma_2$ are independent,~\eqref{eqn-sun-alt} can be reduced to the expectation of
 \begin{equation}\label{eqn-sun-alt-first}
 \ep(\Sigma_1)2^{\abs{\Sigma_1}} N^{-n+\chi(\Sigma_1)}
 \end{equation}
 times the expectation of
 \begin{equation}\label{eqn-sun-alt-second}
 \ep(\Sigma_2) (2/N^2)^{\abs{\Sigma_2}}.
 \end{equation}

 Using Proposition~\ref{prop-sigma-shift} and the constant $Z_{A,C}$ defined there, we have that~\eqref{eqn-sun-alt-first} equals to $Z_{2,1/2}$ times the expectation of $\ep(\Sigma) N^{-n+\chi(\Sigma)}$, and~\eqref{eqn-sun-alt-second} equals to $Z_{2/N^2,1/2}$ times the expectation of $\ep(\Sigma^{1/N^2})$.

	Finally, $Z_{2,1/2}Z_{2/N^2,1/2}$ times~\eqref{eqn-const-diff} equals to
	\[\exp\left({\frac{1}{2N^2}\sum_{m=1}^M {|\lambda_{c(m)}|} + \frac{1}{N^2}\sum_{\stackrel{m<m^*}{c(m) = c(m^*)}} \abs{\lambda_{c(m)}}}\right) = \exp\left(\frac{1}{2N^2}\sum_{{c(m) = c(m^*)}} \abs{\lambda_{c(m)}}\right)\]
	and collecting all terms prove the desired statement.
\end{proof}

The \textit{master field} $\Phi^{G}$ of $\bm{\Gamma}$ is the limit of a normalized holonomy field defined as
\begin{equation}
 \Phi^G(\bm{\Gamma}) := \lim_{N\to \infty} \left(\frac1N\right)^{n} \Phi^G_N(\bm{\Gamma}),
\end{equation}
if exists, that is the Wilson loop expectation with respect to the normalized trace $\tr := \frac1N \Tr$ in the $N\to \infty$ limit, also called the 't Hooft limit. Recall that for a connected compact surface $\mathcal S$, the Euler characteristic $\chi$ and the genus $g\ge 0$ are related as
	\begin{equation*}
		\chi = \begin{cases}
			2-2g & \text{if $\mathcal S$ is orientable,} \\
			2-g & \text{if $\mathcal S$ is non-orientable.}
		\end{cases}
	\end{equation*}
	In particular, the sphere $\mathbb S^2$ is the unique connected surface (up to homeomorphism) that maximizes $\chi$ by the classification theorem of connected compact surfaces. This explains why Wilson loop expectations are nontrivial when normalized by the number of loops raised to the power of $1/N$. The next corollaries show that $\Phi^G(\bm \Gamma)$ exists and is the same in all classical compact Lie groups.

\begin{cor}\label{cor-exists}
 For $G=\UN, \SON, \SUN, \SphN$, the master field $\Phi^G(\bm \Gamma)$ exists.
\end{cor}
\begin{proof}
 In the expression of $\left(\frac1N\right)^{n} \Phi^G_N(\bm{\Gamma})$ from Lemma~\ref{lem-main-poisson-gen}, the constant~\eqref{eqn-lem-const} is convergent when $N\to\infty$. By Lemma~\ref{lem-un-gen}, and~\ref{lem-other-gen}, the other term is expected value of $\ep(\Sigma)$ times the sum of $(-1)^{A} N^{-2n+\chi-B}$ over at most $2^{\abs{\Sigma}}$ surfaces for some $A,B\ge 0$ depending on the choice of $G$. In any case, the expectation is dominated by the expected value of $2^{\abs{\Sigma}}$ (which is finite by Proposition~\ref{prop-sigma-shift}) because $-2n+\chi\le 0$. Therefore, the dominated convergence theorem implies $\Phi^G(\bm \Gamma)$ exists.
\end{proof}
	
\begin{cor}[{\cite[Theorem 2.2]{Levy2011a}}]\label{cor-equals} For $G=\SUN, \SON,\SphN$, we have
	\begin{align}
		\Phi^G(\bm \Gamma) &= \Phi^{\UN}(\bm \Gamma)\nonumber \\
		&= \exp\left(-\frac{1}{2} \sum_{m=1}^M \abs{\lambda_{c(m)}} + \sum_{\stackrel{m<m^*}{c(m) = c(m^*)}} |\lambda_{c(m)}| \right) \Ex{\ep(\Sigma)(-1)^{\abs{\Sigma}} \1_{\{\chi=2n\}}}.\label{eqn-master}
 \end{align}
\end{cor}
\begin{proof}
 The case when $G=\SUN$ follows from Corollary~\ref{cor-sun}, so suppose $G$ is either $\SON$ or $\SphN$.

 We again use Lemma~\ref{lem-main-poisson-gen}, where the constant~\eqref{eqn-lem-const} for $G=\SON,\SphN$ converges to~\eqref{eqn-lem-const} for $G=\UN$ as $N\to\infty$ because $\lim_{N\to \infty}\fc_\fg = -1 = \fc_{\uN}$ for $\fg=\soN,\sphN$ (see Table~\ref{table-lie}). Hence, it remains to compute the expectation part.

 Continued from the proof of Corollary~\ref{cor-exists}, we can take $N\to \infty$ limit first in the expression of Lemma~\ref{lem-un-gen} and~\ref{lem-other-gen} (after normalized by $N^{-n}$) before taking the expectation to compute $\Phi^G(\bm \Gamma)$. In this limit, the only surviving terms are those with vanishing exponent of $N$, which happens exactly when $\chi=2n$ as already discussed. In other words, the glued surface formed from $H^{(1)}, \dots, H^{(n)}$ should be a collection of $n$ spheres, which cannot happen if we glue any of edges in the same orientation. That is, for $G=\SON, \SphN$, the $\alpha$ in~\eqref{eqn-cov-N-multiple} should be $\abs{\Sigma}$ for nontrivial terms after taking the $N\to \infty$ limit.
 
 In conclusion, $\Phi^G(\bm\Gamma)$ equals to
 \[
 \exp\left(-\frac{1}{2} \sum_{m=1}^M \abs{\lambda_{c(m)}} + \sum_{\stackrel{m<m^*}{c(m) = c(m^*)}} |\lambda_{c(m)}| \right)
 \]
 times the expectation of $\ep(\Sigma)(-1)^{\abs{\Sigma}} \1_{\{\chi=2n\}}$ for any $G=\UN,\SON,\SUN,\SphN$ as desired.\qedhere

\end{proof}

This immediately implies the following.
\begin{cor}[{\cite{Biane1997a},~\cite[Proposition 6.25]{Levy2011a}}]\label{cor-factorization}
	The master field can be factorized as
	\begin{equation}
		\Phi^{\UN}(\bm{\Gamma}) = \prod_{i=1}^n \Phi^{\UN}(\Gamma^{(i)}).
 \end{equation}
\end{cor}
\begin{proof}
	Starting from~\eqref{eqn-master}, recall that the event $\chi=2n$ happens when the glued surface formed from $H^{(1)},\dots, H^{(n)}$ is a collection of $n$ spheres. Therefore, in such a case, the gluing restricted to $H^{(i)}$ yields a sphere for each $i\in [n]$. 
\end{proof}

\begin{cor} Let $\NC_{2K}$ be the set of non-crossing partitions of $[2K]$ and $\NC:=\bigcup_{K=0}^\infty \NC_{2K}$. Then
	\begin{equation}
 \Phi^{\UN}(\bm{\Gamma}) =
 \exp\left(-\frac{1}{2} \sum_{m=1}^M \abs{\lambda_{c(m)}} + \sum_{\stackrel{m<m^*}{c(m) = c(m^*)}} |\lambda_{c(m)}| \right) \prod_{i=1}^n \Ex{\ep(\Sigma^{(i)})(-1)^{\abs{\Sigma^{(i)}}} \1_{\{\pi(\Sigma^{(i)})\in \NC \}}}.\label{eqn-master-nc}
	\end{equation}
\end{cor}
\begin{proof}
 By Corollary~\ref{cor-factorization}, it enough to show that in a single-loop case $H(\pi(\bm{\sigma}),\{\romnum{1}\}^{\abs{\bm{\sigma}}})\equiv\mathbb S^2$ if and only if $\pi(\bm{\sigma})\in \NC$. Suppose that $\bm{\sigma}=(\sigma_1, \sigma_2, \dots, \sigma_K)$. Define $\mathbf{r}^{(i)} := (r_1^{(i)}, r_2^{(i)}, \dots, r_K^{(i)})$ where $r_k^{(i)} = \romnum{1}$ if $k\le i$ and $r_k^{(i)} = \romnum{3}$ if $k> i$. Then $H^{(i)} := H(\pi(\bm{\sigma}),\mathbf{r}^{(i)})$ defines a sequence of intermediate surfaces for constructing $H^{(K)} = H(\pi(\bm{\sigma}),\{\romnum{1}\}^{\abs{\bm{\sigma}}})$.
 
 Note that $\chi_{H^{(i)}}$ is non-decreasing as $i$ increases. 
 This is because when the gluing rule is changed from $\mathbf{r}^{(i)}$ to $\mathbf{r}^{(i+1)}$, the number of edges increases by 1 and the number of vertices increases at most by 1, while the number of faces is fixed as 1.
 Hence $H^{(K)}\equiv \mathbb S^2$ if and only if $H^{(i)}\equiv \mathbb S^2$ for all $i\in [K]$.

 If $\pi(\bm{\sigma})\not\in \NC$, without loss of generality, we may assume the first and the second pair of $\pi(\bm{\sigma})$ crosses each other. In this case, $H^{(2)}$ made by gluing of a square is equivalent to the genus-1 torus $\mathbb T$. Thus $H^{(K)}\not\equiv \mathbb S^2$.

 Conversely, if $\pi(\bm{\sigma})\not\in \NC$ then the $(i+1)$'th pair of $\pi(\bm{\sigma})$ does not cross with any of the previous pairs. This implies that when we change the gluing rule from $\mathbf{r}^{(i)}$ to $\mathbf{r}^{(i+1)}$, the number of vertices exactly increases by 1 because two ends of the glued edge by $(i+1)$'th pair is not identified by one of the previous pairs, otherwise there is a crossing. This implies that $H^{(i)}\equiv \mathbb S^2$ for all $i\in [K]$.
\end{proof}

\subsection{The Makeenko-Migdal equation}\label{sec-mm}
In this section, we discuss an alternative proof of the so-called Makeenko-Migdal equation~\cite{makeenko1979, Levy2011a, driver2017a,Driver2017}, which is a variational formula for Wilson loop expectations. See Theorem~\ref{thm-mm} for a version written in our framework.



\begin{prop}\label{prop-smooth}
Let $\Gamma$ be a loop in a planar graph $\BB{G}$, and let $\mcl I$ be the space of plane embeddings $\iota$ of $\BB{G}$ into the plane such that each embedded edge is a non-crossing rectifiable curve. Then the mapping $\iota \mapsto \Phi^G_N(\iota(\Gamma))$ is a smooth function of the areas of the bounded faces of $\iota(\BB{G})$.
\end{prop}
\begin{proof}
    We use Proposition~\ref{prop-new-wilson} for the expression of $\Phi_N^G$ to check that it is a smooth function of the area of the bounded faces.
 Straightforward computation shows that for fixed $r$, $J\ge 1$, $x\ge x_0>0$ and $f(x) = x^K(1+\frac{x}{J})^{-K}$, the $r$-th derivative $f^{(r)}$ satisfies
 \begin{equation*}
 |f^{(r)}(x)|\le x_0^{-r}2^r(K+r)^rf(x);
 \end{equation*}
 and for $x_1, \dots, x_l\ge x_0>0$, $r=\sum_{i=1}^l r_i$, $K=\sum_{i=1}^l K_i$, and $f(x_1, \dots, x_l) = \prod_{i=1}^l x_i^{K_i}(1+\frac{x_i}{J})^{-K_i}$, one has
 \begin{equation}
 \abs{\frac{\partial^r}{\partial x_1^{r_1}\dots x_l^{r_l}}f(x_1, \dots, x_l)}\le x_0^{-r}2^r(K+r)^rf(x_1, \dots, x_r).
 \end{equation}
 Therefore, if we take the $r$-th derivative of the sum~\eqref{eqn-omega-a} times $Z_J^{-1}$, where $Z_J$ is defined in~\eqref{eqn-zj}, it follows that each term is bounded above from the original expression by the factor of $|\lambda_0|^{-r}2^r(K+r)^r$ where $|\lambda_0|$ is the area of the smallest face. From exactly the same argument as in the proof of Proposition~\ref{prop-limit-legal} and~\ref{prop-new-wilson}, this derivative is uniformly bounded from above. As a consequence, the $r$-th derivative of the product of $Z_J^{-1}$ and the sum~\eqref{eqn-omega-a} converges uniformly as $J\to\infty$ by the dominated convergence theorem. Since it is clear that $Z_J$ converges to a smooth function as $J\to\infty$, it follows that $\Phi_N^G(\iota(\Gamma))$, when viewed as a function of the area of faces, is smooth.
\end{proof}

\begin{defn}
 We define an \textbf{elementary} loop as a loop in the plane with transverse and finite self-intersection.
\end{defn}

\begin{thm}[{\cite[Proposition 6.24]{Levy2011a}}]\label{thm-mm}
 Let $\Gamma$ be an elementary loop in the plane, and let $(\BB G, T, \Lambda)$ be a lasso configuration such that $\Gamma$ can be traced as $e_1e_2\gamma_1e_3e_4\gamma_2$ where $e_i$'s are directed edges and $\gamma_i$'s are some directed paths. 
 Suppose the paths $e_1e_2$ and $e_3e_4$ intersect transversally at a vertex $v$ in $\BB{G}$.
 Let $\Gamma_1$ and $\Gamma_2$ be the loops formed by changing which incoming edges of $\Gamma$ at $v$ connect to which outgoing edges. More precisely, $\Gamma_1$ (resp., $\Gamma_2$) is traced as $e_2\gamma_1e_3$ (resp., $e_4\gamma_2e_1$). (See Figure~\ref{fig-mm-loops}.) Finally, let $\abs{F_1},\ldots,\abs{F_4}$ denote the areas of the four regions of $\BB{R}^2 \backslash \Gamma$ and adjacent to $v$, in cyclic order around $v$. Then
\begin{equation} \label{eqn-mm}
	\left(\frac{\partial}{\partial \abs{F_1}} - \frac{\partial}{\partial \abs{F_2}} + \frac{\partial}{\partial \abs{F_3}} - \frac{\partial}{\partial \abs{F_4}} \right) \Phi_N^G(\Gamma) =
	\begin{cases}
		\Phi_N^G(\Gamma_1, \Gamma_2)  &\quad\text{if}\quad G=\UN\\
		\left(1+\frac{1}{N^2}\right)\Phi_N^G(\Gamma_1, \Gamma_2) &\quad\text{if}\quad G=\SUN,\\
		\Phi_N^G(\Gamma_1, \Gamma_2) + \frac{1}{N}\Phi_N^G(\widehat{\Gamma}) &\quad\text{if}\quad G=\SON,\\
		\Phi_N^G(\Gamma_1, \Gamma_2) - \frac{1}{N}\Phi_N^G(\widehat{\Gamma}) &\quad\text{if}\quad G=\SphN.
	\end{cases}
\end{equation}
\end{thm}

\begin{figure}[ht!]
    \centering
    \includegraphics[width=0.4\textwidth]{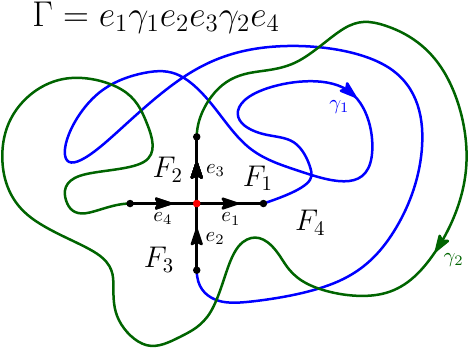}\qquad
    \includegraphics[width=0.4\textwidth]{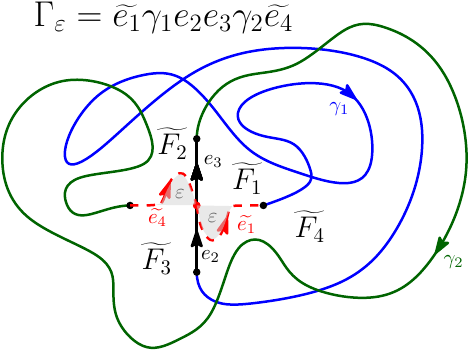} \\
	\includegraphics[width=0.4\textwidth]{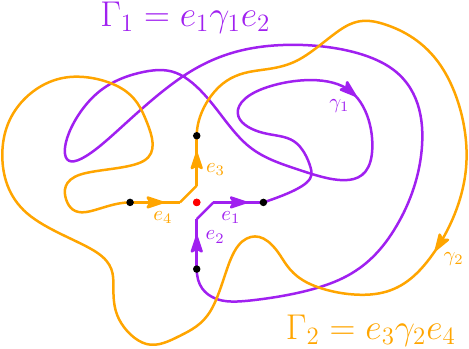}
	\qquad 
	\includegraphics[width=0.4\textwidth]{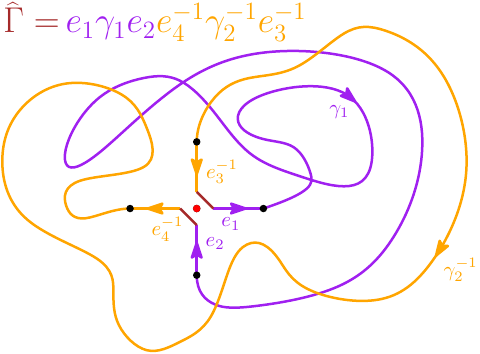}
	\caption{An illustration of loops appearing in the statement and the proof of Theorem~\ref{thm-mm}.
	}\label{fig-mm-loops}
\end{figure}

\begin{proof}
	We subdivide each edge $e_i$ by adding four new vertices and add two new edges among those vertices to construct the new graph $\BB G'$ such that
	 
	\begin{itemize}
		\item Each new vertex $v_i$ is on $e_i$, subdividing $e_i$ into $e_i'$ and $e_i''$ so that $\Gamma = e_1\gamma_1e_2e_3\gamma_2e_4$ in $\BB G$ can be traced as $e_1'e_1''\gamma_1e_2''e_2'e_3'e_3''\gamma_2e_4''e_4$ in $\BB G'$;
		\item Add two (directed) edges $d_4=\overrightarrow{v_2v_1}, d_2=\overrightarrow{v_4v_3}$ so that the edge $d_2$ (resp.~$d_4$) is contained in the face $F_2$ (resp.~$F_4$) of $\BB G$ that has $e_1, e_2$ (resp.~$e_3,e_4$) as its boundary;
		\item Each new edge subdivide the face $F_i$ into $F_i'$ and $F_i''$ for $i=2,4$ where $F_i'$ has the root $v$ as its vertex.
	\end{itemize}
 See Figure~\ref{fig-mm-G-T} (Left) for the illustration.
 Choose some spanning tree $T$ of $\BB G'$ which contains 8 edges $d_2, d_4, e_2', e_4', e_1'', e_2'', e_3'', e_4''$. The edges of $T$ including these edges are drawn with red shades in Figure~\ref{fig-mm-G-T}.
 \begin{figure}[ht]
 \centering
 \begin{tabular}{cc}
	\includegraphics[width=0.5\linewidth]{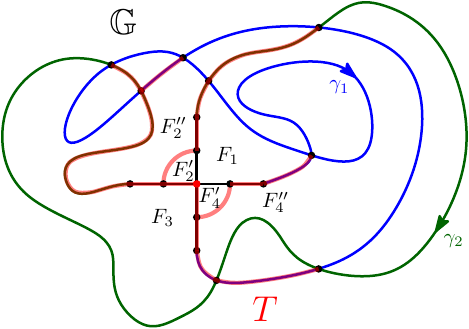}
	&
	\includegraphics[width=0.4\linewidth]{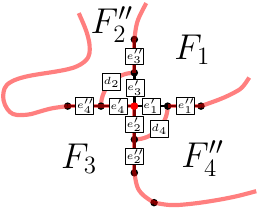}
	\end{tabular}
 \caption{An example of $(\mathbb{G},T)$ on which the loop $\Gamma$ is embedded. In this figure, $T$ is represented as the red spanning tree. The right figure zooms in near the origin.}
 \label{fig-mm-G-T}
 \end{figure}
 Let $(\Gamma, \BB G', T, \Lambda)$ be a loop-lasso configuration with $\Lambda = \{\lambda_\ell\}_{\ell=1}^L$.
 Without loss of generality, we assume $\lambda_1 = \lambda_{F_2'}, \lambda_2 = \lambda_{F_4'}$. Also let $\Gamma_i$ be the lasso representation of $\prod_{e\in \gamma_i\setminus T}\beta_{e}$ for $i=1,2$.

  From the choice of spanning tree, we have $\beta_{e_1'} = \lambda_2$, and $\beta_{e_3'} = \lambda_1^{-1}$ in terms of lassos, as seen in Figure~\ref{fig-mm-G-T}. Note that $\Lambda(\Gamma) := \beta_{e_1'}\Lambda_1\beta_{e_3'}\Lambda_2 = \lambda_2\Lambda_1\lambda_1^{-1}\Lambda_2$ is a lasso representation of $\Gamma$.
 We perturb this loop around $v$ to define a new loop $\Gamma_{\ep}$ whose lasso representation is
 $\Lambda(\Gamma_{\ep}) := \Lambda_1\beta_{e_3'}\Lambda_2\beta_{e_3'}^{-1} = \Lambda_1\lambda_1^{-1}\Lambda_2\lambda_1$ as illustrated in Figure~\ref{fig-mm-loops}.
 Therefore, we obtain lasso representations
 \begin{equation}\label{eq-lasso-mm}
 \Lambda(\Gamma) = \lambda_2\Lambda_1\lambda_1^{-1}\Lambda_2, \quad
 \Lambda(\Gamma_\ep) = \Lambda_1\lambda_1^{-1}\Lambda_2\lambda_1 \equiv \lambda_1\Lambda_1\lambda_1^{-1}\Lambda_2.
 \end{equation}

Let $\mcl C$ and $\Cep$ be the set of matching-color pairs of $\Gamma$ and $\Gamma_\ep$, respectively. (See Definition~\ref{defn-space-match}.) Note that $\Cep$ has exactly one more pair than $\mcl C$ which comes from the pair between $\lambda_1$ and $\lambda_1^{-1}$ explicitly shown in~\eqref{eq-lasso-mm}. Denote this new pair by $(1,\wt m)\in \Cep$; in other words, the explicit $\lambda_1^{-1}$ in the lasso expressions of~\eqref{eq-lasso-mm} is located at the $\wt m$'th position. Note that if $(m,m^*)\in \mcl C$, there are three cases:
\begin{itemize}
	\item $(m,m^*)\in \Cep$ and $\ep(m)=\wt \ep(m)$, $\ep(m^*)=\wt \ep(m^*)$;
	\item $(1,m-1)\in \Cep$ and $\ep(m)=\wt \ep(m-1)=1$;
	\item $(1,m+1)\in \Cep$ and $\ep(m)=\wt \ep(m+1)=-1$.
\end{itemize}
This is because if $\Lambda_1$ or $\Lambda_2$ includes $\lambda_2$ (resp.\ $\lambda_2^{-1}$), the left (resp.\ right) lasso basis element must be $\lambda_1$ (resp.\ $\lambda_1^{-1})$. As a consequence, we have a bijection between $\mcl C$ and $\Cep\setminus\{(1,\wt m)\}$. Denote this bijection by $\iota$. Let $\Sigma'$ be the points in $\Sigma$ that are not from the pairs of $\lambda_1$ or $\lambda_2$. Also define $\Sigma^{(i)}$ as the points in $\Sigma$ that are from the pairs of $\lambda_i$ for each $i=1,2$.

 To compute the Wilson loop expectations of $\Gamma$ and $\Gamma_\ep$, let $\Sigma$ and $\Sigma_\ep$ be the Poisson point processes defined as in Lemma~\ref{lem-main-poisson-gen} for $\Sigma$ and $\Sigma_\ep$, respectively.
 Let $A = \Sigma'$ and $B = \Sigma\setminus\Sigma'$ which form a partition of $\Sigma$. Define similarly $A_\ep$ and $B_\ep$ which form a partition of $\Sigma_\ep$; then $A_\ep=A$ from the previous observation.
 
 For $E\subset \Sigma$, denote by $\Sigma(E),\Sigma_\ep(E)$ the number of points of $\Sigma, \Sigma_\ep$ in $E$, respectively. We separate the Poisson sum in Lemma~\eqref{lem-main-poisson-gen} by conditioning on $\Sigma(B)$ and $\Sigma_\ep(B_\ep)$. We also simply denote by $\sgn(\Sigma)\omega_{\Sigma}^G$ the signed weight term in Lemma~\ref{lem-main-poisson-gen} defined as signed sum of ~\eqref{eqn-exp-multiple-2}. Here, we use $\sgn(\Sigma)$ instead of $\ep(\Sigma)$ to avoid confusion with the infinitesimal variable $\ep$ used in this proof.
 \begin{itemize}
 \item Because $A=A_\ep$ on which the intensity measures of $\Sigma$ and $\Sigma_\ep$ coincide, we have
 \begin{equation}
 \Ex^{\Sigma}{\sgn(\Sigma)\omega_{\Sigma}^G | \Sigma(B)=0} = \Ex^{\Sigma_\ep}{\sgn(\Sigma)\omega_{\Sigma}^G | \Sigma_\ep(B_\ep)=0}
 \end{equation}
 Therefore,
 \[
 \Ex^{\Sigma}{\sgn(\Sigma)\omega_{\Sigma}^G \1_{\Sigma(B)=0}} = e^\ep \Ex^{\Sigma_\ep}{\sgn(\Sigma)\omega_{\Sigma}^G \1_{\Sigma_\ep(B_\ep)=0}}
 \]
 as
 \[
 \Pr{\Sigma(B)=0}=e^{-\mu(\Sigma\setminus \Sigma')} = e^{\ep}e^{\mu(\Sigma\setminus \Sigma')} = e^{\ep}\Pr{\Sigma_\ep(B_\ep)=0}.
 \]
 \item Conditioned on $\Sigma(B)=\Sigma_\ep(B_\ep)=1$, a crucial fact to use is that the bijection $\iota$ preserves $\sgn(\Sigma)\omega_{\Sigma}$.
 This is because the sequence $\Sigma$ has only one pair $\{m_1,m_2\}$ in either of $\Sigma^{(i)}$ after conditioning. The relative orders of $m_1$ and $m_2$ among the other paired lassos not from $\Sigma^{(i)}$ in $\Gamma$ remain the same with those of $\wt m_1$ and $\wt m_2$ among the other paired lassos not from $\Sigma_\ep^{(i)}$ in $\Gamma_\ep$ if $\iota(\{m_1,m_2\}) = \{\wt m_1, \wt m_2\}$.
 Therefore, given $\Sigma(B)=\Sigma_\ep(B_\ep)=1$,
 \begin{align*}
 &\Ex^{\Sigma}{\sgn(\Sigma)\omega_{\Sigma}^G | \Sigma(B)=1} - \Ex^{\Sigma_\ep}{\sgn(\Sigma)\omega_{\Sigma}^G | \Sigma_\ep(B_\ep)=1}\\ = &\Ex^{\Sigma_\ep}{\sgn(\Sigma)\omega_{\Sigma}^G \1_{\Sigma_\ep(C)=1} | \Sigma_\ep(B_\ep)=1}
 \end{align*}
 where $C:=\{\{1, \wt m\}\}$. Since $\Pr{\Sigma(B)=1}=e^{-\mu(\Sigma\setminus \Sigma')}\ep$ and $\Pr{\Sigma_\ep(B_\ep)=1}=e^{-\mu(\Sigma_\ep\setminus \Sigma'_\ep)}\ep$, we have
 \begin{align}
 &\Ex^{\Sigma}{\sgn(\Sigma)\omega_{\Sigma}^G \1_{\Sigma(B)=1} } - e^{\ep}\Ex^{\Sigma_\ep}{\sgn(\Sigma)\omega_{\Sigma}^G \1_{\Sigma_\ep(B_\ep)=1}}\\
 =\quad & e^{-\mu(\Sigma_\ep\setminus \Sigma'_\ep)}\ep \times \Ex^{\Sigma_\ep}{\sgn(\Sigma)\omega_{\Sigma}^G \1_{\Sigma_\ep(C)=1} | \Sigma_\ep(B_\ep)=1}.
 \end{align}
 \item Finally, note that $\Pr{\Sigma(B)\ge 2} = O(\ep^{2})$, and thus
 \[
 \Ex^{\Sigma}{\sgn(\Sigma)\omega_{\Sigma}^G \1_{\Sigma(B)\ge 2} } = O(\ep^{2}) \;\text{ and }\; \Ex^{\Sigma_\ep}{\sgn(\Sigma)\omega_{\Sigma}^G \1_{\Sigma_\ep(B_\ep)\ge 2}} = O(\ep^{2}).
 \]
 \end{itemize} 
 As a consequence, we have
 \begin{equation}\label{eq-phi-differnce}
 \frac{\Phi_N^G(\Gamma)-\Phi_N^G(\Gamma_\ep)}{\ep} = Z(\Lambda(\Gamma_\ep)) e^{\ep} \Ex^{\Sigma_\ep}{\sgn(\Sigma)\omega_{\Sigma}^G \1_{\Sigma_\ep(C)=1} | \Sigma_\ep(B_\ep)=1} + O(\ep),
 \end{equation}
 where $Z(\Lambda(\Gamma_\ep))$ is the exponential constant described in Lemma~\ref{lem-main-poisson-gen}.
 Now it remains to simplify the RHS of~\eqref{eq-phi-differnce} using the geometric description of $\omega_{\Sigma}$ described in Lemma~\ref{lem-other-gen}.
 Conditioned on $\Sigma_\ep(B_\ep)=1$, suppose that $\Sigma_\ep(\Sigma'_\ep)=K\ge 0$ and $\Sigma_\ep(C)=1$.


 First, assume $G=\UN$. Recall how we interpret $\omega_{\Sigma}$ in this case as sums over surfaces by gluing $2(K+1)$-gon. Once we prescribe how the pair $\{\{1, \wt m\}\}$ is glued, 
 \begin{equation}\label{eqn-mm-partial}
 \Phi_N(\Gamma)-\Phi_N(\Gamma_\ep) = \ep \Phi_N(\Lambda_1, \Lambda_2)+O(\ep^{2}).
 \end{equation}
 Note that $\Gamma_1 = \beta_{e_1'}\Lambda_1 = \lambda_2\Lambda_1$ and $\Gamma_2 = \beta_{e_3'}\Lambda_2 = \lambda_1^{-1}\Lambda_2$. As a consequence, $\Phi_N(\Gamma_1, \Gamma_2)=\Phi_N(\Lambda_1, \Lambda_2)+O(\ep)$. Combined with~\eqref{eqn-mm-partial}, it follows that
 \begin{equation}
 \frac{\Phi_N(\Gamma)-\Phi_N(\Gamma_\ep)}{\ep} = \Phi_N(\Gamma_1, \Gamma_2)+O(\ep)
 \end{equation}
 which was to be proved.

 If $G=\SON$ or $\SphN$, there is an extra way to glue the pair $(1, \wt m)$, namely in non-orientable way. This non-orientable gluing corresponds to the Wilson loop expectation of another loop, corresponding to $\widehat \Gamma$ depicted in Figure~\ref{fig-mm-loops}. This explains the extra term in the theorem statement, compared to the previous $G=\UN$ case.
 If $G=\SUN$, we also have an option to contract the pair $(1, \wt m)$ instead, which contributes another factor $1/N^2$. Therefore, we proved the theorem for all cases.
\end{proof}

\begin{remark}
 In~\cite[Proposition 6.24]{Levy2011a} the Makeenko-Migdal equation is stated for any finite collection of loops, and in particular a multiple-loop version of Wilson loops are referred to as ``Wilson skeins'' there. Note that the perturbation we give in the proof of the Makeenko-Migdal equation only depends on the two loops locally at the intersection, and thus works for any arbitrary finite collection of elementary loops. Therefore, almost the same argument above can derive the multiple-loop version of Theorem~\ref{thm-mm}.
\end{remark}


 Note that in the large $N$ limit, we have shown in Corollary~\ref{cor-equals} and~\ref{cor-factorization} that the right hand side of Theorem~\ref{thm-mm} factorizes. This implies the Makeenko-Migdal equation for master field, but the proof of Theorem~\ref{thm-mm} already gives another proof of the same conclusion without using the previous corollaries.

\begin{cor}[Makeenko-Migdal for the master field] We have
$$\left(\frac{\partial}{\partial \abs{F_1}} - \frac{\partial}{\partial \abs{F_2}} + \frac{\partial}{\partial \abs{F_3}} - \frac{\partial}{\partial \abs{F_4}} \right) \Phi^{G}(\Gamma) = \Phi^{\UN}(\Lambda_1)\Phi^{\UN}(\Lambda_2)$$
for $G=\UN, \SUN, \SON, \SphN$.
\end{cor}
\begin{proof}
 Starting from~\eqref{eq-phi-differnce} in the proof of Theorem~\ref{thm-mm}, note that as $N\to \infty$
 the non-trivial contribution comes only when the pair $(1, \wt m)\in\Cep$ is glued in the orientable way, as other ways gives a factor of $1/N$ or $1/N^2$.
 
 Given the pair $(1, \wt m)\in\Cep$ glued in the orientable way, any further gluings (both orientable or non-orientable) between two parts separated by $(1,\wt{m})$ makes the total contribution zero as $N\to \infty$ from Lemma~\ref{lem-other-gen}. In addition, the surface corresponding to each part should be the sphere in order to have a non-trivial contribution as $N\to \infty$, also by Lemma~\ref{lem-other-gen}. This proves the desired result.
\end{proof}

\subsection{Random walks on permutations}\label{sec-permutations}
Let $S_M$ be the symmetric group on $[M]$, which is the set of permutations of $[M]$. The goal of this section is to express the Wilson loop expectations as sums over paths on permutations. In other words, we provide a random walk interpretation instead of the random surface interpretation we already established.

We note that~\cite{Levy2008} gave an explicit version of the Gross-Taylor expansion in a very special case, based on the relation of the random walk on the symmetric group and another random object in irreducible representations of $G=\UN$.
In particular, L\'evy wrote Wilson loop expectations of several windings of the same loop as sums over ramified covers on the disk, or in terms of a random ramified cover.

Using the surface sum story, we give an alternative proof and a generalization of the previous random walk interpretation, which is applicable to \emph{all} loops on the plane. Also, we note that the proof in~\cite{Levy2008} utilizes the Schur-Weyl duality, while our proof does not explicitly use any representation theory.

We first define the \emph{deficit} of a path on permutations.

\begin{defn}\label{defn-deficit}
	Let $\gamma=(\gamma_1,\gamma_2,\dots, \gamma_k)$ be a path on $S_M$. The deficit $d(\gamma)$ of $\gamma$ is defined as the total number of steps $\gamma_i\to \gamma_{i+1}$ such that the number of cycles of $\gamma_{i+1}$ (when represented as the cycle decomposition) is less than that of $\gamma_{i}$ for $i=1,2,\dots, k-1$.
\end{defn}

\begin{remark}\label{rmk-genus} For $G=\UN$, the genus of the surface $H(\pi,\mathrm{I}^k)$ defined in Definition~\ref{defn-H} is exactly the same as the so-called the genus of permutation $\pi$. From the other point of view, we may glue two vertices according to the given permutation, instead of two edges. In this case, the numbers of vertices and edges are fixed, while the number of faces varies. See Figure~\ref{fig-eg-genus} for an illustration. As a consequence, if we encode the face data by the cycle decomposition of underlying permutation, the deficit change is exactly the same as the genus change.
	\begin{figure}[ht!]
		\centering
		\begin{subfigure}{.4\textwidth}
			\centering
			\includegraphics[width=\textwidth]{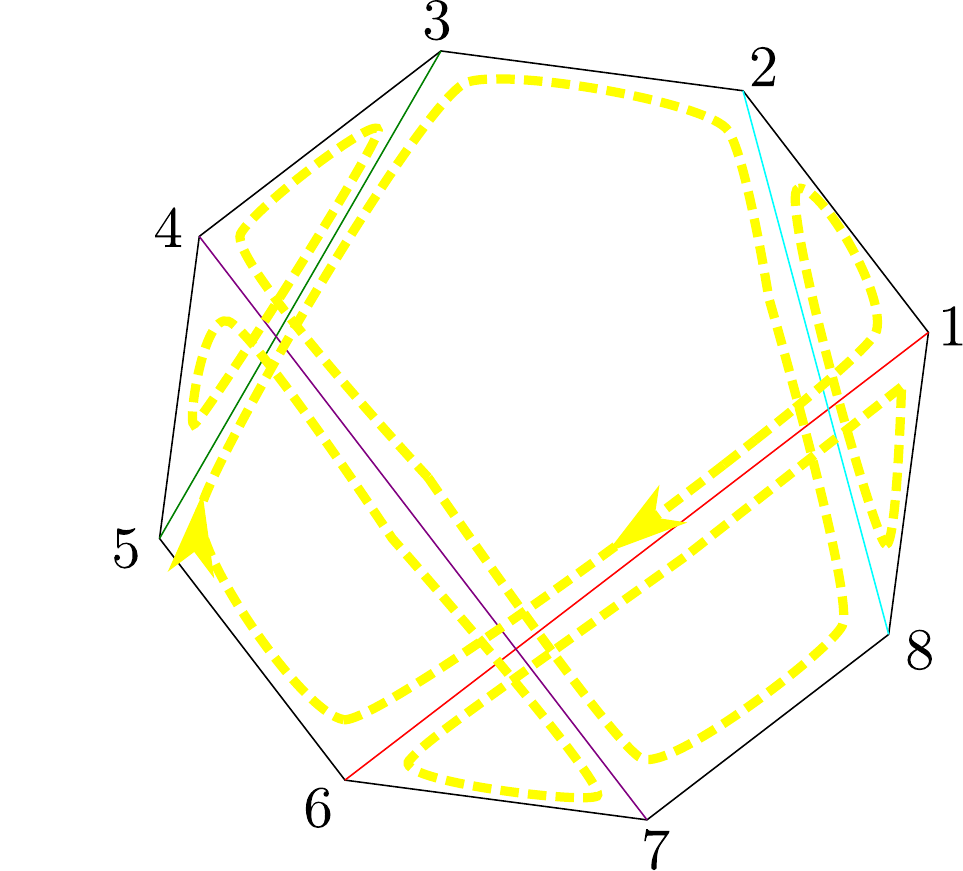}
		\end{subfigure}
		\hspace{5em}
		\begin{subfigure}{.4\textwidth}
			\centering
			\includegraphics[width=\textwidth]{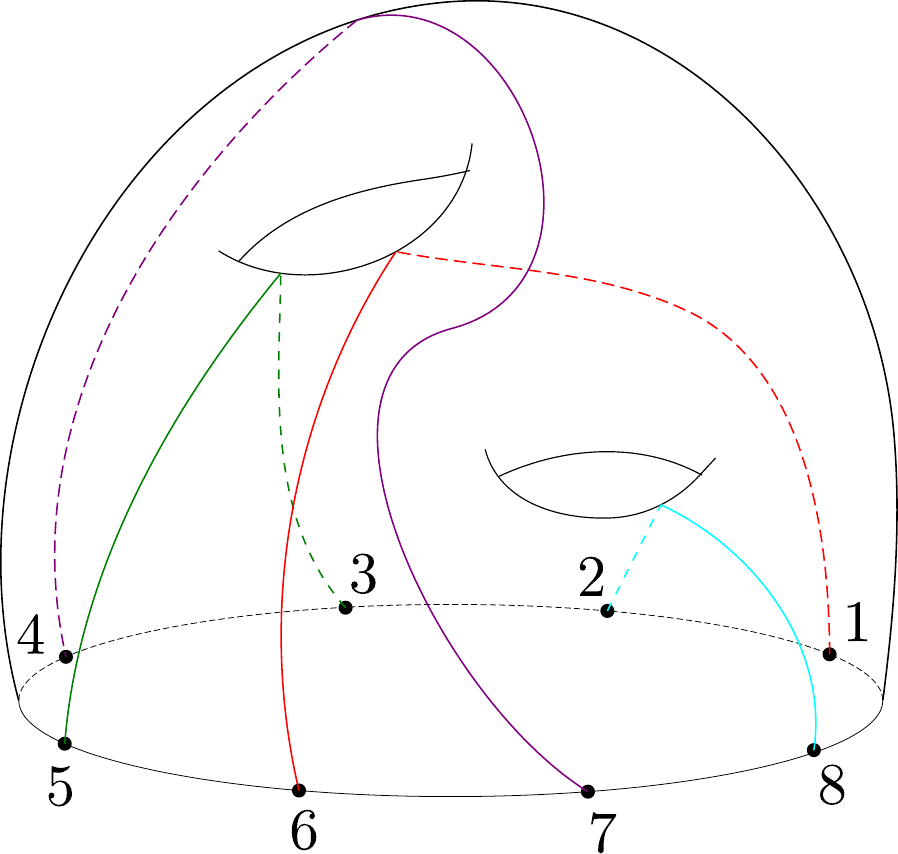}
		\end{subfigure}
		\caption{A graphical representation of the matching $\pi=(1\, 6)(2\, 8)(3\, 5)(4\,7)$ on top of $\zeta_8=(1\,2\,3\,4\,5\,6\,7\,8)$ and its embedding on a surface of genus 2. The yellow dashed line indicates how we trace a single face of the graph by following its boundary. The cyclic order determined by this boundary coincides with $\zeta_8^{-1}\pi = (1\,5\, 2\, 7\, 3\, 4\, 6\, 8)$.}\label{fig-eg-genus}
	\end{figure}
	\end{remark}

	From the above remark, we immediately recover L\'evy's formula~\cite{Levy2008}.
	\begin{cor}\label{cor-Levy-rw}
		Let $\mathfrak G$ be the Cayley graph of $S_M$ generated by transpositions, also called the transposition graph. Let $\Omega(\mathfrak G)$ be the set of paths on $\mathfrak G$ starting from $\zeta_n=(1\, 2\, \dots\, M)$. Then
		\begin{equation}
			\frac{1}{N}\ETr{U_t^\mu} = e^{-\mu t/2} \sum_{\gamma\in\Omega(\mathfrak G)}\frac{1}{\abs{\gamma}!} (-t)^{\abs{\gamma}}N^{-2d(\gamma)}
		\end{equation}
where $U_t$ is the Brownian motion on $G=\UN$ run for time $t$.

	\end{cor}
	\begin{proof}
		This follows from the bijection between the Poisson point process $\Sigma$ in Lemma~\ref{lem-main-poisson} with $\abs{\Sigma}=K$ and the path in $\Omega(\mathfrak G)$ of length $K$. In this bijection, the genus of the corresponding surface is exactly the same as the deficit of the path because of the previous remark.
	\end{proof}

	A generalized formula is obtained by pushing this idea further. Indeed, we can consider the \emph{dual picture} of our random surface. Instead of gluing edges following gluing rules we defined, we consider putting a path with two vertices in the diagram depicted in Figure~\ref{fig-expansion-diagram}. Each such path will add 3 new edges and 2 new vertices. In addition, we add open spaces on both sides of each path, so that we actually add an open \emph{strip} instead of a path. Then we can define a minimum genus surface which can complete these open spaces to make a closed manifold. (This is also called the \emph{riboon graph}.)
    See Appendix~\ref{sec-appendix} for another description of this duality.
    
    We claim that this manifold has the same genus with the corresponding surface made by gluing explained in the previous section.

	\begin{prop}\label{prop-:bijection}
		Consider a surface $H(\pi, \mathbf{r})$ defined in Definition~\ref{defn-H}. Suppose that $\pi$ was induced from $\bm{\sigma}\in \Sigma^K$. From the diagram in Figure~\ref{fig-expansion-diagram}, we put either an open orientable strip or an open non-orientable strip (that is, a M\"obius strip) based on the gluing rule. Let $\wh{H}(\pi,\mathbf{r})$ be a topological surface with the minimum genus which completes all open strips, together with the original disk, making a closed surface. This defines a bijection, and furthermore, the genus of $H$ and $\wh{H}$ are the same.
	\end{prop}
	\begin{proof}
		Given $H$, we can ``disconnect'' the glued edges by adding an open strip spanning two edges at the end. The orientability of this strip is determined by the gluing rule used to glue two edges. On the other hand, from $\wh{H}$, we can contract each strip to glue two edges at the ends. This proves the desired result.
	\end{proof}

The bijection described in Proposition~\ref{prop-:bijection} also gives a bijection between a random surface defined in Definition~\ref{defn-H} and a random walk on permutations, because we can recover $\wh{H}$ from how the face data changes while adding strips one by one. Furthermore, this implies that when a pair is added to $\bm{\sigma}$, it will move the corresponding path on permutations by one step, and thus we can define an action of $\Sigma$ acting on permutations $S_M$.

\begin{defn}\label{defn-action}
	Let $\mathfrak G(\Lambda_c^\ep)$ be the action graph on $S_M$ such that there is a directed edge $e=(x,y)$ if and only if $x$ represents the face data of some $\wh{H}$ with $\bm{\sigma}$ and gluing rule $\mathbf{r}$; and $y$ is obtained as the face data of $\wh{H^*}$ with $(\bm{\sigma},\sigma^*)$ and gluing rule $(\mathbf{r},r^*)$.
\end{defn}

This bijection works for any classical Lie group. In order to express Wilson loop expectations as expected weights of a simple random walk on this action graph, the weight on each edge can be assigned compatible with Lemma~\ref{lem-other-gen}. Note that when an added strip was orientable, an increase of the deficit should decrease the Euler characteristic by 2, and a decrease of the deficit should make the Euler characteristic remain the same. However, if an added strip was non-orientable, an increase of deficit should decrease the Euler characteristic by 3, and a decrease of deficit should decrease the Euler characteristic by 1. Therefore, there is additional $1/N$ factor for each non-orientable gluing.

To make the statement simple, we assume $G=\UN$ in the following.

\begin{thm}
	\label{thm::walk-on-permutations}
	Let $G=\UN$ and $A$ be the (weighted) adjacency matrix of $\mathfrak G(\Lambda_c^\ep)$ whose edges are weighted compatible to Lemma~\ref{lem-other-gen}. Then
	\begin{align*}
		\Ex{\tr(\Lambda_c^\ep)} &= \exp\left(-\frac{\sum_{n=1}^\nu c_n t_n}{2}\right) \sum_{\gamma \in \Omega(\mathfrak G(\Lambda_c^\ep))}\frac{1}{\abs{\gamma}!} \omega(\gamma) N^{-2d(\gamma)}\\
		&= \exp\left(-\frac{\sum_{n=1}^\nu c_n t_n}{2}\right) (1\, 1\, 1 \cdots 1)e^{A}(1\, 0\, 0\, \cdots 0)^{T}
	\end{align*}
	where $\omega(\gamma)$ denotes the product of weights on $\gamma$.
\end{thm}
\begin{proof}
	This follows immediately from the bijection between the paths on the action graph and the pairing configurations.
\end{proof}



\begin{eg}
	Let $w=\wA\wA\wA$. In this case $G(w)$ is the same as the transposition	 graph of $\fS_3$, where each edge is weighted by $-A$. 
	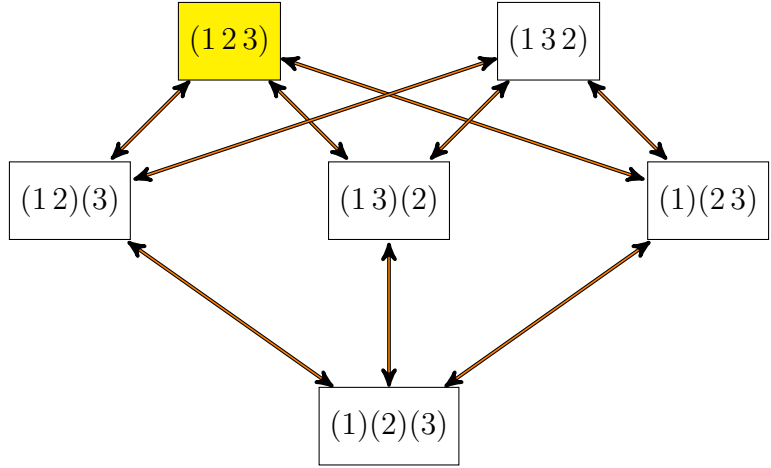
\begin{figure}[ht!]
	\centering
	\begin{tikzpicture}[>=stealth',shorten >=1pt,node distance=3cm,on grid,initial/.style ={}]
	\tikzset{every state/.style={rectangle}} 
	 \node[state] (Id)[fill=yellow] {$(1\, 2\, 3)$};
	 \node[state] (A) [below left =of Id] {$(1\,2)(3)$};
	 \node[state] (B) [below right =of Id] {$(1\,3)(2)$};
	 \node[state] (AB) [above right =of B] {$(1\,3\,2)$};
	 \node[state] (C) [right =of B, below right =of AB] {$(1)(2\, 3)$};
	 \node[state] (AC) [below =of B] {$(1)(2)(3)$};
	\tikzset{mystyle/.style={<->,double=orange}}
	\path (Id) edge [mystyle] node {} (A)
	 (Id) edge [mystyle] node {} (B)
	 (Id) edge [mystyle] node {} (C)
	 (AB) edge [mystyle] node {} (A)
	 (AB) edge [mystyle] node {} (B)
	 (AB) edge [mystyle] node {} (C)
	 (AC) edge [mystyle] node {} (A)
	 (AC) edge [mystyle] node {} (B)
	 (AC) edge [mystyle] node {} (C);
	\end{tikzpicture}
	\caption{The component of action graph starting from $\zeta_3$ for the pairing semigroup action associated with $\wA\wA\wA$}\label{fig-ag-ex1}
	\end{figure}
	The adjacency matrix $A$ is
	\begin{equation*}
		A = \begin{pmatrix}
		0	&	0	&	-A	&	-A	&	-A	&	0	\\
		0	&	0	&	-A	&	-A	&	-A	&	0	\\
		-A/N^2	&	-A/N^2	&	0	&	0	&	0	&	-A	\\
		-A/N^2	&	-A/N^2	&	0	&	0	&	0	&	-A	\\
		-A/N^2	&	-A/N^2	&	0	&	0	&	0	&	-A	\\
		0	&	0	&	-A/N^2	&	-A/N^2	&	-A/N^2	&	0	\\
		\end{pmatrix},
	\end{equation*}
	and thus the Wilson loop expectation is given by, after normalization,
	\begin{align*}
	& (1\,0\,0\,0\,0\,0)\cdot e^A\cdot (1\,1\,1\,1\,1\,1)^T \\
	=& 1+\frac{2+M^2}{3}\left(\cosh\frac{3A}{N}-1\right)-N\sinh\frac{3A}{N}.
	\end{align*}
	This computation is indeed exactly the same with what is described in~\cite{Levy2008}.
\end{eg}

Even in the case $G=\UN$, this formula is more general than Corollary~\ref{cor-Levy-rw} because we allow a pair between a lasso and its inverse. In this case, the action is not simply multiplication by transposition. If there are two copies of such a pair is added, the second pair does not affect the genus in any case. For an example, refer to Figure~\ref{fig-eg-surface-g1}.

\begin{eg}
	Let $w=\wA\wB\wAi\wB$. The component of $G(w)$ containing $\zeta_n$ has the following graph representation.
	\begin{figure}[h!]
	\centering
	\begin{tikzpicture}[>=stealth',shorten >=1pt,node distance=3cm,on grid,initial/.style ={}]
	\tikzset{every state/.style={rectangle}} 
	 \node[state] (Id)[fill=yellow] {$(1\, 2\, 3\, 4)$};
	 \node[state] (A) [below =of Id] {$(1\,2)(3\,4)$};
	 \node[state] (B) [below right =of Id, right =of A] {$(1\,4)(2\,3)$};
	 \node[state] (AB) [above right =of A, above =of B, right =of Id] {$(1\,4\,3\,2)$};
	\tikzset{mystyle/.style={->,double=orange}} 
	\tikzset{every node/.style={fill=white}} 
	\path (Id) edge [mystyle] node {$A$} (A)
	 (B) edge [mystyle] node {$A$} (AB);
	\tikzset{mystyle/.style={<->,double=orange}} 
	\path (Id) edge [mystyle, bend left=30] node {$-B$} (B)
	 (A) edge [mystyle, bend left=30] node {$-B$} (AB);
	\tikzset{every node/.style={}} 
	\path (A) edge [mystyle, loop left] node {$A$} ()
	 (AB) edge [mystyle, loop right] node {$A$} ();
	\end{tikzpicture}
	\caption{The component of action graph starting from $\zeta_4$ for the pairing semigroup action associated with $\wA\wB\wAi\wB$}\label{fig-ag-ex2}
	\end{figure}
	
	The adjacency matrix is then given by
	\begin{equation}
		A = \begin{pmatrix}
		0	&	0	&	A	&	-B	\\
		0	&	A	&	-B	&	0	\\
		0	&	-B/N^2	&	A	&	0	\\
		-B/N^2	&	A/N^2	&	0	&	0	\\
		\end{pmatrix}
	\end{equation}
	and therefore
	\begin{align}
	\Ex{\tr(w)} =& (1\,0\,0\,0)\cdot e^A\cdot (1\,1\,1\,1)^T \\
	=& e^A\cosh\frac{B}{N} - \frac{e^A}{N}\sinh\frac{B}{N} - \frac{N^2-1}{N}\sinh\frac{B}{N}.
	\end{align}
\end{eg}

In the simplest setting where Corollary~\ref{cor-Levy-rw} is applicable, there is another natural bijection between a random walk and a random ramified cover as described in~\cite{Levy2008}. Likewise, it would be interesting if our bijection leads to any non-trivial enumeration problem of an interesting class of surfaces.




\section{Explicit Wilson loop expectations for \texorpdfstring{$\UN$}{U(N)}}\label{sec-chart}
\subsection{A table of examples} \label{subsec:table}
In this section, we present a table of explicit Wilson loop expectations, which was computed using Theorem~\ref{thm::walk-on-permutations} and a Mathematica calculation to explicitly exponentiate the corresponding matrix. For simplicity, we only consider the gauge group $\UN$ here, but similar calculations could be made for other groups. For completeness, we list below all the entries from~\cite[Appendix B]{Levy2011a} (which correspond to all loops with at most three intersection points) but we leave three entries not computed.  (Computing these entries with our approach would require one to exponentiate a symbolic matrix of size up to 288 by 288; we suspect this is doable with the right implementation, but it took too long within the simple Mathematica program we wrote.) In the charts we follow the labeling convention of~\cite{Levy2011a} for lassos (meaning that we use variables $s$, $t$, $u$, $s_i$, $t_i$, etc.\ instead of the $\lambda_i$ used elsewhere in this paper). The \emph{lasso representation} is a word in an alphabet whose symbols are these variables; conversely, in the \emph{Wilson loop expectation} these variables are treated as positive real numbers (representing area). In addition, the Wilson loop expectations in this section are normalized by $\frac1N$ for single loops, compared to Definition~\ref{defn-wilson-loop}, and thus equals to $\frac1N \Phi_N(\Gamma)$. For each loop, a rooted spanning tree is shown in red, the non-tree edges are labelled in order, and the corresponding lasso word is given.

We include some remarks on extreme cases after the chart.

{

\begin{longtable}{ccc}
\toprule 
Loop & Lasso representation & Wilson loop expectation for $\UN$ \\ \otoprule 
\endhead
\loopimg{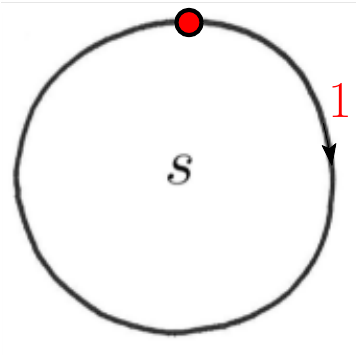} & $(s)$ & $e^{-\frac{\nabs{s}}{2}}$ \\ \midrule
\loopimg{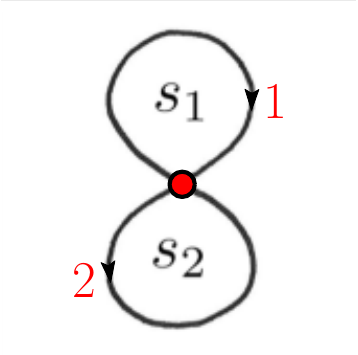} & $(s_1)(s_2)^{-1}$ & $e^{-\frac{\nabs{s_1}+\nabs{s_2}}{2}}$ \\ \midrule
\loopimg{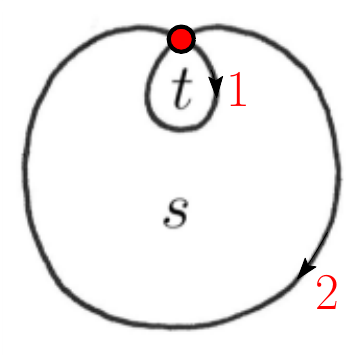} & $(t)(ts)$ & $e^{-\frac{2\nabs{t}+\nabs{s}}{2}}\biggl(\cosh\frac{\nabs{t}}{N}-N\sinh\frac{\nabs{t}}{N}\biggr)$ \\ \midrule
\loopimg{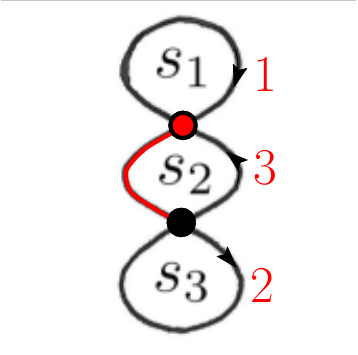} & $(s_1)(s_3)(s_2)^{-1}$ & $e^{-\frac{\nabs{s_1}+\nabs{s_2}+\nabs{s_3}}{2}}$ \\ \midrule
\loopimg{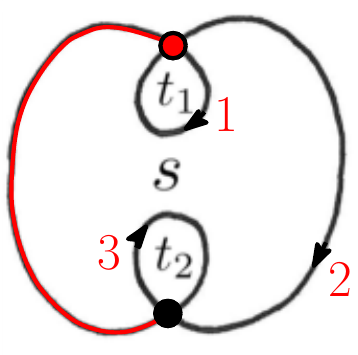} & $(t_1)(t_1t_2s)(t_2)$ & $\begin{aligned} e^{-\frac{2\nabs{t_1}+2\nabs{t_2}+\nabs{s}}{2}}\biggl(\cosh\frac{\nabs{t_1}}{N}-N\sinh\frac{\nabs{t_1}}{N}\biggr)\\ \biggl(\cosh\frac{\nabs{t_2}}{N}-N\sinh\frac{\nabs{t_2}}{N}\biggr)\end{aligned}$ \\ \midrule
\loopimg{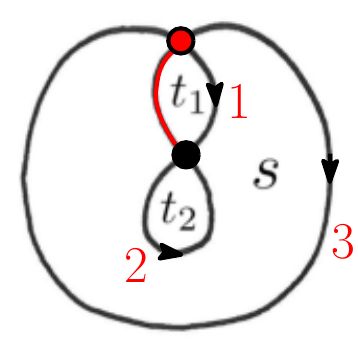} & $(t_1)(t_2)^{-1}(t_1t_2s)$ &
$\begin{aligned}
e^{-\frac{2\nabs{t_1}+2\nabs{t_2}+\nabs{s}}{2}}
\biggl( e^{\nabs{t_2}}\cosh\frac{\nabs{t_1}}{N}-\frac{e^{\nabs{t_2}}}{N}\sinh\frac{\nabs{t_1}}{N}\\-\frac{N^2-1}{N}\sinh\frac{\nabs{t_1}}{N} \biggr)\end{aligned}$ \\ \midrule
\loopimg{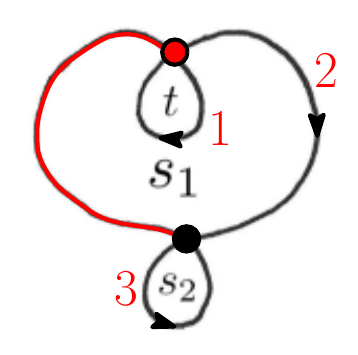} & $(t)(ts_1)(s_2)^{-1}$ & $e^{-\frac{2\nabs{t}+\nabs{s_1}+\nabs{s_2}}{2}}\biggl(\cosh\frac{\nabs{t}}{N}-N\sinh\frac{\nabs{t}}{N}\biggr)$ \\ \midrule
\loopimg{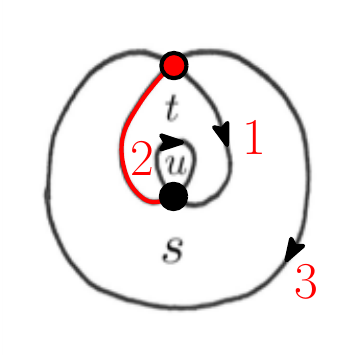} & $(ut)(u)(uts)$ & $\begin{aligned}e^{-\frac{2\nabs{t}+3\nabs{u}+\nabs{s}}{2}}\biggl(-\frac{N^2-1}{3}\cosh\frac{\nabs{t}}{N}\\+\frac{N^2+2}{3}\cosh\frac{3\nabs{u}+\nabs{t}}{N}-N\sinh\frac{3\nabs{u}+\nabs{t}}{N}\biggr)\end{aligned}$ \\ \midrule
\loopimg{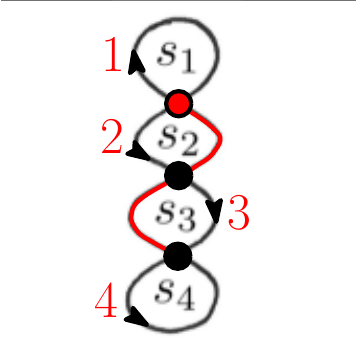} & $(s_1)(s_2)^{-1}(s_3)^{-1}(s_4)^{-1}$ & $e^{-\frac{\nabs{s_{1}}+\nabs{s_{2}}+\nabs{s_{3}}+\nabs{s_{4}}}{2}}$ \\ \midrule
\loopimg{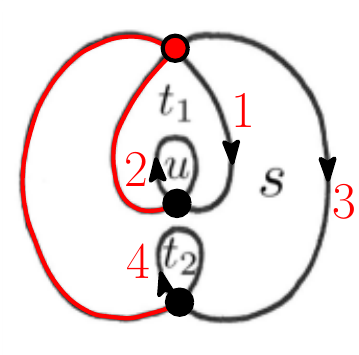} & $(ut_1)(u)(ut_1t_2s)(t_2)$ & $\begin{aligned}e^{-\frac{2\nabs{t_1}+2\nabs{t_2}+3\nabs{u}+\nabs{s}}{2}}\biggl(-\frac{N^2-1}{3}\cosh\frac{\nabs{t_1}}{N}\\+\frac{N^2+2}{3}\cosh\frac{3\nabs{u}+\nabs{t_1}}{N}
-N\sinh\frac{3\nabs{u}+\nabs{t_1}}{N}\biggr)\\ \biggl(\cosh\frac{\nabs{t_2}}{N}-N\sinh\frac{\nabs{t_2}}{N}\biggr)\end{aligned}$ \\ \midrule
\loopimg{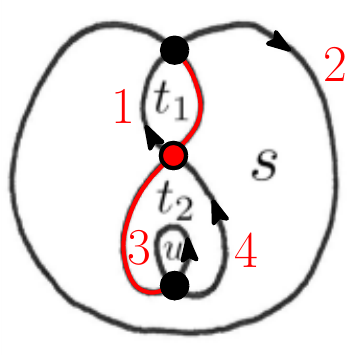} & $(t_1)(ut_2t_1s)(u)^{-1}(ut_2)^{-1}$ & $
\begin{aligned}
    e^{-\frac{2\nabs{t_1}+2\nabs{t_2}+3\nabs{u}+\nabs{s}}{2}}
    \biggl(\sinh\frac{\nabs{t_1}}{N}\\
    \left(e^{\nabs{u}}+(N^2-2)\sinh\frac{\nabs{u}}{N}-N\cosh\frac{\nabs{u}}{N}\right) \\
    +e^{\nabs{t_2}+\nabs{u}}\left(\cosh\frac{\nabs{t_1}}{N}-\frac{1}{N}\sinh\frac{\nabs{t_1}}{N}\right)\biggr)
\end{aligned}$ \\ \midrule
\loopimg{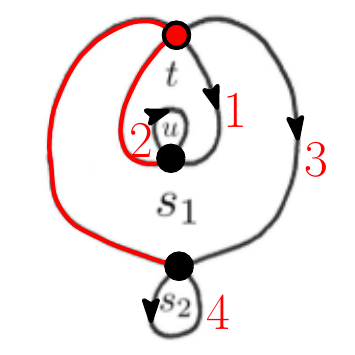} & $(ut)(u)(uts_1)(s_2)^{-1}$ & $\begin{aligned}e^{-\frac{2\nabs{t}+3\nabs{u}+\nabs{s_1}+\nabs{s_2}}{2}}\biggl(-\frac{N^2-1}{3}\cosh\frac{\nabs{t}}{N}\\+\frac{N^2+2}{3}\cosh\frac{3\nabs{u}+\nabs{t}}{N}-N\sinh\frac{3\nabs{u}+\nabs{t}}{N}\biggr)\end{aligned}$ \\ \midrule
\loopimg{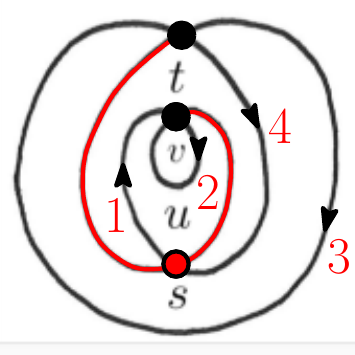} & $(vu)(v)(vut)(vuts)$ & Not computed here. \\ \midrule
\loopimg{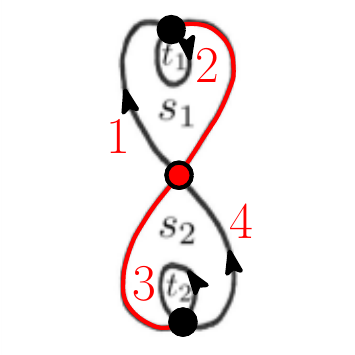} & $(t_1s_1)(t_1)(t_2)(t_2s_2)$ & $\begin{aligned}e^{-\frac{2\nabs{t_1}+2\nabs{t_2}+\nabs{s_1}+\nabs{s_2}}{2}}\biggl(\cosh\frac{\nabs{t_1}}{N}-N\sinh\frac{\nabs{t_1}}{N}\biggr)\\\biggl(\cosh\frac{\nabs{t_2}}{N}-N\sinh\frac{\nabs{t_2}}{N}\biggr)\end{aligned}$ \\ \midrule
\loopimg{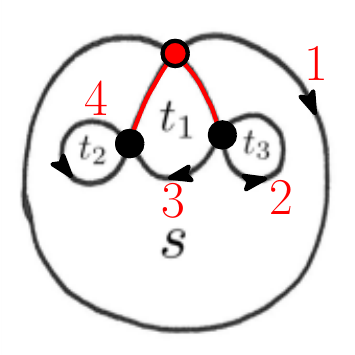} & $(t_3t_1t_2s)(t_3)^{-1}(t_1)(t_2)^{-1}$ & $\begin{aligned}
    e^{-\frac{2\nabs{t_1}+2\nabs{t_2}+2\nabs{t_3}+\nabs{s}}{2}}\\
    \biggl(\frac{N^2-1}{N^2}(e^{\nabs{t_2}}-1)\cosh\frac{\nabs{t_1}}{N}
    -\frac{N^2-1}{N}\sinh\frac{\nabs{t_1}}{N}\\
    +\frac{N^2+e^{\nabs{t_2}}-1}{N^2}e^{\nabs{t_3}}\cosh\frac{\nabs{t_1}}{N}
    -\frac{1}{N}e^{\nabs{t_2}+\nabs{t_3}}\sinh\frac{\nabs{t_1}}{N}\biggr)
\end{aligned}
$ \\ \midrule
\loopimg{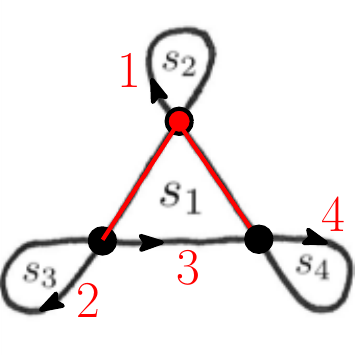} & $(s_2)(s_3)(s_1)^{-1}(s_4)$ & $e^{-\frac{\nabs{s_1}+\nabs{s_2}+\nabs{s_3}+\nabs{s_4}}{2}}$ \\ \midrule
\loopimg{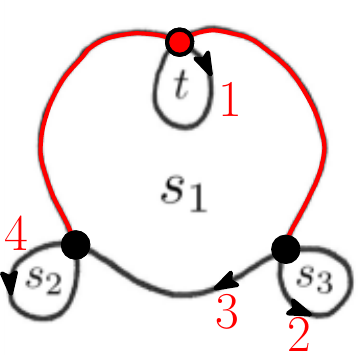} & $(t)(s_3)^{-1}(ts_1)(s_2)^{-1}$ & $e^{-\frac{2\nabs{t}+\nabs{s_1}+\nabs{s_2}+\nabs{s_3}}{2}}\biggl(\cosh\frac{\nabs{t}}{N}-N\sinh\frac{\nabs{t}}{N}\biggr)$ \\ \midrule
\loopimg{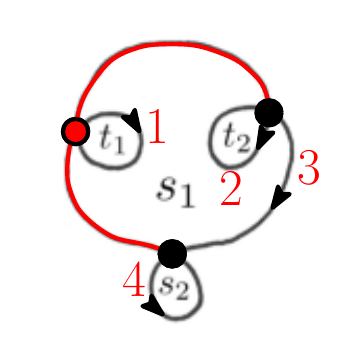} & $(t_1)(t_2)(t_2t_1s_1)(s_2)^{-1}$ & $\begin{aligned}e^{-\frac{2\nabs{t_1}+2\nabs{t_2}+\nabs{s_1}+\nabs{s_2}}{2}}\biggl(\cosh\frac{\nabs{t_1}}{N}-N\sinh\frac{\nabs{t_1}}{N}\biggr)\\ \biggl(\cosh\frac{\nabs{t_2}}{N}-N\sinh\frac{\nabs{t_2}}{N}\biggr)\end{aligned}$ \\ \midrule
\loopimg{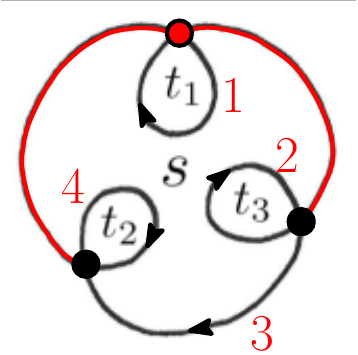} & $(t_1)(t_3)(t_3t_1t_2s)(t_2)$ & $\begin{aligned}e^{-\frac{2\nabs{t_1}+2\nabs{t_2}+2\nabs{t_3}+\nabs{s}}{2}}\biggl(\cosh\frac{\nabs{t_1}}{N}-N\sinh\frac{\nabs{t_1}}{N}\biggr)\\
\biggl(\cosh\frac{\nabs{t_2}}{N}-N\sinh\frac{\nabs{t_2}}{N}\biggr)\biggl(\cosh\frac{\nabs{t_3}}{N}-N\sinh\frac{\nabs{t_3}}{N}\biggr)\end{aligned}$ \\ \midrule
\loopimg{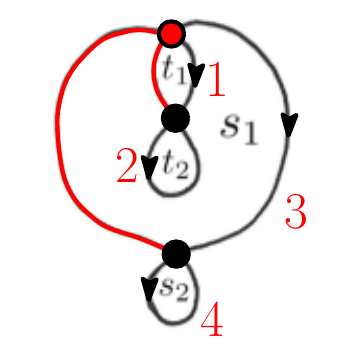} & $(t_1)(t_2)^{-1}(t_1t_2s_1)(s_2)^{-1}$ & $\begin{aligned}e^{-\frac{2\nabs{t_1}+2\nabs{t_2}+\nabs{s_1}+\nabs{s_2}}{2}}
\biggl( e^{\nabs{t_2}}\cosh\frac{\nabs{t_1}}{N}-\frac{e^{\nabs{t_2}}}{N}\sinh\frac{\nabs{t_1}}{N}\\-\frac{N^2-1}{N}\sinh\frac{\nabs{t_1}}{N} \biggr)\end{aligned}$ \\ \midrule
\loopimg{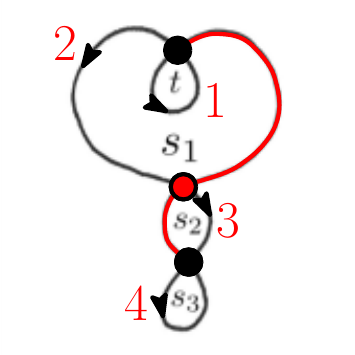} & $(ts_1)(t)(s_3)(s_2)^{-1}$ & $e^{-\frac{2\nabs{t}+\nabs{s_1}+\nabs{s_2}+\nabs{s_3}}{2}}\biggl(\cosh\frac{\nabs{t}}{N}-N\sinh\frac{\nabs{t}}{N}\biggr)$ \\ \midrule
\loopimg{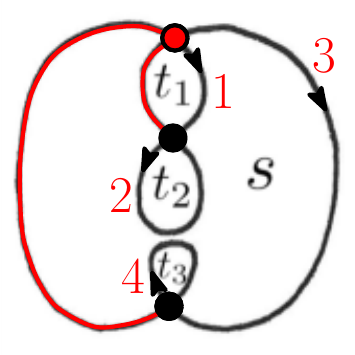} & $(t_1)(t_2)^{-1}(t_1t_2t_3s)(t_3)$ & $\begin{aligned}e^{-\frac{2\nabs{t_1}+2\nabs{t_2}+2\nabs{t_3}+\nabs{s}}{2}}
\biggl( e^{\nabs{t_2}}\cosh\frac{\nabs{t_1}}{N}-\frac{e^{\nabs{t_2}}}{N}\sinh\frac{\nabs{t_1}}{N}\\
-\frac{N^2-1}{N}\sinh\frac{\nabs{t_1}}{N} \biggr)\biggl(\cosh\frac{\nabs{t_3}}{N}-N\sinh\frac{\nabs{t_3}}{N}\biggr)\end{aligned}$ \\ \midrule
\loopimg{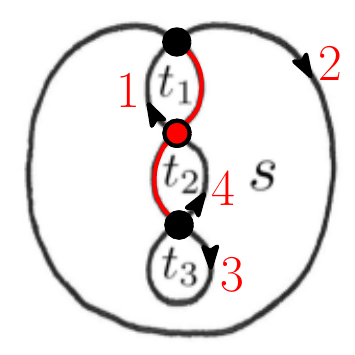} & $(t_1)(t_2t_3t_1s)(t_3)(t_2)^{-1}$ &
    $\begin{aligned}
        e^{-\frac{2\nabs{t_1}+2\nabs{t_2}+2\nabs{t_3}+\nabs{s}}{2}}\\
        \biggl(e^{\nabs{t_2}}\cosh\frac{\nabs{t_1}}{N}\cosh\frac{\nabs{t_3}}{N}-Ne^{\nabs{t_2}}\cosh\frac{\nabs{t_1}}{N}\sinh\frac{\nabs{t_3}}{N}\\
        -\frac{N^2+e^{\nabs{t_2}}-1}{N}\sinh\frac{\nabs{t_1}}{N}\cosh\frac{\nabs{t_3}}{N}\\
        +e^{\nabs{t_2}}\sinh\frac{\nabs{t_1}}{N}\sinh\frac{\nabs{t_3}}{N}\biggr)
    \end{aligned}$ \\ \midrule
\loopimg{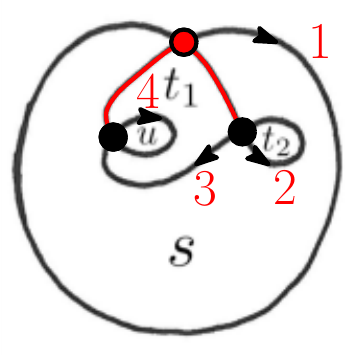} & $(t_2ut_1s)(t_2)^{-1}(ut_1)(u)$ &
    $\begin{aligned}e^{-\frac{2\nabs{t_1}+2\nabs{t_2}+3\nabs{u}+\nabs{s}}{2}}\\ 
        \biggl(-\frac{N^2-1}{3}\cosh\frac{\nabs{t_1}}{N}+\frac{N^2-1}{3N}(e^{\nabs{t_2}}-1)\sinh\frac{\nabs{t_1}}{N}\\ 
        +\left(\frac{N^2-1}{3}+e^{\nabs{t_2}}\right)\cosh\frac{\nabs{t_1}+3\nabs{u}}{N}\\
        -\frac{(N^2+2)e^{\nabs{t_2}}+2(N^2-1)}{3N}\sinh\frac{\nabs{t_1}+3\nabs{u}}{N}\biggr)\end{aligned}$ \\ \midrule
\loopimg{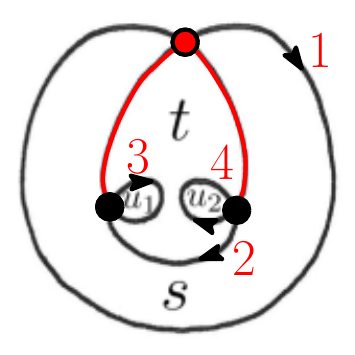} & $(u_2u_1ts)(u_2)(u_2u_1t)(u_1)$ & Not computed here.
\\ \midrule
\loopimg{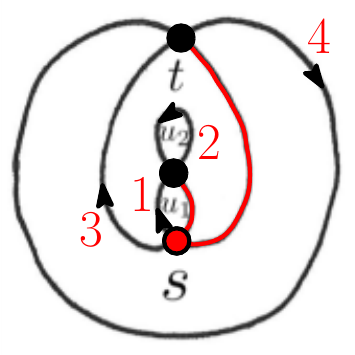} & $(u_1)(u_2)^{-1}(u_1u_2t)(u_1u_2ts)$ & Not computed here. \\ \midrule
\loopimg{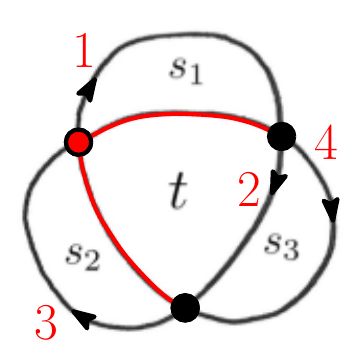} & $(s_1)(t)(s_2)(ts_3)$ & $e^{-\frac{2\nabs{t}+\nabs{s_1}+\nabs{s_2}+\nabs{s_3}}{2}}\biggl(\cosh\frac{\nabs{t}}{N}-N\sinh\frac{\nabs{t}}{N}\biggr)$ \\ \midrule
\loopimg{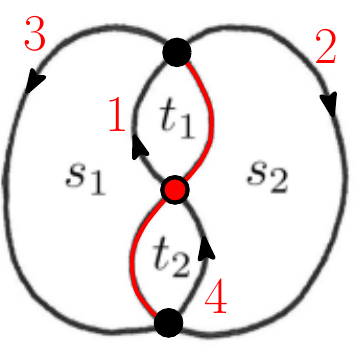} & $(t_1)(t_2s_2)(t_1s_1)^{-1}(t_2)^{-1}$ & $\begin{aligned}e^{-\frac{2\nabs{t_1}+2\nabs{t_2}+\nabs{s_1}+\nabs{s_2}}{2}}\biggl(e^{\nabs{t_1}}+e^{\nabs{t_2}}-1\\+\frac{1}{N^2}(e^{\nabs{t_1}}-1)(e^{\nabs{t_2}}-1)\biggr)\end{aligned}$ \\ \bottomrule
\end{longtable}

}

\subsection{Surface interpretation and large-area asymptotics when \texorpdfstring{$N=\infty$}{N is infinite}}

When $N \to \infty$ the limits of the expressions in the chart above are the same as those in L\'evy's ``master field'' table~\cite{Levy2011a}. Let us remark that Chatterjee's work in~\cite{Chatterjee2019a} also concerns the $N\to\infty$ limit, but on a coarse lattice approximation. In both Chatterjee's approach and our approach, taking $N\to\infty$ essentially allows one to ignore the non-simply-connected surfaces in the surface sum expansion. In Chatterjee's approach one effectively weights by the exponential of minus the total number of plaquettes---which means that small-area surfaces contribute the most to the sum.  If the $\lambda_i$ are large, one might guess that the string trajectories in~\cite{Chatterjee2019a} whose weights dominate the infinite sum would be those that approximately trace out minimal area simply connected surfaces (embedded with finitely many ramification points) as in our story.

Indeed, we highlight below a few examples of the leading terms of the $N = \infty$ (master field) equation for the $U(N)$ Wilson loop expectation as the area parameters tend to $\infty$. In each of these examples, the term in the exponent corresponds to $\frac12$ times the minimal area possible for a spanning loop, or equivalently, the minimal area ``swept out'' by a homotopic contraction of the loop to the identity. For example, in the top-right example the $t$ area is surrounded twice and must be swept out twice, hence the $2t$ in the exponent denominator. A more interesting example is the bottom right figure, where it is possible to divide the loop into two ``halves'' at the bottom vertex, and contract these halves ``around'' the $t_2$ region toward the bottom vertex (so that $t_1$ is swept out twice and $t_2$ not at all). It is also possible to divide the loop into two halves at the upper vertex and contract them ``around'' the $t_1$ region toward the upper vertex (so that $t_2$ is swept out twice and $t_1$ not at all). Taking the minimal-area option accounts for the $2 \min \{t_1, t_2 \}$ in the exponent denominator.

The polynomial factor accounts for the freedom in the ramification point placement, with each point also contributing a sign factor. In the top right example of the table below, the $-t$ accounts for the fact that we can place a ramification point anywhere in the $t$ region (which is surrounded twice). In the third-row-left example the $\Bigr(\frac32 u^2 + tu \Bigr)$ accounts for the fact that one can \emph{either} place $2$ ramification points within the $u$ region (roughly speaking: $u^2/2$ is the volume of the set of places to put two unordered points in the $u$ region, and the $3$ somehow accounts for the different ways to join the $3$ copies of $u$ that appear in the reference surface defined in Section~\ref{sec-spanning}) \emph{or} place one in the $t$ region and one in the $u$ region. (Section~\ref{sec-forest} gives a somewhat different explanation of these coefficients, as well as the lower-degree polynomial coefficients.)

\begin{center}
    
{
\begin{longtable}{cccc}
\toprule 
{\normalsize Loop} & \hspace{-.15in} {\normalsize Leading term when $N\!=\!\infty$ } & {\normalsize Loop} & \hspace{-.15in} {\normalsize Leading term when $N\!=\!\infty$} \\ \otoprule \endhead
\loopimg{1} & $e^{-\frac{s}{2}}$ &
\loopimg{3} & $-t e^{-\frac{2t+s}{2}}$ \\ \midrule
\loopimg{5} & $\begin{aligned}t_1t_2 e^{-\frac{2t_1+2t_2+s}{2}}\end{aligned}$ &
\loopimg{6} &
$\begin{aligned}
e^{-\frac{2t_1+s}{2}}
\end{aligned}$ \\ \midrule
\loopimg{8} &  \hspace{-.1in} $\begin{aligned}\Bigl(\frac{3}{2} u^2 + tu\Bigr)e^{-\frac{2t+3u+s}{2}}\end{aligned}$ &
\loopimg{19} & \hspace{-.1in} $\begin{aligned}-t_1t_2t_3e^{- \frac{2t_1+2t_2+2t_3+s}{2}}\end{aligned}$ \\ \midrule
\loopimg{27} & $-te^{-\frac{2t+s_1+s_2+s_3}{2}}$ &
\loopimg{28} & $\begin{aligned}e^{-\frac{2 \min \{t_1,t_2 \} +s_1+s_2}{2}}\end{aligned}$ \\ \bottomrule
\end{longtable}
}
\end{center}

These results are consistent with what one would expect (heuristically) based on the surface sum (vanishing string trajectory) expansion in~\cite{Chatterjee2019a}. We remark that the results in~\cite{Chatterjee2019a} also apply in dimension $d > 2$, so one might expect some version of this idea (that the minimal spanning surface area dictates the exponential decay rate in the $N=\infty$, large-area regime) to apply when $d>2$ as well.

\subsection{Surface interpretation when \texorpdfstring{$N=1$}{N is 1}}

We noted above that when $N=\infty$, the higher genus surfaces can be effectively ignored. By contrast, when $N=1$ all surfaces are treated equally regardless of their genus. Furthermore, when $N=1$ the gauge group is $U(1)$, which is abelian, and using simple scaling (multiplying a Gaussian by $k$ multiplies its variance by $k^2$) one can show that the Wilson loop expression is $$\prod_i e^{-k_i^2 \lambda_i/2}$$ where $k_i$ is the net number of times (positive minus negative) that the $\lambda_i$ region is surrounded --- i.e.\ the net number of times $\lambda_i$ occurs in the lasso representation. When verifying this in the first few examples in the chart, note that $\cosh(\frac{x}{N})-N\sinh(\frac{x}{N}) = e^{-x/N} = e^{-x}$ when $N=1$. There is an exact cancellation (removing the terms that grow as $e^{x/N}$) that only happens when $N=1$.

\subsection{Surface interpretation when \texorpdfstring{$N\in (1,\infty)$}{N is greater than 1} and areas are small}
When $N \in (1,\infty)$ is fixed and the $\lambda_i$ are small the first-order Taylor approximation for small $t$ gives $\cosh(\frac{t}{N}) \approx 1$ and $N \sinh(\frac{t}{N}) \approx t$. Because of this, in the first few chart items above, one can easily see that in the small $t$ limit the Wilson loop expectation for general $N$ agrees (to first order) with the Wilson loop expectation for $N=1$. This is because for the first order approximation, one need only consider surfaces with at most one ramification point. When we are considering only a single loop, a single ramification point, which amounts to one double arrow in the analog of e.g.\ Figure~\ref{fig-eg-surface-g1}, cannot change the genus. (Conversely, a \emph{pair} of ramification points may or may not change the genus, depending on whether the two double arrows cross each other.) Intuitively, this is because the theory is ``abelian at small scales'' in the sense that matrices that are close to the identity approximately commute with each other.
    
\subsection{Surface interpretation when \texorpdfstring{$N\in (1,\infty)$}{N is greater than 1} and areas are large} \label{subsec::largeareafixedN}
When $N \in (1,\infty)$ and the $\lambda_i$ are large, the Wilson loop expectation expressions are dominated by the terms with maximal exponential growth rate in the $\lambda_i$. (Note that $\cosh(t) \approx \sinh(t) \approx \frac{e^{t}}{2}$ when $t$ is large.) Intuitively, the ``typical'' number of ramification points should grow like the area (since the points come from a Poisson point process w.r.t.\ to a measure whose total mass is proportional to the area). Thus one should think of the ``typical'' surfaces in this regime as being surfaces with many marked points, and changing $N$ somehow affects the typical density of such points (since one expects the genus to grow roughly linearly in the number of such points).

Let us now remark on an interesting example. Consider the last example of the (finite $N$) chart in Section~\ref{subsec:table} in the degenerate case where $s_1=s_2=0$. Then the word becomes $t_1 t_2 t_1^{-1} t_2^{-1}$ and the leading term is the constant $\frac{1}{N^2}$ --- so that in particular the Wilson loop expectation does not decay to zero as the $t_i$ tend to $\infty$. This is because if $A$ and $B$ are matrices obtained by running Brownian motion on the Lie group for times $t_1$ and $t_2$ respectively, then in the $t_i \to \infty$ limit both matrices converge in law to Haar measure, and the expected trace of $A BA^{-1} B^{-1}$ is not zero in this case: expected traces of words of this form are precisely what were recently computed by Magee and Puder~\cite{Magee2019,Magee2019a} using techniques very different from ours. For example,~\cite[Table 1]{Magee2019} computed the expected trace of $A BA^{-1} B^{-1}$ as $\frac 1N$, which is compatible with our $1/N^2$ result for the normalized trace.  Several more complicated examples are worked out in~\cite{Magee2019,Magee2019a}.

\begin{figure}[ht!] \centering
	\includegraphics[width=.8\textwidth]{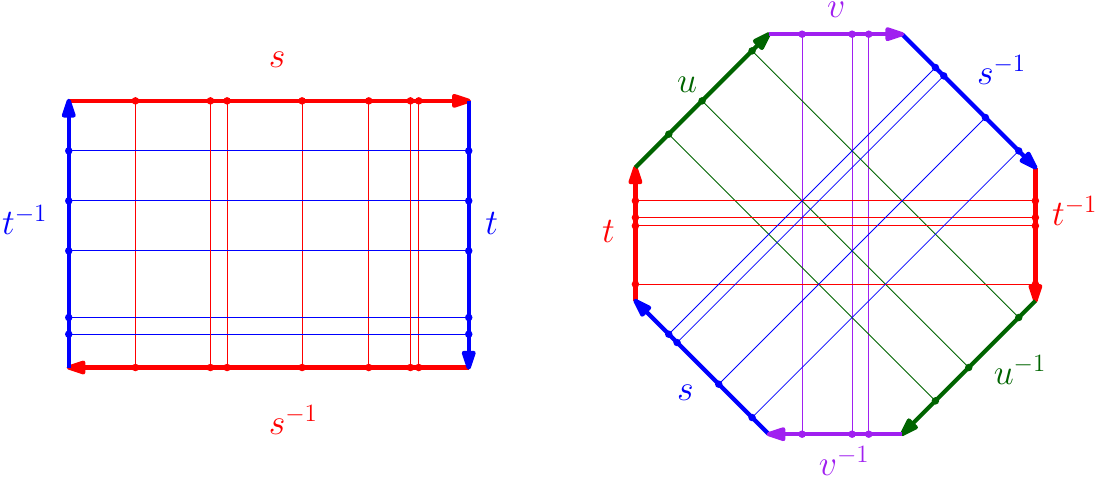}
	\caption{Two examples of Poisson point processes on the space of matched pairs (each pair shown as two dots connected by a thin line).}
 \label{fig-heavilycrossed}
\end{figure}

Let us roughly sketch a couple of ways to see this result. First, looking at the left image in Figure~\ref{fig-heavilycrossed}, it is not hard to see that the genus will be $1$ provided there is at least one vertical crossing and one horizontal crossing. Using the main result of this paper, one therefore obtains $e^{-(2t+2t)/2} e^{t_1}e^{t_2} (1/N^2) = 1/N^2$ in the limit as $t_1, t_2 \to \infty$. In fact, a similar calculation shows that for any surface where each symbol occurs at most once, the large area limit will be $1/N^{2g}$ where $g$ is the genus of the surface obtained when each edge is glued to its matched pair, see the right image in Figure~\ref{fig-heavilycrossed}.  More generally, if one wants to compare a word like $s t^2 s^{-1} t^{-2}$ with a word like $s s^{-1} t^2 t^{-2}$ (where the latter is equivalent to the identity and former is not) one has to somehow understand law of the change to the genus that occurs when one keeps the Poisson point process on matched pairs but moves around some entire edges (swapping the $t^2$ and the $s^{-1}$). We will not work this out here.

Alternatively, in the framework of~\cite{Chatterjee2019a}, the intuitive reason that the leading term is constant in these examples is non-zero is that it is possible to create a zero area surface spanning the loop (or equivalently, a ``vanishing string trajectory'') that sweeps out zero area. The reason that the constant is not $1$ is that any such surface cannot have zero genus (i.e., in the language of~\cite{Chatterjee2019a} there must be some ``division'' and ``merging'' of the strings in the string trajectory). The surface maps in~\cite{Magee2019,Magee2019a} can be interpreted as vanishing string trajectories that sweep out zero area, weighted appropriately by genus. We stress that the expected traces in~\cite{Magee2019,Magee2019a} are only non-zero when each symbol and its inverse occur the same number of times.


\section{Counting of non-crossing forests}\label{sec-forest}
In this section, we provide another combinatorial explanation of planar Yang-Mills in the large $N$ limit. We will argue that the the master field given in~\cite{Levy2011a} could also be computed via counting of non-crossing forests. The starting point is that for a simple word $w=A^n$, the counting of non-crossing forests would create a polynomial which matches perfectly with the moments of multiplicative free Brownian motion. For $N=\infty$, the Wilson loop expectation could be expressed via words of free Brownian motion using the Lasso basis. Using the moment-cumulant relation in free probability and the non-crossing rule of counting forests, we will see that if no inverses ever appears in the Wilson loop expectation then the master field will be identical to the previous `forest polynomial'. For general loops, we can also extend the counting of forests to words with inverses via splitting and show that this extension coincides with the planar master field.
 
In Section~\ref{sec-forest-discrete}, we consider the problem of counting non-crossing forests with respect to words. Namely, in each components all the vertices have the same letter, and components are not allowed to cross each other. We count the non-crossing forests over words without inverses and establish the forest polynomial. In Section~\ref{sec-forest-cont}, we extend our definition of forest polynomial to all words by splitting each letter into infinitely many pieces and reversing the order for inverses. Using equations from free probability theory, we show that this ``forest polynomial'' is identical to the master field given in~\cite{Levy2011a}. Finally in Section~\ref{sec-forest-pois}, we show that the splitting with respect to Poisson measure also produces the same result as in Section~\ref{sec-forest-cont}. Note that in this section we focus on the master field (i.e.\ $N=\infty$). We do not need to directly use the sum of surfaces interpretation from previous sections to derive the main results in this section, but the results themselves fit into that framework, and provide a systematic way to compute and understand the overall sum.

\subsection{The non-crossing forests at discrete level}\label{sec-forest-discrete}
Suppose we have a word $w=a_1\dots a_n$ without any inverses and some circle $|z|=l$. We break this circle into $n$ segments and put the letters $a_1, \dots , a_n$ in order on the segments. Also assume that the segments with the same letter have same length. For each segment we also put a vertex at its middle, and suppose $F$ is a forest over these $n$ vertices. We say $F$ is a non-crossing forest over the word $w$ if: (i) in each component, every vertex has the same letter; (ii) different components do not cross each other, while self-crossing is allowed. We let $\mathcal{F}_w$ be the collection of non-crossing forests over $w$.

\begin{defn}\label{def-forest-discrete}
	Given a word $w=a_1\dots a_n$ without inverses, we define its (non-crossing) \textbf{forest polynomial} $p(w)$ as follows. For a non-crossing forest $F=(T_1,\dots ,T_m)$ where the letter of the tree $T_k$ is $x_k$, set its forest polynomial $p(F)=\prod_{k=1}^m \frac{(-x_k)^{|T_k|-1}}{(|T_k|-1)!}$, and the forest polynomial $p(w)$ is defined by summing $p(F)$ over all $F\in \mathcal{F}_w$. Also define the normalized forest polynomial $\Phi(w) = p(w)e^{-\frac{\sum_{i=1}^na_i}{2}}$.  
\end{defn}

Here we have slightly abused the symbols, as the letter $a_i$ in the words $w$ corresponds to the generator of free groups, while in polynomials $p(F)$, $p(w)$ and $\Phi(w)$ they correspond to variables and hence commute with each other. This could be distinguished easily. 

In our discussion below, sometimes it is more convenient to consider components rather than forests. Parallel to the definition of non-crossing forests, let $\cG_w$ be the subset of $NC(w)$, the non-crossing partitions of $w$ such that in each component every vertex has the same letter. As we extend to inverses $a$ and $a^{-1}$ will be considered as the same letter. There is a natural projection $\pi$ from $\cF$ to $\cG$ defined for any word $w$, namely $\pi(F)$ is obtained by letting the components be the vertex set of each tree of $F$.

\begin{prop}\label{prop-forest-a-to-n}
	For the word $w=A^n$, its forest polynomial is given by $$p(w)=Q_n(A) = \sum_{k=0}^{n-1}\tbinom{n}{k+1}n^{k-1}\frac{(-A)^k}{k!}.$$
\end{prop}

Indeed, when $k=0$ where there are no edges, the trivial term in $p(w)$ is 1. Now for $1\le k\le n-1$, note that the number $\tbinom{n}{k+1}n^{k-1}$ is all the ways to put $k-1$ red numbers $1,\dots ,k-1$ and $k+1$ green circles over the $n$ segments such that the green circles do not overlap each other. We prove Proposition~\ref{prop-forest-a-to-n} by relating this with non-crossing forests with $k$ edges.

\begin{lem}\label{lem-parking-bij}
	 In the context of Definition~\ref{def-forest-discrete} with $w = A^n$, let $\mathcal{B}_1$ be the set of configurations of $k-1$ red numbers and $k+1$ green circles on the $n$ segments such that the green circles do not overlap. Let $\mathcal{B}_2$ be the set of labeled non-crossing forests $F = (T_0,\dots ,T_{n-k-1})$ with total number of $k$ edges where each $T_i$ is labeled by a set $C_i\subset \{0,1,\dots ,k-1\}$ with $\cup_{i=0}^{n-k-1} C_i = \{0,1,\dots ,k-1\}$ and $|C_i|=|T_i|-1$. Then $|\mathcal{B}_1|=|\mathcal{B}_2|=\tbinom{n}{k+1}n^{k-1}$.
\end{lem}

A similar result of this lemma has been proved in~\cite[Proposition 6.6]{Levy2008}, and here we use a different approach. 

\begin{proof}
Let $\mathcal{B}_3$ be the collection of $G=(G_0, \dots , G_{n-k-1})\in\cG_w$ where $G$ has $n-k$ components with each component $G_i$ labeled by a set $C_i\subset \{0,1,\dots ,k-1\}$ such that $\cup_{i=1}^m C_i = \{0,1,\dots ,k-1\}$ and $|C_i|=|G_i|-1$. Note that the natural projection $\pi$ could be extended to $\mathcal{B}_2\to \mathcal{B}_3$ as the labels are inherited. From the matrix tree theorem, there are $|G_i|^{|G_i|-2}$ possible spanning trees within each component and hence \begin{equation*}|\pi^{-1}(G)| = \prod_{i=0}^{n-k-1} |G_i|^{|G_i|-2}.\end{equation*}

We now define a mapping $f$ from $\mathcal{B}_1$ to $\mathcal{B}_2$ as follows (see Figure~\ref{fig-forest-a-to-n} as an example):
\begin{enumerate}
    \item  Take out all the vertices with no green circles/red numbers on it. 
    \item  Find a red number at position $i\in [N]$ such that there is a green circle located at $i+1$ (all in mod $n$ sense; this $i+1$ is in the sense after step 1). Draw an edge from the red number to the green circle, and then take away the red number and the green circle. If some vertex has no green circles/red numbers on it, then take it away. 
    \item  Iterate the above process on the remaining graph until there are only two green circles left. Join an edge between the two green circles and mark this edge 0.
    \item  The labels of each tree is inherited from the red numbers (or 0) it contain. We call the component with 0 by 0-component and other trees by 1-component. 
\end{enumerate}
\begin{figure}[ht!] \centering
\includegraphics[scale=0.26]{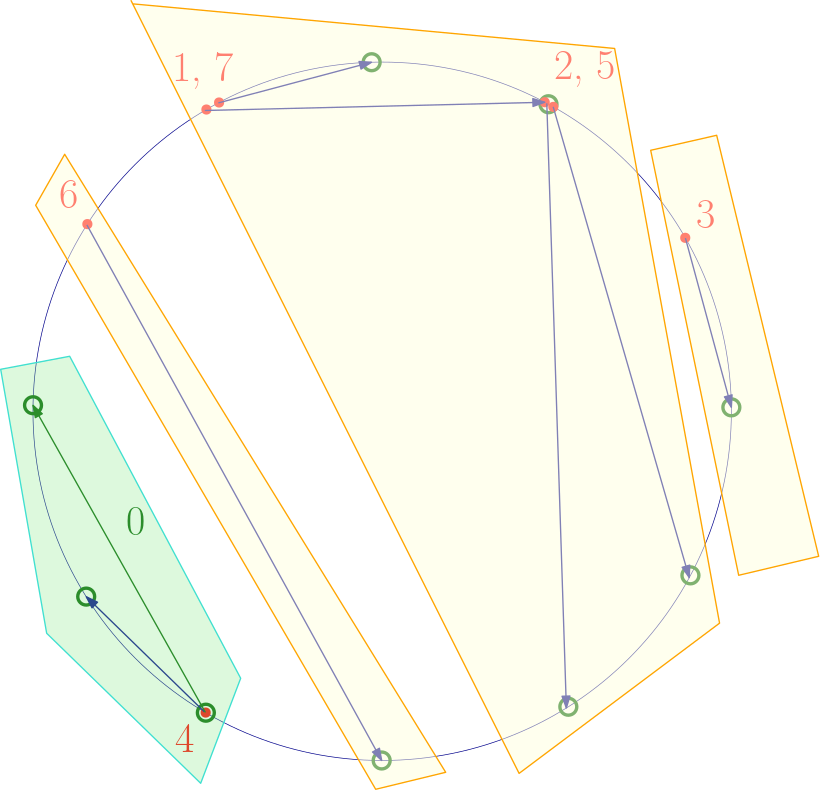}
\caption{An example of the mapping $f$ in the case $n=12$, $k=8$, i.e., 12 segments with 7 red numbers and 9 green circles.
}\label{fig-forest-a-to-n}
\end{figure}

	In general, the mapping $f$ we constructed above is not a bijection. However, if we take $g = \pi\circ f$ where $\pi$ is the canonical projection above, we will show that $|g^{-1}(G)|=|\pi^{-1}(G)|$ for each $G$ in $\mathcal{B}_3$. We begin with several properties of $f$:
	
	\begin{enumerate}
	    \item The forest we get do not cross each other and hence $f$ is well-defined. Indeed this follows from an induction argument since connecting $i$ and $i+1$ in the second step will never create any crossing components as we finish with the remaining graph.
	    \item  Each 1-component with $a_i$ points contains $a_i-1$ green circles and $a_i-1$ red numbers. The 0-component (with $a_0$ points) contains $a_0$ green circles and $a_0-2$ red numbers. The isolated points are those without green circles/red numbers on it.  
	    \item  If we start from the 0-component and go over all the vertices on both sides of this component, then each 1-component starts from some $i\in [N]$ only with red numbers. Moreover, each component can be traced via a stack: we begin with a new 1-component in the stack once we meet a point with no green circles, and we take the component out of the stack once the total number of green circles is the same as the total number of red numbers in this component.
    \end{enumerate}

    With these properties, for a labeled partition $G=(G_0, \ldots , G_{n-k-1})$ where $G_0$ is the component containing 0, we can describe $g^{-1}(G)$ as follows. First the isolated points does not contain any circles/numbers. The 0-component contains $|G_0|-2$ red numbers and $|G_0|$ green circles with no other constraints. For other components $G_i$ with its vertices $\{a_{i,1}, \ldots , a_{i,|G_i|}\}$ in the order as we explore from the 0-component, the point $a_{i,1}$ is covered only by red numbers, and all the other $|G_i|-1$ points contains a green circle. Moreover, denote $z_{i,j}$ the number of red numbers on $a_{i,j}$, then we have $z_{i,1}+\dots +z_{i,j}\ge j$ for $j=1,\dots ,|S_i|-1$. No further constraints imposed on the configuration. There are $|G_0|^{|G_0|-2}$ ways to assign the circles and numbers in the component $G_0$; in other components $G_i$, for each solution to the inequalities $z_{i,j}\ge 0; z_{i,1}+\dots +z_{i,j}\ge j, j=1, \dots , |G_i|-1$ there are $\frac{(|G_i|-1)!}{z_{i,1}!\dots (z_{i,|G_i|-1})!}$ ways to assign the values of the red points such that there are $z_{i,j}$ numbers on $a_{i,j}$. Therefore, the number of the ways to assign numbers and dots in $G_i$ is given by
\begin{equation}\label{eqn-parkingfunc}
	\sum_{\substack{z_{i,j}\ge 0; z_{i,1}+\cdots +z_{i,j}\ge j \\ z_{i,1}+\cdots +z_{i,|G_i|-1}=|G_i|-1}}\frac{(|G_i|-1)!}{z_{i,1}!\cdots (z_{i,|G_i|-1})!}=|G_i|^{|G_i|-2}.
	\end{equation} 
The equation~\eqref{eqn-parkingfunc} follows from the theory of parking functions, as the left side is precisely the number of parking functions with $s=|G_i|-1$ cars and it is known that this number is $(s+1)^{s-1}=|G_i|^{|G_i|-2}$. (For instance see~\cite{foata1974}). Therefore, it follows that
\begin{equation}\label{eqn-parkingfunc-1}
	|g^{-1}(G)| = \prod_{i=0}^{n-k-1} |G_i|^{|G_i|-2} = |\pi^{-1}(G)|
	\end{equation}
	As we sum~\eqref{eqn-parkingfunc-1} over all $G\in \mathcal{B}_3$, it follows that $|\mathcal{B}_1|=|\mathcal{B}_2|$.
\end{proof}

Given a non-crossing forest $F=(T_0,\dots ,T_{n-k-1})$, the number of ways to label it by disjoint set $C_i\subset \{0,1,\dots ,k-1\}$ with $\cup_{i=0}^{n-k-1} C_i = \{0,1,\dots ,k-1\}$ and $|C_i|=|T_i|-1$ is precisely $k!\prod_{i=0}^{n-k-1} \frac{1}{(|T_i|-1)!}$. When we sum $p(F)$ over all $F\in \cF_w$ with $k$ edges, we get the term $(-A)^k|\mathcal{B}_2|/k!$. Now dividing the two quantities in Lemma~\ref{lem-parking-bij} by $k!$, the coefficient of $(-A)^k$ is nothing but $\tbinom{n}{k+1}n^{k-1}\frac{1}{k!}$, which proves Proposition~\ref{prop-forest-a-to-n}.

Observe that, as given in~\cite{Biane1997a}, $\Phi(A^n)=e^{-\frac{nA}{2}}Q_n(A)$ is precisely the $n$-th moment of the free multiplicative Brownian motion. Using the rule of non-crossing, this further implies that the normalized forest polynomial $\Phi(w)$ could be identified with expectation of words of free Brownian motion. Recall that given a non-commutative probability space $(\mathcal{A}, \tau)$, a free Brownian motion is a family $(u_t)_{t\ge 0}$ of unitary elements such that for any $0\le t_1\le \dots \le t_n$ the increments $u_{t_2}u_{t_1}^*, \dots , u_{t_n}u_{t_{n-1}}^*$ are free with distributions $\nu_{t_2-t_1}, \dots , \nu_{t_n-t_{n-1}}$, where $\nu_t$ is a probability measure on the unit circle $\{|z|=1\}$ with 
\begin{equation}\label{free-bm-distribution}
    \int_{|z|=1}z^n\nu_t(dz) = \int_{|z|=1}z^{-n}\nu_t(dz) =  e^{-\frac{nt}{2}}\sum_{k=0}^{n-1}\tbinom{n}{k+1}n^{k-1}\frac{(-t)^k}{k!}
\end{equation}
for any $n\ge 0$. For a general word $w = a_1^{\epsilon_1}\dots a_n^{\epsilon_n}$, let $(u_{t,A})_{t\ge0}$ be mutually free multiplicative Brownian motions as above and $u(w) = u_{t_1, a_1}^{\epsilon_1}\dots  u_{t_n, a_n}^{\epsilon_n}$ with $t_1, \dots , t_n$ correspond to times $a_1, \dots , a_n$, and let $\tau(w) = \tau(u(w))$. For simplicity we let $u_{a_i}$ be the corresponding $u_{t_i, a_i}$ at time $t_i=a_i$.

\begin{prop}\label{prop-forest-free-bm-discrete}
For any words $w$ without inverses, the normalized forest polynomial $\Phi(w)$ is equal to $\tau(w)$. 
\end{prop}

To prove the proposition, we use the following result in free probability. For a reference, one may look at~\cite[Chapter 2]{Mingo2017}.

\begin{prop}\label{prop-free-cumulant}
Suppose that $(\mathcal{A}, \tau)$ is a non-commutative probability space, for $i = 1, \dots , n$, $\mathcal{A}_i\subset \mathcal{A}$ are sub-algebras and $x_i\in \mathcal{A}_i$. Then we have
\begin{equation}\label{eqn-free-cumulant}
    \tau(x_1\dots x_n) = \sum_{\sigma\in NC(n)}\kappa_\sigma[x_1, \dots , x_n]
\end{equation}
where $\kappa_\sigma[x_1, \dots , x_n]$ are the cumulants. Moreover, if there exists $i\neq j$ such that $\mathcal{A}_i$ and $\mathcal{A}_j$ are free, then the cumulant $\kappa_n(x_1, \dots , x_n)=0$.
\end{prop}

\begin{proof}[Proof of Proposition~\ref{prop-forest-free-bm-discrete}]
Given a word $w=a_1\dots a_n$ and $\sigma=(G_1, \dots , G_m)\in\cG_w$, using the matrix tree theorem, let $P(G) = \sum_{F\in \pi^{-1}(G)}p(F) =\prod_{k=1}^m \frac{(-x_k)^{|G_k|-1}|G_k|^{|G_k|-2}}{(|G_k|-1)!}$ where the letter of $G_i$ is $x_i$. Then it is clear that 
\begin{equation}\label{forest-block-sum}
    \Phi(w) = \sum_{G\in \cG_w} P(G)e^{-\frac{\sum_{i=1}^n a_i}{2}}.
\end{equation}
On the other hand, from Proposition~\ref{prop-free-cumulant} we see that 
\begin{equation}\label{freebm-block-sum}
    \tau(w) = \sum_{G\in NC(n)} \kappa_G[u_{a_1}, \dots , u_{a_n}] = \sum_{G\in \cG_w}\kappa_G[u_{a_1}, \dots , u_{a_n}].
\end{equation}
The second equation follows because when $G$ contains a component with two different letters, the freeness assumption with Proposition~\ref{prop-free-cumulant} implies that the corresponding cumulant is 0. Therefore, it suffices to show that for each $G\in \cG_w$, $e^{-\frac{\sum_{i=1}^na_i}{2}}P(G) = \kappa_G[u_{a_1}, \dots , u_{a_n}]$. In particular, we only need to show that for each single block $G_i$ consisting of $n_i$ letter $A$'s,  
\begin{equation}\label{freebm-culmulant}
    e^{-\frac{n_iA}{2}}P(G_i)=e^{-\frac{n_iA}{2}} \frac{(-A)^{n_i-1}n_i^{n_i-2}}{(n_i-1)!} = \kappa_{n_i}(u_A, \dots , u_A).
\end{equation}
This follows by induction. Indeed if $n_i=1, 2$ it is straightforward to check~\eqref{freebm-culmulant} is correct. Suppose that~\eqref{freebm-culmulant} is correct for $1, \dots , n_i-1$. Then using induction hypothesis and applying~\eqref{eqn-free-cumulant}, it follows that $\kappa_\sigma[u_A, \dots , u_A] = P(\sigma) e^{-\frac{n_iA}{2}}$ for some constant $c_\sigma$ and $s_{\sigma}\le n_i-2$ whenever $\sigma\in NC(n_i)$ (with respect to $w_{n_i}=A^{n_i}$ contains at least two blocks. On the other hand, from Proposition~\ref{prop-forest-a-to-n} we know that in this case $\tau(w_{n_i}) = \Phi(w_{n_i})$. Therefore comparing~\eqref{forest-block-sum} and~\eqref{freebm-block-sum} for $w_{n_i}$ we see that~\eqref{freebm-culmulant} holds for $n_i$. Fitting in~\eqref{freebm-culmulant} for~\eqref{forest-block-sum} and~\eqref{freebm-block-sum} once more we see that the statement is true.

\end{proof}

A quick corollary is the following splitting property:

\begin{prop}\label{prop-forest-split}
For a word $w=a_1,\ldots,a_n$ over the letters $A,B,C,\dots $, if we replace every letter $A$ appeared in $w$ by  $A_1A_2$ and get $w'$, then the corresponding forest polynomial is $(p(w'))(A_1,A_2,B,C,\dots )=(p(w))(A_1+A_2, B,C,\dots )$. The equation above is also true if we replace $p$ with its normalized version $\Phi$.
\end{prop}

\begin{proof}
Indeed it suffices to prove the proposition for $\Phi$. Since we know that $\Phi(w) = \tau(w)$, the proposition follows once we notice that for two free unitary Brownian motions $(u_{t, A_1})_{t\ge 0}$ and $(u_{t, A_2})_{t\ge 0}$ that are mutually free, then $u_{t_1, A_1}u_{t_2, A_2}$ has the same law as a free unitary Brownian motion $(u_{t, A})_{t\ge0}$ stopped at time $t = t_1+t_2$. 
\end{proof}

\subsection{The splitting and the continuum level}\label{sec-forest-cont}

In Proposition~\ref{prop-forest-split}, we have justified the splitting property for words without inverses. In this section, we are going to apply this result and split each letter into infinitely many pieces, while from this we could extend the forest polynomial in Definition~\ref{def-forest-discrete} to general words containing inverses. Indeed, we could imagine that each letter $A$ is now a concatenation of a large number of pieces $A_1, A_2, \dots , A_J$ and $A^{-1}$ could be expressed by $A_J\dots A_2A_1$. The `information' in each piece $A_i$ is $\frac{A}{J}$. Since there is cancellation between $A$ and $A^{-1}$, we suppose that connecting two segments in $A$ and $A^{-1}$ produces $+1$, while connecting two segments which are both in $A$ or $A^{-1}$ produces $-1$ as before. Then the definition of forest polynomial could be extended to general words by taking $J\to\infty$.

\begin{defn}\label{def-forest-inverse-poly}
For a word $w = a_1\dots a_n$ and a non-crossing forest $F = (T_1, \dots , T_m)\in \cF_w$, suppose the letter of $T_i$ is $x_i$. Parallel to Definition~\ref{def-forest-discrete}, set $$q(F) = \prod_{k=1}^m(-1)^{s(T_k)}\frac{x_k^{|T_k|-1}}{(|T_k|-1)!}$$ where $s(T_k)$ is the number of edges in $T_k$ connecting two $x_k$'s or two $x_k^{-1}$'s and let $$q(w) = \sum_{F\in \cF_w}q(F).$$ Suppose $w_J$ is the word obtained from $w$ by replacing each $A$, \dots  with $A_1,\dots A_J, \ldots$  and each $A^{-1}$, \dots  with $A_J^{-1},\dots A_1^{-1}, \ldots$. Then the \emph{forest polynomial} is defined by 
\begin{equation}\label{eqn-def-forest-inverse-poly}
    (p(w))(A, B, \dots ) = \lim_{J\to\infty} (q(w_J))(\frac{A}{J}, \dots , \frac{A}{J}; \frac{B}{J}, \dots , \frac{B}{J}; \dots ).
\end{equation}
Finally, the normalized polynomial $\Phi(w)$ is given by $\Phi(w) = e^{-\frac{\sum_{i=1}^na_i}{2}}p(w)$ where in the sum $\sum_{i=1}^na_i$ $A^{-1}$ is identified with $A$. 
\end{defn}

Indeed, it is straightforward from Proposition~\ref{prop-forest-split} that for word $w$ without inverses, the two definitions of forest polynomial agree. Before proving the main result of this section, let us compute some simple examples. 

\begin{enumerate}
    \item For the word $w = AA^{-1}$ with $w_J = A_1\dots A_JA_J^{-1}\dots A_1^{-1}$, we simply choose which edges to glue, and it follows that when we fit in $A_1=\dots =A_J=\frac{1}{J}$, $q(w_J)=\sum_{k=0}^J\tbinom{J}{k}(\frac{A}{J})^k = (1+\frac{A}{J})^J$ and therefore $$p(w)=\lim_{J\to\infty}(1+\frac{A}{J})^J=e^A.$$
    \item For the word $w = AAA^{-1}$, one can calculate that in the sum $q(w_J)$, the contribution from trees in $\cF_{w_J}$ with some component having two edges is $-\frac{A^2}{2J}(1+\frac{A}{J})^{J-1}$, while if all the edges have distinct letters, the sum is $(1+\frac{A}{J})^N$. Therefore $$p(w)=\lim_{J\to\infty}(1+\frac{A}{J})^J-\frac{A^2}{2J}(1+\frac{A}{J})^{J-1}=e^A.$$
    \item Suppose we have the word $w = stu_1u_2^{-1}tu_1u_2^{-1}u_1u_2$ as illustrated in Figure~\ref{fig-forest-continuum-ex-1} and let us compute $q(w_J)$. The letter $s$ appears only once and could be dropped. Conditioning on whether there is an edge between two $t$'s, we see that $q(w_J) = q((u_1u_2^{-1}u_1u_2^{-1}u_1u_2)_J) - tq((u_1u_2^{-1})_J)q((u_1u_2^{-1}u_1u_2)_J)$. Then we condition on how the $u_1$'s are connected. Using the previous two examples, we see that $q((u_1u_2^{-1}u_1u_2^{-1}u_1u_2)_J) = (1+\frac{u_2}{J})^J-\frac{u_2^2}{2N}(1+\frac{u_2}{J})^{J-1} - 2u_1(1+\frac{u_2}{J})^J - u_1(1-u_2) + \frac{3}{2}u_1^2$ and
    $$ q(w_N) =  (1+\frac{u_2}{J})^J-\frac{u_2^2}{2J}(1+\frac{u_2}{J})^{J-1} - 2u_1(1+\frac{u_2}{J})^J - u_1(1-u_2) + \frac{3}{2}u_1^2 - t((1+\frac{u_2}{J})^J - u_1).$$ Therefore,
    $$ p(w) = \lim_{J\to\infty} q(w_J) =  u_1(t+u_2 -1) +   e^{u_2}(1-2u_1-t) + \frac{3}{2}u_1^2.$$
\end{enumerate}

\begin{figure}
\begin{tabular}{ccc} 
\includegraphics[width=0.42\textwidth]{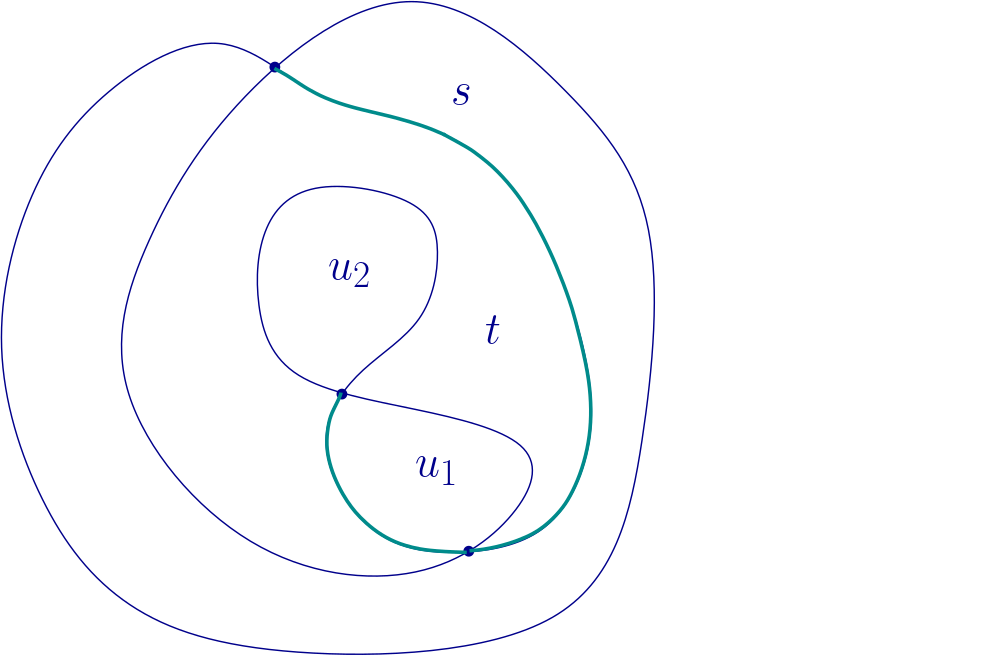}
& \qquad &
\includegraphics[width=0.39\textwidth]{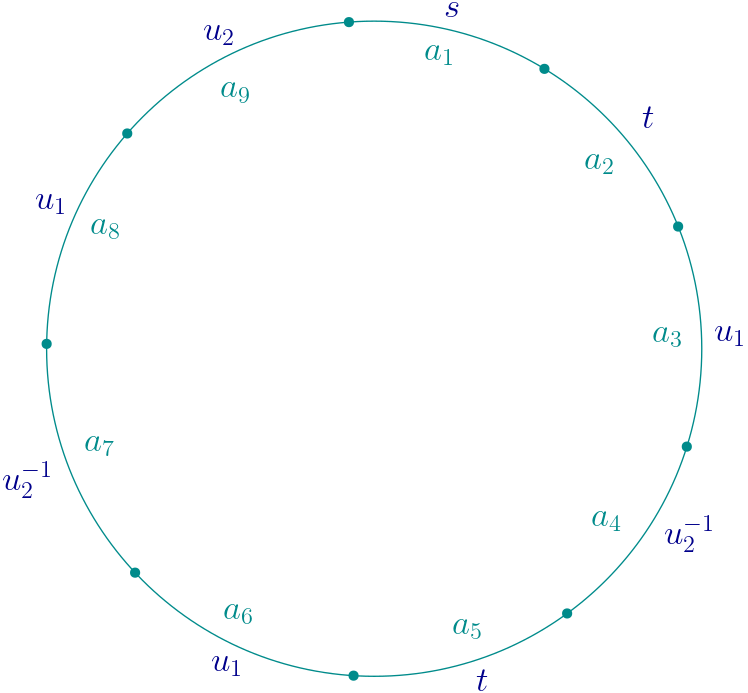}
\end{tabular}
	\caption{Suppose we have the loop illustrate on the left panel with the chosen spanning tree, its associated word is $w = stu_1u_2^{-1}tu_1u_2^{-1}u_1u_2$. If we normalize $p(w)$ by $e^{-\frac{s}{2}-t-3\frac{u_1}{2}-3\frac{u_1}{2}}$, which is $\exp(-\text{length}(w)/2)$, then we get $e^{-\frac{s}{2}-t-\frac{3}{2}u_1-\frac{3}{2}u_2}p_Q = e^{-\frac{s}{2}-t_1-\frac{3}{2}u_1-\frac{u_2}{2}}(e^{-u_2}(u_1(t+u_2-1)+\frac{3}{2}u_1^2)+1+t-2u_1)$, which matches the corresponding figure in the appendix of~\cite{Levy2011a}.}\label{fig-forest-continuum-ex-1}
\end{figure}

Now we state the main result of this section. If $a_i=A^{-1}$ then we let $u_{a_i}$ be $(u_{t, A}^{-1})_{t\ge 0}$ stopped at time $A$.
\begin{prop}\label{prop-forest-freebm}
For any word $w$, the normalized forest polynomial $\Phi(w)$ given in Definition~\ref{def-forest-inverse-poly} is well-defined and is equal to $\tau(w)$. 
\end{prop}

Our first lemma shows that in the sum $q(w_J)$, we could actually focus on the term from trees in $\cF_{w_J}$ with all the edges having distinct letters. 

\begin{lem}\label{lem-forest-rep-neg}
In the setting of Definition~\ref{def-forest-inverse-poly}, let $\cF_{w_J}^{\text{rep}}$ be the trees in $\cF_{w_J}$ such that there are at least two edges having the same letter. If we set $A_i = \frac{A}{J}$, $B_i = \frac{B}{J}$, etc., then 
\begin{equation}\label{eqn-forest-rep-neg}
    \lim_{J\to\infty}\sum_{F\in \cF_{w_J}^{\text{rep}}} |q(F)| = 0.
\end{equation}
\end{lem}

\begin{proof}
    To bound the terms from forests with replicated letters, we construct the word $\tilde{w}_J$ by identifying $A_i$'s with $A_i^{-1}$'s and rearranging all the letters in $w_J$ in lexicographical order. The order of the $a_{ij}'s$ with the same letter are kept. (For instance if $w_J = A_1\dots A_JA_J^{-1}\dots A_1^{-1}$ then $\tilde{w}_J=A_1A_1\dots A_JA_J$). We consider the following natural embedding $\sigma:\cF_{w_J}\to\cF_{\tilde{w}_J}$. Suppose that the rearrangement permutes the letter at place $i$ to $\phi(i)$, then $\phi(i)$ and $\phi(j)$ are connected in $\sigma(F)$ if and only if $i$ and $j$ are connected in $F$. Then it is straightforward to check that $\sigma(F)$ is non-crossing, and $F\in\cF_{w_J}^{\text{rep}}$ if and only if $\sigma(F)\in\cF_{\tilde{w}_J}^{\text{rep}}$, while by definition $|q(F)|=|q(\sigma(F))|$. Hence it suffices to check that  $\sum_{F\in \cF_{\tilde{w}_J}^{\text{rep}}} |q(F)|$ converges to 0.
    
    Recall that in Lemma~\ref{lem-parking-bij} we have given a summation over $F\in \cF_{A^n}$ where $F$ has $k$ edges. Indeed if we sum $|q(F)|$ instead of $p(F)$ over all $F$ with $k$ edges, we will get $\tbinom{n}{k+1}n^{k-1}\frac{|A|^k}{k!}$, and the final result will be $\hat{p}_n(A) = \sum_{k=0}^{n-1}\tbinom{n}{k+1}n^{k-1}\frac{|A|^k}{k!}$. Since $\tilde{w}_N$ has a clear product structure,  it follows that $\sum_{F\in \cF_{\tilde{w}_J}} |q(F)| = \prod_{A_i\in\frk{L}(\tilde{w}_J)}\hat{p}_{s(A_i)}(A_i)$ where $s(A_i)$ is the total number of $A_i$'s appeared in $\tilde{w}_J$ and $\frk{L}(w)$ is the set of letters appeared in $w$. 
	
	Since the degree of $q(F)$ represents how many edges with letter $A_i$ exist in $F$, we could read off that $\sum_{F\in \cF_{\tilde{w}_J}^{\text{rep}}} |q(F)|$ corresponds to the terms in polynomial $ \prod_{A_i\in\frk{L}(\tilde{w}_J)}\hat{p}_{s(A_i)}(A_i)$ where some $A_i$ has degree at least 2. As a result,  
	\begin{equation}\label{eqn-forest-rep-neg-1}
	\sum_{F\in \cF_{\tilde{w}_J}^{\text{rep}}} |q(F)| = \prod_{A_i\in\frk{L}(\tilde{w}_J)}\hat{p}_{s(A_i)}(A_i) - \prod_{A_i\in\frk{L}(\tilde{w}_J)}(1+\tbinom{s(A_i)}{2}|A_i|)
	\end{equation}
	where $\frk{L}(w)$ is the set of different letters appeared in $w$. If we fit in $A_i=\frac{A}{J}$ then 
	\begin{equation}\label{eqn-forest-rep-neg-2}
	    \sum_{F\in \cF_{\tilde{w}_J}^{\text{rep}}} |q(F)| = \prod_{A\in\frk{L}(w)} \hat{p}_{s(A)}(\frac{A}{J})^J-\prod_{A\in\frk{L}(w)}(1+\tbinom{s(A)}{2}\frac{A}{J})^J,
	\end{equation} 
	where $s(A)$ is the total number of $A$'s and $A^{-1}$'s in $w$. One can check that the two terms on the right hand side of~\eqref{eqn-forest-rep-neg-2} both converges to $\exp(\sum_{A\in\frk{L}(w)} \tbinom{s(A)}{2}A)$. Therefore the limit in~\eqref{eqn-forest-rep-neg} holds and the proof is complete.
\end{proof}

With this result that forests with duplicated letters tend to be negligible for large $J$, we now show that if a letter $A$ is next to an $A^{-1}$ then they will cancel in the following sense.

\begin{lem}\label{lem-forest-cancel}
 	For a word $w=a_1\dots a_n$ and $\tilde{w}=(AA^{-1}w)$, 
 	\begin{equation}
 	\label{eqn-forest-cancel}
 	    \lim_{J\to \infty}|q_{\tilde{w}_J}(\frac{A}{J},\dots ,\frac{A}{J};\frac{B}{J},\dots ,\frac{B}{J};\dots )-(1+\frac{A}{J})^Nq_{w_J}(\frac{A}{J},\dots ,\frac{A}{J};\frac{B}{J},\dots ,\frac{B}{J};\dots )|=0,
 	\end{equation}
 	where $q_{w_J}$ and $q_{\tilde{w}_J}$ are the polynomials from Definition~\ref{def-forest-inverse-poly}. As an immediate result, for the word $w=a_1\dots a_{k+l}$ consisting of $k$ $A$'s and $l$ $A^{-1}$'s, 
 	\begin{equation}
 	\label{eqn-forest-a-to-n-1}
 	    p(w) = Q_{k,l}(A) = e^{\min\{k,l\}A}Q_{|k-l|}(A) = e^{\min\{k,l\}A}\sum_{i=0}^{|k-l|}\tbinom{|k-l|}{i+1}|k-l|^{i-1}\frac{(-A)^i}{i!}.
 	\end{equation}
 \end{lem}

\begin{proof}
    We write $\tilde{w} = b_1b_2a_1\dots a_n$ where $b_1=A$ and $b_2=A^{-1}$. Let $\cF_J^0\subset{\cF_{\tilde{w}_J}}$ be the collection of non-crossing forests in $\tilde{w}_J$ where there exists an edge from $\{b_1,b_2\}$ to $\{a_1,\dots ,a_n\}$.  Then by definition we have $$\sum_{F\in \cF_{\tilde{w}_J}\backslash{\cF_N^0}}q(F) = (1+\frac{A}{J})^Jq_{w_J}.$$ Therefore by definition it suffices to show that $|\sum_{F\in \cF_J^0}q(F)|$ goes to 0 for large $J$.
    
    To prove this cancellation, set $\alpha_1$, (resp. $\alpha_2$) the largest $i_1$, (resp. $i_2$) such that the sub-segment $A_{i_1}$ in $b_1$ (resp. $A_{i_2}$ in $b_2$) is connected to $\{a_1,\dots ,a_n\}$. If such $i_1$ or $i_2$ does not exist then put the corresponding $\alpha$ to be 0. Define $\cF_J^1:=\{F\in \cF_0^J:\alpha_1>\alpha_2\}\backslash \cF_{\tilde{w}_J}^{\text{rep}} $ and $\cF_J^2:= \{F\in \cF_0^J:\alpha_2>\alpha_1\}\backslash \cF_{\tilde{w}_N}^{\text{rep}} $. Consider the rotation $\sigma$ from $\cF_J^1$ to $\cF_J^2$: the edge from $A_{\alpha_1}$ in $b_1$ going to $\{a_1,\dots ,a_n\}$ is diverted to start from the $A_{\alpha_1}$ in $b_2$, while all the other edges are remained. See Figure~\ref{fig-forest-cancel} for an example.
    
    \begin{figure}
		\centering
		\begin{tabular}{ccc} 
        \includegraphics[width=0.45\textwidth]{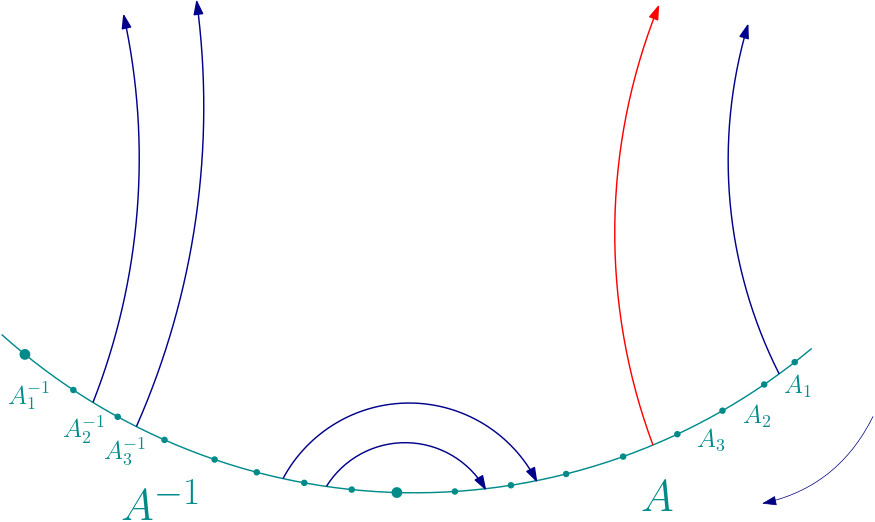}
        & \qquad &
        \includegraphics[width=0.45\textwidth]{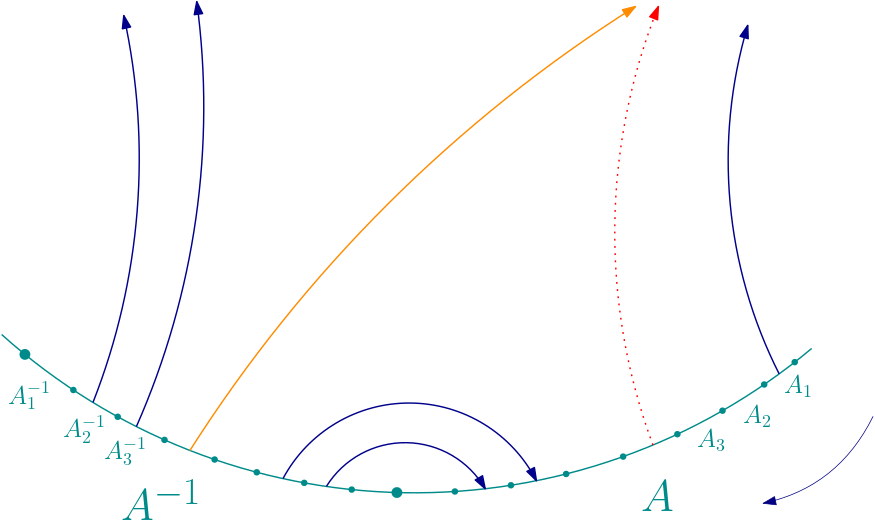}
        \end{tabular}
		\caption{An example of the rotation $\sigma$ in the proof of Lemma~\ref{lem-forest-cancel}. On the left panel $\alpha_1>\alpha_2$ (note that we are going clockwise), and if we divert the red edge from $A_{\alpha_1}$ in $b_1$ to the orange edge from $A_{\alpha_1}^{-1}$ in $b_2$ as on the right panel, we obtain a forest with $\alpha_1<\alpha_2$. Also the symbol of the orange edge is $-1$ times the symbol of the red edge. }
		\label{fig-forest-cancel}
	\end{figure}
	
	It is clear that if we rotate the edge according to $\sigma$, new crossings will not happen. By definition of $\cF_J^1$, the segment $A_{\alpha_1}^{-1}$ is isolated and $\sigma(F)$ does lie in $\cF_J^2$.  Moreover, $\sigma$ is a bijection from $\cF_J^1$ to $\cF_J^2$, and by our rule of symbols $+1/-1$ of the edges, we have $q(\sigma(F))=-q(F)$. Therefore
	\begin{equation}
	\sum_{F\in \cF_J^1}q(F)+\sum_{F\in \cF_J^2}q(F) = \sum_{F\in \cF_J^1}(q(F)+q(\sigma(F)))=0.
	\end{equation}
	The terms in $\sum_{F\in \cF_J^0\backslash (\cF_J^1\cup \cF_J^2)}p(F)$ are those $F$ with $\alpha_1=\alpha_2>0$ or having trees with multiple edges. In this case, $F$ contains two edges with the same letter, and it follows from Lemma~\ref{lem-forest-rep-neg} that $\sum_{F\in \cF_J^0\backslash (\cF_J^1\cup \cF_J^2)}|q(F)|$ goes to 0 as $J\to\infty$. Therefore $|\sum_{F\in \cF_J^0}q(F)|$ goes to 0 for large $J$ and the first part follows.
	
	For the second part, if $k=0$ or $l=0$, then the claim follows directly from Propositions~\ref{prop-forest-a-to-n} and~\ref{prop-forest-split}. Now we assume $k\ge l>0$ and suppose that the claim holds for $k\ge0$ and $l-1$. By rotation invariance, we may assume that $w = AA^{-1}a_3\dots a_{k+l} = AA^{-1}w'$. Our induction hypothesis implies that  $q_{w'_J}(\frac{A}{J},\dots ,\frac{A}{J})$ converges to $Q_{k-1,l-1}(A) =e^{\min\{k-1,l-1\}A}Q_{|k-l|}(A)$ and therefore~\eqref{eqn-forest-a-to-n-1} follows directly from~\eqref{eqn-forest-cancel}.
\end{proof}

\begin{proof}[Proof of Proposition~\ref{prop-forest-freebm}]

Let $\cG_w$ be the collection of non-crossing partitions (where $A^{-1}, B^{-1}, \dots $ are regarded as the same letter as $A, B, \dots $), $\pi:\cF\to\cG$ the natural projection, and $\pi_J:\cF_{w_J}\to\cG_w$ obtained by first mapping trees to its connected components and then gluing every copy of $A_1\dots A_J$, \dots  back to $A$, \dots  and every copy of $A_J^{-1}\dots A_1^{-1}$, \dots  back to $A^{-1}$. Note that $\pi_J$ is well-defined for any word $w$. 

    From the definition of $q(w_J)$, we have 
    \begin{equation}\label{eqn-polygon-partition}
	\begin{split}
	q_{w_J}&=\sum_{G=(G_1,\dots ,G_m)\in \mathcal{G}_w}\sum_{F=(T_1,\dots ,T_r)\in \pi_J^{-1}(G)}q(F)\\
	&=\sum_{G=(G_1,\dots ,G_m)\in \mathcal{G}_w}\sum_{F_1\in \pi_J^{-1}(G_1), \dots , F_m\in\pi_N^{-1}(G_m)}\prod_{j=1}^mq(F_j)\\
	&=\sum_{G=(G_1,\dots ,G_m)\in \mathcal{G}_w}\prod_{j=1}^m\sum_{F_j\in \pi_J^{-1}(G_j)}q(F_j).
	\end{split}
	\end{equation}
Here $G_j$ could be viewed as $|G_j|$-gon but also a word with $|G_j|$ letters, and  $F_j\in \pi_J^{-1}(G_j)$ means that after the gluing all $A_1\dots A_J$'s back to $A$'s, $F_j$ has only one component $G_j$. For $G=(G_1,\dots ,G_m)\in \mathcal{G}_w$, define
$$ P_J(G) =  \prod_{j=1}^m\sum_{F_j\in \pi_J^{-1}(G_j)}q(F_j).$$

On the other hand, from free probability theory (Proposition~\ref{prop-free-cumulant}) we know that 
\begin{equation}\label{freebm-block-sum-1}
    \tau(w) = \sum_{G\in NC(n)} \kappa_G[u_{a_1}, \dots , u_{a_n}] = \sum_{G\in \cG_w}\kappa_G[u_{a_1}, \dots , u_{a_n}].
\end{equation}

To finish the proof it suffices to show that in the case where $w=G_j$ has the form $w = A^{\epsilon_1}\dots A^{\epsilon_{k+l}}$, we have the limit 
\begin{equation}\label{eqn-forest-conditioning}
    P(w):=\lim_{J\to\infty}P_J(w)=\lim_{J\to \infty}\sum_{F'\in \pi_J^{-1}(w)}q(F') = e^{\frac{k+l}{2}A}\kappa_{k+l}(u_{A^{\epsilon_1}}, \dots , u_{A^{\epsilon_{k+l}}}).
\end{equation}
Then it shall follow that 
\begin{equation}\label{eqn-forest-comp-measure}
    P(G):=\lim_{J\to\infty}P_N(G)=\lim_{J\to\infty} \prod_{j=1}^m\sum_{F_j\in \pi_J^{-1}(G_j)}q(F_j) = e^{\frac{\sum_{i=1}^na_i}{2}}\kappa_G[u_{a_1}, \dots , u_{a_n}]
\end{equation}
 for any $G\in \cG_w$ and therefore 
\begin{equation}\label{eqn-polygon-partition-2}
    \Phi(w) = e^{-\frac{\sum_{i=1}^na_i}{2}}\sum_{G\in \cG_w} P(G) = \sum_{G\in \cG_w}\kappa_G[u_{a_1}, \dots , u_{a_n}] = \tau(w).
\end{equation}

Suppose $w$ contains $k$ $A$'s and $l$ $A^{-1}$'s. Again the claim is immediate if $k=0$ or $l=0$, as we know that if $l=0$ (resp. $k=0$) $P(w)$ is nothing but $\frac{(-A)^{k-1}}{(k-1)!}$ (resp. $\frac{(-A)^{l-1}}{(l-1)!}$).  Now we do induction on $k+l$ through  inclusion-exclusion. Using~\eqref{eqn-polygon-partition},
\begin{equation}\label{eqn-polygon-partition-1}
	\begin{split}
	q_{w_J}&=\sum_{G=(G_1,\dots ,G_m)\in \mathcal{G}_w}\prod_{j=1}^m\sum_{F_j\in \pi_J^{-1}(G_j)}q(F_j)\\
	&=\sum_{F\in \pi_J^{-1}(w)}q(F)+\sum_{\substack{G=(G_1,\dots ,G_m)\in \mathcal{G}_w \\ m\ge 2}}\prod_{j=1}^m\sum_{F_j\in \pi_J^{-1}(G_j)}q(F_j).
	\end{split}
	\end{equation}
As we do the splitting and let $A_i = \frac{A}{J}$, Lemma~\ref{lem-forest-cancel} tells us that as $J\to\infty$, $q_{w_J}$ converges to $Q_{k,l}(A)$ given in~\eqref{eqn-forest-a-to-n-1}. On the other hand, using the cancellation along with the moment result in~\cite{Biane1997}, we know that 
\begin{equation}\label{free-bm-block-2}
    \begin{split}
	Q_{k,l}(A) = e^{\frac{k+l}{2}A}\tau(w)&=e^{\frac{k+l}{2}A}\sum_{G=(G_1,\dots ,G_m)\in \mathcal{G}_w}\kappa_{G}[A^{\epsilon_1}, \dots , A^{\epsilon_{k+l}}]\\
	&=e^{\frac{k+l}{2}A}\kappa_{k+l}(u_{A^{\epsilon_1}}, \dots , u_{A^{\epsilon_{k+l}}})+e^{\frac{k+l}{2}A}\sum_{\substack{G=(G_1,\dots ,G_m)\in \mathcal{G}_w \\ m\ge 2}}\kappa_{G}[u_{A^{\epsilon_1}}, \dots , u_{A^{\epsilon_{k+l}}}].
	\end{split}
\end{equation}
Meanwhile, the induction hypothesis shows that when $G$ has $m$ components with $m\ge 2$, the product $\prod_{j=1}^m\sum_{F_j\in \pi_N^{-1}(G_j)}p(F_j)$ converges to $P(G)=e^{\frac{k+l}{2}A}\kappa_G[A^{\epsilon_1}, \dots , A^{\epsilon_{k+l}}]$. Therefore comparing~\eqref{eqn-polygon-partition-1} and~\eqref{free-bm-block-2}, we know that the limit in~\eqref{eqn-forest-conditioning} exists and is equal to the corresponding free Brownian motion cumulant. The proof is now complete.
\end{proof}

As shown in~\cite[Theorem 5.1]{Levy2011a}, the master field of a loop $\ell$ could be computed by first expressing $\ell$ via the Lasso basis and then calculate the corresponding word of free Brownian motion. And now as in Proposition~\ref{prop-forest-freebm} we know that the normalized forest polynomial $\Phi(w(\ell))$ produces the same result.

\subsection{An alternative Poisson description}\label{sec-forest-pois}
In this section, we give an alternative Poisson description, namely by running a Poisson point process over the segments to sample marked points and then consider the previous counting of non-crossing forests. We are going to show that in the case that if we impose the further condition that all the edges have different letters, then the partition function we obtain coincide with the forest polynomial $p(w)$.

For a word $w$, define $\hat{\cF}_w$ the subset of $\cF_w$ where all the edges have different letters. Now for each letter $A, B, \dots $, we sample independent Poisson random variables $n_A, n_B, \dots $ and generate the random word $\varphi(w)$ by replacing every copy $A$ with $A_1\dots A_{n_A}$, $A^{-1}$ by $A_{n_A}^{-1}\dots A_1^{-1}$, etc. For $F \in  \hat{\cF}_{\varphi(w)}$, let $\hat{q}(F) = (-1)^{s(F)}$ where $s(F)$ is the number of edges in $F$ connecting two $A_i$'s or two $A_i^{-1}$'s. 
$$\hat{q}(\varphi(w)) = \sum_{F\in \hat{\cF}_{\varphi(w)}} \hat{q}(F).$$
We define the Poisson forest polynomial
\begin{equation}\label{eqn-def-forest-poisson}
    \hat{p}(w) = \mathbb{E} \hat{q}(\varphi(w))
\end{equation} 

Let us compute two simple examples. First suppose that $w=AA^{-1}$. Given $n_A=k$, $\hat{q}(F)$ always take 1 and there are in total $2^k$ possible $F$'s. Hence $\hat{p}(w) = \sum_{k=0}^\infty 2^k \frac{A^k}{k!}e^{-A}=e^A=p(w)$. Next suppose that $Q=A^4$. Given $n_A=k$, we can check that if we glue one edge then $\hat{q}(F)$ takes $-1$ and there are $6k$ possible ways; if we glue two edges $\hat{q}(F)=1$ and the total number of $F$ is $12\tbinom{k}{2}+2\cdot 2\tbinom{{k}}{2}=16\tbinom{k}{2}$. Finally if $F$ has 3 edges, is non-crossing and all the three edges have different labels, then $\hat{q}(F)=-1$ and the total number of $F$ is $16\tbinom{k}{3}$ (See Figure~\ref{fig-forest-poisson} for more details). Therefore $$\hat{p}(w)=\sum_{k=0}^\infty (1-6k+16\tbinom{k}{2}-16\tbinom{k}{3})\frac{A^k}{k!}e^{-A}=1-6A+8A^2-\frac{8}{3}A^3=p(w).$$
\begin{figure}
\centering
	\begin{tabular}{ccccc} 
        \includegraphics[width=0.29\textwidth]{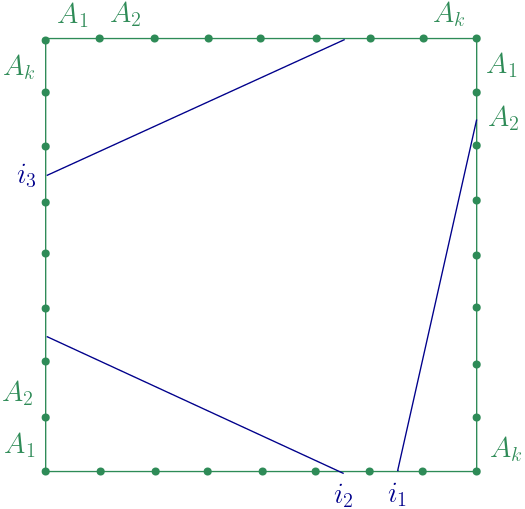}
        & \qquad &
        \includegraphics[width=0.29\textwidth]{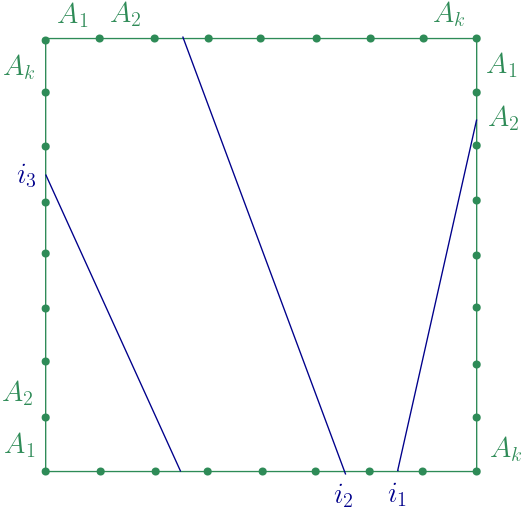} & \qquad &
        \includegraphics[width=0.29\textwidth]{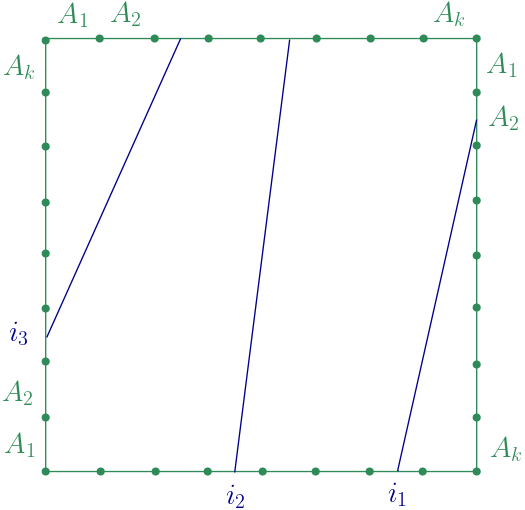}
        \end{tabular}
	\caption{Consider all the non-crossing spanning trees over $A^4$. Then there are 4 trees isomorphic to the tree on the left side, 4 to the middle, and 4 trees to the right. On the left side, suppose the edges have labels $A_{i_1}, A_{i_2}, A_{i_3}$ as depicted, then the non-crossing condition implies $i_1<i_2<i_3$, so the total sum is $\#\{1\le i_1<i_2<i_3\le k\}=\tbinom{k}{3}$; similarly for the tree on the middle side the number is also $\#\{1\le i_1<i_2<i_3\le k\}=\tbinom{k}{3}$. Finally for the tree on the right panel the non-crossing constraint is $i_1<i_2;i_3<i_2$; since we required that all the edges have different labels, $i_1\neq i_3$, the number of $F$ is $\#\{1\le i_1,i_3<i_2\le k;i_1\neq i_3\}=2\tbinom{k}{3}$. Now the total number of $F$ is $16\tbinom{k}{3}$. }
	\label{fig-forest-poisson}
\end{figure}

The main result of this section is to show that the Poisson forest polynomial is equivalent to the previous $p(w)$.

\begin{prop}\label{prop-forest-poisson}
For any words $w = a_1\dots a_n$, the Poisson forest function $\hat{p}(w)$ given by~\eqref{eqn-def-forest-poisson} is the same as the forest polynomial $p(w)$ given by Definition~\ref{def-forest-inverse-poly}. 
\end{prop}

\begin{proof}
The main idea of the proof is to trace back the whole procedure of establishing the forest polynomial for general words.

    \emph{Step 1.} We first show that $\hat{p}(w)=p(w)$ when $w$ contains no inverses. Then in this case we know that $p(w)$ is a polynomial over $\frk{L}(w)$. Suppose $$p(w) = \sum_{m_1, \dots , m_r}\gamma_{m_1, \dots , m_r}A_1^{m_1}\dots A_r^{m_r}.$$
    Then given $n_{A_i} = k_i$, as we split each $A_i$ into $A_{i,1}\dots A_{i,k_i}$ the corresponding forest polynomial for $\varphi(w)$ is nothing but $$(p(\varphi(w)))(A_{1,1}, \dots , A_{1,k_1}; \dots ) = (p(w))(A_{1,1}+\dots +A_{1,k_1};\dots ) = \sum_{m_1, \dots , m_r}\gamma_{m_1, \dots , m_r}\prod_{i=1}^r(\sum_{j=1}^{k_i}A_{i,j})^{m_i}.$$ From this we could read off the number of non-crossing forest $F$'s with all the edges having different letters. Namely, given sets $J_i\subset [k_i]$, the total number of $F\in\cF_{\varphi(w)}$ with edges having letters $\{A_{i,j}:j\in J_i\}$ is $|J_1|!\dots |J_r|!(-1)^{|J_1|+\dots +|J_r|}\gamma_{|J_1|, \dots , |J_r|}$, and when we fit in $A_{i,j}=1$, the corresponding $q(F) = (-1)^{|J_1|+\dots +|J_r|}$. Therefore it follows that
    \begin{equation}
        (\hat{q}(\varphi(w))) = \sum_{s_1, \dots , s_r}\prod_{i=1}^r(-1)^{s_i}\tbinom{k_i}{s_i}s_i!\gamma_{s_1, \dots , s_r}
    \end{equation}
    and 
    \begin{equation}
        \begin{split}
            \hat{p}(w) &= \sum_{k_1, \dots , k_r}\sum_{0\le s_i\le k_i}\prod_{i=1}^r(-1)^{s_i}\tbinom{k_i}{s_i}s_i!\gamma_{s_1, \dots , s_r}(-1)^{s_i}\frac{A_i^{k_i}}{k_i!}e^{-A_i}\\
            &= \sum_{s_1, \dots , s_r\ge 0}\sum_{k_i\ge s_i}\prod_{i=1}^r\gamma_{s_1, \dots , s_r}\frac{A_i^{k_i}}{(k_i-s_i)!}e^{-A_i}\\
            &= \sum_{s_1, \dots , s_r\ge 0}\gamma_{s_1, \dots , s_r}\prod_{i=1}^r A_i^{s_i} = p(w).
        \end{split}
    \end{equation}
    
    \emph{Step 2.} In this step we establish the rule of cancellation, namely for $\tilde{w} = AA^{-1}w$,  $\hat{p}(\tilde{w}) = \hat{p}(w)$.  Suppose $n_B, n_C, \dots $ are fixed and we are given $n_A=J$. Again we write $\tilde{w} = b_1b_2a_1\dots a_n$. Let $\hat{\cF}_{\varphi(\tilde{w})}^0\subset \hat{\cF}_{\varphi(\tilde{w})}$ be the forests such that there is an edge from segments $\{b_1,b_2\}$ to $\{a_1,\dots ,a_n\}$. Define $\alpha_1$ and $\alpha_2$ as in Lemma~\ref{lem-forest-cancel}, namely the largest subscript $i_1$ and $i_2$ such that $b_1$ or $b_2$ has an outgoing edge to $\{a_1,\dots ,a_n\}$ at place $A_{i_1}$ or $A_{i_2}$. Now let $\hat{\cF}_{\varphi(\tilde{w})}^1 = \{F\in \hat{\cF}_{\varphi(\tilde{w})}^0:\alpha_1>\alpha_2\}$ and $\hat{\cF}_{\varphi(\tilde{w})}^2 = \{F\in \hat{\cF}_{\varphi(\tilde{w})}^0:\alpha_1<\alpha_2\}$. Note that we are requiring that every edge has different letter, $\alpha_1\neq \alpha_2$. Then the rotation $\sigma$ in Lemma~\ref{lem-forest-cancel} is still well-defined (recall Figure~\ref{fig-forest-cancel}) and
    \begin{equation}\label{eqn-forest-pois-cancel-1}
        \sum_{F\in \hat{\cF}_{\varphi(\tilde{w})}^0}\hat{q}(F) = \sum_{F\in \hat{\cF}_{\varphi(\tilde{w})}^1}\hat{q}(F)+\sum_{F\in \hat{\cF}_{\varphi(\tilde{w})}^2}\hat{q}(F) = \sum_{F\in \hat{\cF}_{\varphi(\tilde{w})}^1}(\hat{q}(F)+\hat{q}(\sigma(F))) = 0.
    \end{equation}
    
    The term where $F\in \hat{\cF}_{\varphi(\tilde{w})}\backslash \hat{\cF}_{\varphi(\tilde{w})}^0$, i.e., $\{b_1, b_2\}$ is disconnected from $\{a_1, \dots , a_n\}$ is slightly more tricky. We condition on the number of edges between $b_1$ and $b_2$ (the $A$ and $A^{-1}$). If we use $k$ labels (there are $\tbinom{J}{k}$ choices) while connecting $b_1$ and $b_2$ then for the remaining part $w$ we could only use $J-k$ $A_i$'s. Then it follows that (note connecting $A_i$ and $A_i^{-1}$ gives 1 in $\hat{q}(F)$)
    \begin{equation}\label{eqn-forest-pois-cancel-2}
        \sum_{F\in \hat{\cF}_{\varphi(\tilde{w})}\backslash \hat{\cF}_{\varphi(\tilde{w})}^0}q(F) = \sum_{k=0}^J\tbinom{J}{k}\mathbb{E}  [\hat{q}(\varphi(w))|n_A = J-k, n_B, \dots ].
    \end{equation}
    Combining~\eqref{eqn-forest-pois-cancel-1} and~\eqref{eqn-forest-pois-cancel-2} we see
    \begin{equation}
        \begin{split}
            \mathbb{E}  [\hat{q}(\varphi(\tilde{w}))|n_B, \dots ]& = \sum_{J=0}^\infty \mathbb{E}  [\hat{q}(\varphi(\tilde{w}))|n_A = J, n_B, \dots ]\frac{A^J}{J!}e^{-A}\\
            & = \sum_{J=0}^\infty\sum_{k=0}^J\tbinom{J}{k}\frac{A^J}{J!}e^{-A}\mathbb{E}  [\hat{q}(\varphi(w))|n_A = J-k, n_B, \dots ]\\
            & = \sum_{k = 0}^\infty \frac{A^k}{k!}\mathbb{E}  [\hat{q}(\varphi(w))|n_B, \dots ] = e^A \mathbb{E}  [\hat{q}(\varphi(w))|n_B, \dots ].
        \end{split}
    \end{equation}
    Therefore if we take further expectation over $n_B, n_C, \dots $ the conclusion follows.
    
    From this we know that when $w$ is consisted of $A$'s and $A^{-1}$'s, then $\hat{p}(w) = p(w)$.
    
    \emph{Step 3.} The remaining words. Again we operate at the level of connected components. Given $\cJ:=(n_A:A\in \frk{L}(w))$, let $\hat{\pi}_\cJ:\hat{\cF}_{\varphi(w)}\to\cG_{w}$ be the projecting and gluing operator similar to $\pi_J$ in the previous section. For $G\in \cG_w$, write $G = (G_A:A\in \frk{L}(w))$ where $G_A$ is the collection of all the $A$-components in $G$. Suppose $|\frk{L}(w)|=r$. Then by definition
    \begin{equation}\label{eqn-forest-pois-partition}
        \begin{split}
        \hat{p}(w) &= \mathbb{E}\hat{q}(\varphi(w)) = \sum_{\cJ=(n_A:A\in \frk{L}(w))\in \mathbb{N}^{r}}\sum_{G\in \cG_w}\sum_{F \in \hat{\pi}_\cJ^{-1}(G)}\hat{q}(F)\prod_{A\in \frk{L}(w)}\frac{A^{n_A}}{n_A!}e^{-A}\\
        &= \sum_{G = (G_A:A\in \frk{L}(w))\in \cG_w}\prod_{A\in \frk{L}(w)}\sum_{n_A=0}^\infty \sum_{F_A\in \hat{\pi}_{n_A}^{-1}(G_A)}\hat{q}(F_A)\frac{A^{n_A}}{n_A!}e^{-A}.
        \end{split}
    \end{equation}
    On the other hand, we know from~\eqref{eqn-polygon-partition-2} that
     \begin{equation}\label{eqn-forest-comp-measure-2}
     p(w) = \sum_{G = (G_A:A\in \frk{L}(w))\in \cG_w}\prod_{A\in \frk{L}(w)} P(G_A)
     \end{equation}
     where $P(G_A)$ is given in~\eqref{eqn-forest-comp-measure}. Therefore it suffices to show that for every $G_A$, 
     \begin{equation}
     \label{eqn-forest-pois-equi}
         \hat{P}(G_A): =  \sum_{n_A=0}^\infty \sum_{F_A\in \hat{\pi}_{n_A}^{-1}(G_A)}\hat{q}(F_A)\frac{A^{n_A}}{n_A!}e^{-A} = P(G_A).
     \end{equation}
     
     At this level we can now assume that $w$ is consisted of $k$ $A$'s and $l$ $A^{-1}$'s, and $G = G_A = (G_1, \dots , G_m)\in \cG_w$. We do induction on $k+l$. If $k+l=1$ then both sides of~\eqref{eqn-forest-pois-equi} are 1 and there is nothing to prove. Suppose now~\eqref{eqn-forest-pois-equi} holds for $1,\dots ,k+l-1$ and corresponding $G\in \mathcal{G}_w$. Then for $(k,l)$, $G = (G_1, \dots , G_m)$ with $m\ge 2$, we write $G'=(G_2, \dots , G_m)$ and condition on the number of the edges in $G_1$. Given $n_A=J$, if there are $s$ edges in the component $G_1$, then by our rule of having distinct edge labels, the situation for the rest of $F\in\hat{\pi}_{n_A}^{-1}(G')$ could be linked to $G'$ with $n_A=J-s$. Therefore we have
     \begin{equation}\label{eqn-forest-pois-partition-1}
	\begin{split}
	&\sum_{J=0}^\infty\sum_{F\in \hat{\pi}_J^{-1}((G_1, G'))}\hat{q}(F)\frac{A^{J}}{J!}e^{-A}\\
	&=\sum_{J=0}^\infty\sum_{s=0}^J\tbinom{J}{s}\sum_{\substack{F_1\in \hat{\pi}_s^{-1}(G_1)\\ |E(F_1)|=s}}\sum_{F'\in \hat{\pi}_{J-s}^{-1}(G')}\hat{q}(F_1)\hat{q}(F')\frac{A^J}{J!}e^{-A}\\
	&=\sum_{s=0}^\infty\sum_{\substack{F_1\in \hat{\pi}_s^{-1}(G_1) \\ |E(F_1)|=s}}\sum_{J=s}^\infty\sum_{F'\in \hat{\pi}_{J-s}^{-1}(G')}\hat{q}(F_1)\hat{q}(F')\frac{A^sA^{J-s}}{s!(J-s)!}e^{-A}\\
	&=\Big(\sum_{s=0}^\infty\sum_{\substack{F_1\in \hat{\pi}_s^{-1}(G_1)\\ |E(F_1)|=s}}\hat{q}(F_1)\frac{A^s}{s!}  \Big) \Big(\sum_{J=0}^\infty\sum_{F'\in \hat{\pi}_J^{-1}(G')}\hat{q}(F')\frac{A^J}{J!}e^{-A}  \Big).
	\end{split}
	\end{equation}
	Clearly the second term on the last line of~\eqref{eqn-forest-pois-partition-1} is nothing but $\hat{P}(G')$. To deal with the first term, we note that when calculating  $\hat{P}(G_1)$, we could also condition on the total number of edges appeared in the forest. That is, we may first choose $r$ labels from the $N$ choices then view $F_1$ as an element in $\hat{\pi}_r^{-1}(G_1)$ with $r$ edges.
	\begin{equation}\label{eqn-forest-pois-partition-2}
	\begin{split}
	\hat{P}(G_1)&=\sum_{J=0}^\infty\sum_{F_1\in \hat{\pi}_N^{-1}(G_1)}\hat{q}(F_1)\frac{A^J}{J!}e^{-A}=\sum_{J=0}^\infty\sum_{r=0}^J\tbinom{J}{r}\sum_{\substack{F_1\in \hat{\pi}_r^{-1}(G_1)\\ |E(F_1)|=r}}\hat{q}(F_1)\frac{A^J}{J!}e^{-A}\\
	&=\sum_{r=0}^\infty\sum_{\substack{F_1\in \hat{\pi}_r^{-1}(G_1) \\ |E(F_1)|=r}}\hat{q}(F_1)\frac{A^r}{r!}.
	\end{split}
	\end{equation}
		Combining~\eqref{eqn-forest-pois-partition-1} and~\eqref{eqn-forest-pois-partition-2} we see that $\hat{P}(G) = \hat{P}(G_1)\hat{P}(G')$. On the other hand, by definition $P(G) = P(G_1)P(G')$. Therefore it follows from the induction hypothesis that $\hat{P}(G) = P(G)$.
		
		It remains to deal with the case where $G$ has one single component $w$. We apply once more the inclusion-exclusion argument in Proposition~\ref{prop-forest-freebm}. On one hand, from~\eqref{eqn-polygon-partition-2} and~\eqref{eqn-forest-pois-partition}, we know that
		\begin{equation}
		    P(w) = p(w) - \sum_{G = (G_1, \dots , G_m)\in \cG_w, \ m\ge 2}P(G); \ \hat{P}(w) = \hat{p}(w) - \sum_{G = (G_1, \dots , G_m)\in \cG_w, \ m\ge 2}\hat{P}(G).
		\end{equation}
		On the other hand, we have shown that $\hat{P}(G)=P(G)$ for $G = (G_1, \dots , G_m)$ with $m\ge 2$ and also from Step 2 $p(w) = \hat{p}(w)$ as $w$ has only one letter $A$. Therefore it follows that $P(w) = \hat{p}(w)$, which finishes the induction step. Now we have shown~\eqref{eqn-forest-pois-equi} and the proof of the whole proposition is complete.
\end{proof}

\section{Open questions}\label{sec-ques}
In addition to the famous problems about Yang-Mills theory in dimensions $d>2$, there are many aspects of the two-dimensional theory that we would like to understand better.

\begin{ques}
In two dimensions, there are actually two approaches for representing a Wilson loop expectation as a sum over surfaces. The first (established in this paper) only involves surfaces that are locally ``flat'' (perhaps with ramification points). The second (established in \cite{cao2023random}) involves all discretized surfaces, flat or otherwise. Our intuition is that the ``non-flat'' surfaces in the second approach cancel each other out. For example, in 2D lattice, if you remove all the plaquettes not surrounded by a loop $\ell$, it does not change the Wilson loop expectation for $\ell$ at all. This means that the surfaces involving plaquettes \emph{not} surrounded by $\ell$ together contribute zero to the sum. 
Can we explain where and why cancellation occurs at least in 2D?
\end{ques}

\begin{ques}
	The surface sum described in Theorem~\ref{thm-main-simple} (and more formally in Lemma~\ref{lem-main-poisson-gen} and~\ref{lem-other-gen}) depends \emph{a priori} on a choice of a spanning tree, which determines a specific lasso basis. Because it describes a Wilson look expectation, we know \emph{a posteriori} that the overall sum does not depend on such a choice. Is there a more direct way to see that different expressions for the Wilson loop expectation (using different lasso bases) are equal?
\end{ques}

\begin{ques}
    The results of this paper involved Wilson loop expectations on the plane. Which of these results can be extended to the sphere or to higher genus surfaces? Note that some work in this direction has been done: see~\cite{dahlqvist2020,lemoine2022large, dahlqvist2022large} for results about the master fields on the sphere and higher genus surfaces.
\end{ques}

\begin{ques}
    Can one give a detailed mathematical explanation of the Gross-Taylor expansion~\cite{Gross1993a} for higher genus surfaces, and prove the convergence of the summation?
\end{ques}


\begin{ques}
    While Theorem~\ref{thm::walk-on-permutations} was formulated only for $G=\UN$, similar random walks on permutations can be constructed for general $G$.
    Can we compute Wilson loop expectations for general Lie groups from the same approach?
\end{ques}

\begin{ques}
    Is there a more computationally efficient way to compute formulas like the ones in the table in Section~\ref{sec-chart}? For example, can one compute the $\UN$ Wilson loop expectation for a complicated loop with 10 or 20 self-intersections? (We expect that the number of terms in the formula can grow at least exponentially with the number of self-intersections of the loop, so the expressions might require some space to write down.)
\end{ques}

\begin{appendix}
\section{Strand interpretation}\label{sec-appendix}
In this appendix, we provide an alternative interpretation of our results in terms of Poisson sums on \emph{strand diagrams} which will be useful in the companion paper~\cite{cao2023random} studying the lattice gauge theory for any dimensions $d\ge 2$. We first recall basic definitions and notations.

Let $\bm \Gamma = (\Gamma_1, \Gamma_2, \dots, \Gamma_k)$ be an (ordered) collection of words $\Gamma_i$ on letters $\{\lambda_1,\cdots,\lambda_L\}$ where
$$\Gamma_i = \lambda_{c_i(1)}^{\ep_i(1)}\cdots \lambda_{c_i(M_i)}^{\ep_i(M_i)}$$
for some $c_i: [M_i]\to [L]$ and $\ep_i: [M_i]\to \{-1,1\}$. By letting $M=M_1+\cdots+M_k$ and concatenating $c_i$'s and $\ep_i$'s, we may define $c:[M]\to [L]$ and $\ep:[M]\to \{-1,1\}$. Our goal is to compute
$\E[\Tr \big(B({\bm \Gamma}) \big)] ,$
where
$$\mathbb \Tr \big(B({\bm \Gamma})) := \Tr(B(\Gamma_1))\cdots \Tr(B(\Gamma_k)), \text{ and } B(\Gamma_i) = B_{T}^{\ep_i(1)} (\lambda_{c_i(1)}) \cdots B_{T}^{\ep_i(M_i)} (\lambda_{c_i(M_i)}),$$
and where $\{B_{T}(\lambda_\ell)\}_{\ell\in [L]}$ is a collection of independent Brownian motions on a classical Lie group $G$ that is one of $\UN, \SON, \SUN, \SphN$ started at the identity and run for time $T>0$. \footnote{Compared to our main setting in this paper, we chose the time parameters to be the same as $T$ for all indices for simplicity.} We also define
$$\mathcal C = \bigcup_{\ell\in [L]}\binom{c^{-1}(\ell)}{2} = \{ (m,m^*) : m < m^* \text{ and } c(m) = c(m^*)\},$$
and
$$\mathcal D_T = \bigsqcup_{(m,m^*) \in \mathcal C} [0,T],$$
equipped with some parametrizing bijection $\eta: \mathcal C \times [0,1] \to \mathcal D_T$.\footnote{We also removed some geometric interpretations compared to Definition~\ref{defn-space-match} for simplicity.} We now consider the Poisson point process $\Sigma$ on $\mathcal D_T$ with intensity given by the Lebesgue measure.

We will interpret the results in Section~\ref{sec-poisson-sum} first for $G=\UN$ to conclude that expectations of unitary Brownian motion may be represented by a certain diagram which is obtained from a Poisson process on $\mc{D}_T$. To begin to make this statement precise, in the following definition, we describe how to associate a diagram to a given collection of points of $\mc{D}_T$.

\begin{defn}[Strand diagram]
\label{defn-strand-diagram}
    Let $\Gamma=\lambda_{c(1)}^{\ep(1)}\cdots \lambda_{c(M)}^{\ep(M)}$ be a word on $\{\lambda_1,\cdots, \lambda_L\}$ and $\Sigma$ be a collection of points in $\mathcal D_T$. Then $\eta^{-1}(\Sigma)$ be a collection of points $((m,m^*), t)\in \mathcal C\times [0,1]$ for $(m,m^*)\in \mathcal C$ and $t\in [0,1]$. Let $n_\ell = |c_i^{-1}(\ell)|$ for each $\ell\in [L]$.
    The \textbf{strand diagram of $(\Gamma, \Sigma)$} is an array of right- or left-directed arrows, each of which is identified as the unit interval $[0,1]$, placed as follows.
    \begin{itemize}
        \item There are $L$ columns and each column is labelled by $\lambda_\ell$ for $\ell \in [L]$;
        \item The column labeled by $\lambda_\ell$ consists of a stack of $n_\ell$ unit-length arrows, each of which corresponds to an element of $c_i^{-1}(\ell)$;
        \item If an arrow corresponds to $m=c_i^{-1}(\ell)\in [M]$, it is right-directed (resp.\ left-directed) if $\ep(m)=1$ (resp.\ $\ep(m)=-1$);
        \item The end of arrow corresponding to $m$ is connected to the origin of the arrow corresponding to $m+1$, modulo $M$;
        \item For each point $((m,m^*), t)\in \eta^{-1}(\Sigma)$, if $\ep(m)\ep(m^*)=1$, we insert a green crossing (called the ``same-direction swap'') on two arrows corresponding to $m$ and $m^*$ at location $t\in [0,1]$. Otherwise, we put a blue double bar (called the ``opposite-direction swap'') on two arrows corresponding to $m$ and $m^*$ at location $t\in [0,1]$.
    \end{itemize}

    In general, if $\bm \Gamma=(\Gamma_1, \dots \Gamma_k)$ is a collection of words $\Gamma_i$ on $\{\lambda_1,\cdots, \lambda_L\}$, we define the strand diagram of $(\bm \Gamma, \Sigma)$ as a collection of strand diagrams of $(\Gamma_1, \Sigma_1), \cdots, (\Gamma_k,\Sigma_k)$ where $\Sigma = \bigsqcup_{i=1}^k \Sigma_i$ with the same labelled columns. See Figure~\ref{fig-strand-diagram} for an example.
\end{defn}
\begin{figure}[ht!]
    \centering
    \includegraphics[width=7cm]{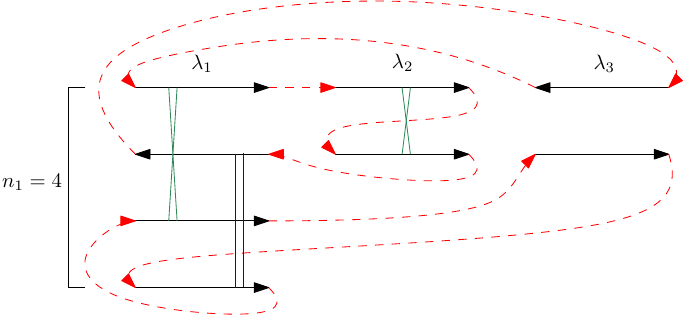}\qquad
    \includegraphics[width=6cm]{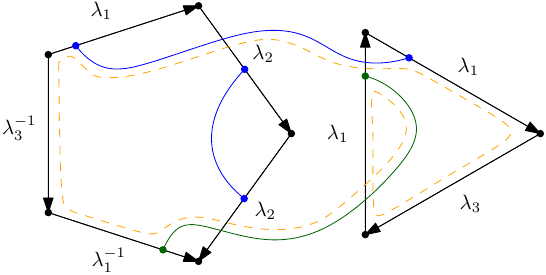}
    \caption{\textbf{Left:} The strand diagram for $((\Gamma_1,\Gamma_2), \Sigma)$ with $\Gamma_1=\lambda_1\lambda_2\lambda_2\lambda_1^{-1}\lambda_3^{-1}$, $\Gamma_2=\lambda_1\lambda_3\lambda_1$, and $\eta^{-1}(\Sigma)=\{((1,6),1/4), ((2,3),1/2), ((4,8),3/4)\}$. For example, the first and sixth alphabet in the word $\Gamma$ are $\lambda_1$'s, so we put a green crossing at the 1/4 location of two unit intervals representing those. Similarly, we put a blue double bar at the 3/4 location of two unit intervals representing $\lambda_1^{-1}$ and $\lambda_1$. \textbf{Right:} The CW-complex constructed from the left strand diagram. By following each 1-cells and closing each cycle by adding a new 2-cell, we obtain a closed surface whose Euler characteristic defines the Euler characteristic of the strand diagram. The orange dashed line is an example of a 2-cell we add. Including this face, we need 3 faces in total to obtain a closed surface with the minimum genus. Therefore, the Euler characteristic of this surface is equal to $V-E+F=14-17+5=2$, and the resulting surface is a sphere.}
    \label{fig-strand-diagram}
\end{figure}

We define a CW-complex from a strand diagram as in Figure~\ref{fig-strand-diagram}. Each word $\Gamma_i$ can be represented by a regular polygon with unit-length arrows (preserving the orientation), and each same-direction swap or opposite-direction swap corresponds to a path connecting two arrows at the specified location in the strand diagram. As a result, we have $k$ 2-cells for each polygon, $M+2|\Sigma|$ 1-cells, and $M+2|\Sigma|$ 0-cells, where $M=n_1+\cdots+n_L$. Then there exists a \emph{closed} surface with the minimum genus constructed by adding extra $F$ 2-cells, that is by following every 1-cell and adding a 2-cell whenever they form a cycle. (Equivalently, it can be viewed as a \emph{ribbon graph}.) By Euler's formula, the number $F$ of extra 2-cells determines the minimum genus, that is $\chi=(M+2|\Sigma|)-(M+3|\Sigma|)+(k+F) = k+F-|\Sigma|$. We define the Euler characteristic $\chi$ of the strand diagram as the Euler characteristic of this surface with the minimum genus.

\begin{defn}\label{defn-chi-S} Let $S$ be the strand diagram of $(\bm \Gamma, \Sigma)$. Define $\numcomp(S)$ by the number of cycles obtained by swaps and the Euler characteristic of $S$ by
$$\chi(S)=\abs{\bm \Gamma}+\numcomp(S)-\abs{\Sigma}.$$
\end{defn}

Instead of constructing a CW-complex by adding extra faces as above, we may construct another CW-complex by ``gluing'' edges according to each swap. More precisely, we divide each unit-length arrow by $J$ equal pieces for sufficiently large $J>0$ so that each piece only contains at most one swap. 
Note that this picture is exactly what Figure~\ref{fig-expansion-diagram} describes, where the arrows there correspond to opposite-direction swaps in the corresponding stand diagram. (In Figure~\ref{fig-expansion-diagram}, we do not have any arrows corresponding to same-direction swaps, but both types of swaps occur for generic words.)
Hence, if we construct a CW-complex from $k$ 2-cells (with total $2\abs{\Sigma}$ edges) as described in Section~\ref{sec-main-results} and gluing operations described in Table~\ref{table-gluing}, it is straightforward to see that the resulting CW-complex has exactly the same Euler characteristic compared to the previous construction from a stand diagram.

These two interpretations are \emph{dual} to each other in the sense that in the former interpretation the number of faces $F$ determines the Euler characteristic, while in the latter the number of vertices after gluings does so.
As a consequence, we obtain the following result from Lemma~\ref{lem-main-poisson} and~\ref{lem-un}.
\begin{lem}[Expected trace as Poisson sums]
\label{lem-poisson-sum}
    Let $G=\UN$, $\bm \Gamma$ be a collection of words on $\{\lambda_1, \ldots, \lambda_L\}$, and $T>0$.
    Let $\Sigma$ be the Poisson point process on $\mathcal D_T$. Consider the strand diagram $S$ for $(\bm \Gamma, \Sigma)$. Then
    $$\mathbb E\big[\Tr(B(\bm \Gamma))\big] = \exp\left(-\frac{1}{2}\sum_{m=1}^M T+\sum_{(m,m^*)\in \mathcal C} T\right) \mathbb E\Bigl[\ep(\Sigma)(-1)^{|\Sigma|} N^{-k+\chi(S)}\Bigr].$$
\end{lem}

Recall from Definition~\ref{defn-chi-S} that $(-1)^{\abs{\Sigma}}N^{-k+\chi(S)} = (-1/N)^{\abs{\Sigma}}N^{\numcomp(S)}$. Since $\numcomp(S)$ depends only on how the ends of strands are hooked regardless of what happens in between, we may rewrite the sum in terms of matchings on the ends $[2n_\ell]$ for $\ell\in [L]$. We denote by $\pi(S)$ this matching induced by a strand diagram $S$. See Figure~\ref{fig-strand-diagram-matchings} for an illustration.
Next, we define the weight
\begin{equs}\label{eq:w-T-def}
w_T(\pi) := \exp\bigg(\sum_{\ell \in [L]} \bigg(\binom{n_\ell}{2} - \frac{n_\ell}{2}\bigg)T\bigg) \E\big[ \varep(\Sigma) (-1/N)^{|\Sigma|} \ind(\pi(S) = \pi) \big],
\end{equs}
associated to $\pi=(\pi_\ell)_{\ell\in [L]}$ where $\pi_\ell$ is a matching on $[2n_\ell]$. We note that not all matchings $\pi$ can be realized as $\pi(S)$ for some strand diagram $S$ because there are two types of strands (left- and right-directed) and two types of swaps we allow for each pair of strands. We write $\mc M_{\bm \Gamma} := \{\pi(S): S\text{ is a strand diagram of $(\bm\Gamma,\Sigma)$ for some $\Sigma$}\}$ for given $\bm \Gamma$.\footnote{In~\cite{cao2023random}, this set will be more carefully defined in the context of a variant of the Brauer algebra.}

\begin{figure}[ht!]
    \centering
    \includegraphics[width=7cm]{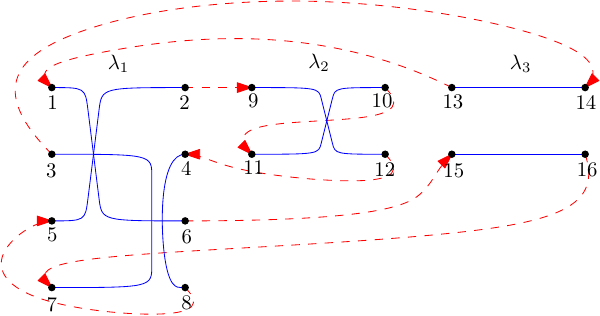}\qquad
    \includegraphics[width=6cm]{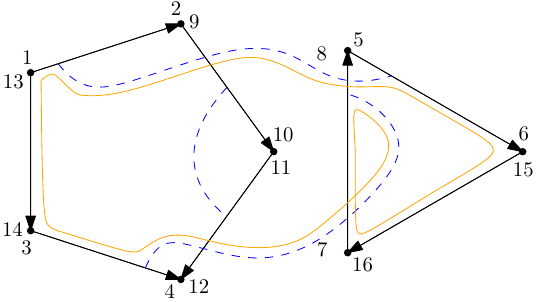}
    \caption{\textbf{Left:} The same strand diagram as Figure~\ref{fig-strand-diagram} but with ends of arrows labelled. By following all swaps, each row of the strand diagram defines a matching on $[2n_\ell]$ for $\ell=1,2,3$. \textbf{Right:} The corresponding CW complex picture with labels. It is straightforward that the number of components in the left picture is exactly the number of faces in this picture.}
    \label{fig-strand-diagram-matchings}
\end{figure}

In the companion paper~\cite{cao2023random}, we are particularly interested in the case where $\bm \Gamma$ contains exactly the same number of $\lambda_i$'s and $\lambda_i^{-1}$'s, that we will call ``balanced'' to consider non-trivial $T\to \infty$ limit.
For bookkeeping, we state a version of our Poisson sum results that is applicable to this specific case.

In particular, by considering a local picture of each $\lambda_i$, we consider only one column of strands with $n$ right-directed and $n$ left-directed arrows.
Let $\mc M_n$ be the set of matchings $\pi(S)$ induced by all possible stand diagrams on these strands. Then we have

\begin{lem}\label{prop:unitary-brownian-motion-expectation}
Let $G=\UN$ and $i_1, \ldots, i_n$, $i_1', \ldots, i_n'$, $j_1, \ldots, j_n$, $j_1', \ldots, j_n' \in [N]$. Then
\begin{equ}\label{eq-un-tensor}
\E \big[(B_T)_{i_1j_1} \cdots (B_T)_{i_n j_n} (\ovl{B}_T)_{i_1' j_1'} \cdots (\ovl{B}_T)_{i_n' j_n'} \big] = \sum_{\pi\in \mc M_n} w_T(\pi) \ind(\text{indices match with $\pi$})
\end{equ}
where $w_T(\pi)$ is defined as~\eqref{eq:w-T-def}.
\end{lem}
\begin{proof}
This is slightly more general than the trace expansion in our results in that the indices can be now arbitrary, not cyclic as in the trace expansion.
We may regard that the cyclic condition was imposed (with strands from other lassos) by following red dotted arrows in Figure~\ref{fig-strand-diagram}, so taking a local picture of one lasso is equivalent to removing those extra conditions imposed by red dotted arrows.

The proof is almost immediate from Section~\ref{sec-poisson-sum} and~\ref{sec-surface-story} by changing the setting accordingly. For example, in order to obtain the Poisson sum for the left hand side of~\eqref{eq-un-tensor}, all we need to change from Lemma~\ref{lem-main-poisson} is the indices of the product $W_{a_1a_2}\cdots W_{a_{2|\Sigma|}a_{2|\Sigma|+1}}$ to $W_{i_1j_1}\cdots W_{i_n'j_n'}$. Hence, the exponential factor in~\eqref{eq:w-T-def} coincides with Lemma~\ref{lem-main-poisson} and the remaining term follows from Proposition~\ref{prop-compatible-one}.\end{proof}

We now consider the case of $G=\SON$ and $\SphN$. The strand interpretations for these groups are exactly the same as that for $G=\UN$, except that we now allow same- or opposite-direction swaps for any pair of strands, whether or not they have the same directions. Note that their gluing operations are identical as described in Table~\ref{table-gluing}. In the example of Figure~\ref{fig-expansion-diagram}, if $G=\SON$ or $\SphN$, we allow gluing a pair of edges either in an orientation-preserving way or not.

To state a similar result like Lemma~\ref{prop:unitary-brownian-motion-expectation} for $G=\SON, \SphN$, we also take a local picture of strands from one lasso; consider one column of strands with $n$ right-directed arrows. We again allow an arbitrary index at each end of strands, which gives a more general formula than the trace expansion. In addition, compared to the Poisson point process $\Sigma$ for $G=\UN$, we consider $\Sigma_{\mrm{OS}}$ which is the union of \emph{two independent copies} of $\Sigma$ so that one represents same-direction swaps and the other represents opposite-direction swaps.

Note that the exponential factor in Lemma~\ref{lem-main-poisson} consists of two terms. The first comes from the normalizing constant $Z$ defined in~\eqref{eqn-z} which is \emph{independent} of the Poisson point process and only dependent of the number of strands and the constant $\fc_\fg$ in Table~\ref{table-lie}. The second comes from the partition function of Poisson point processes.
Therefore, the factor in Lemma~\ref{lem-main-poisson} now becomes $\exp\left(- \frac{\fc_\fg}{2}nT + 2\binom{n}{2}T\right)$ where $\fc_\fg = -1+\frac{\ep}{N}$ as defined in Table~\ref{table-lie}. Here, we set $\ep=1$ for $G=\SON$ and $\ep=-1$ for $G=\SphN$ following the convention of~\cite{Dahlqvist2017}. Because we used two copies of $\Sigma$, there is the factor 2 in the second term.

Consider $G=\SON$ where we may define the weight, for $\ep=1$,
\begin{equs}\label{eq:w-T-SON-def}
w_T^{\ep}(\pi) := \exp\bigg(2\binom{n}{2}T - \frac{n}{2}\left(1-\frac{\ep}{N}\right)T\bigg) \E\big[ F_{\ep}(\Sigma_{\mrm{OS}}) \ind(\pi(S) = \pi) \big]
\end{equs}
where $F_+(\Sigma_{\mrm{OS}})$ is the weight defined from Proposition~\ref{prop-compatible-one}, and $\pi(S)$ is defined as before for $G=\UN$. 
Let $\mc M_n^{\mrm{OS}}$ be the set of possible matchings induced from $\Sigma_{\mrm{OS}}$.
Again, by following the proofs in Section~\ref{sec-surface-story} in the modified setting, we obtain the Poisson sum results for $G=\SON$.

\begin{lem}\label{prop:orthogonal-symplectic-brownian-motion-expectation}
Let $G = \SON$ and $B_T$ be a $G$-valued Brownian motion at time $T$. For $\mbf{i} = (i_1, \ldots, i_n), \mbf{j} = (j_1, \ldots, j_n) \in [N]^n$, we have that
\begin{equs}
\E[(B_T)_{i_1 j_1} \cdots (B_T)_{i_n j_n}] = \sum_{\pi\in \mc M^{\mrm{OS}}_n}w_T^{+}(\pi) \ind(\text{indices match with $\pi$}).
\end{equs}
\end{lem}

For $G=\SphN$, we have almost the same formula, except that now the labels at ends of strands do not need to match exactly, as seen in Proposition~\ref{prop-compatible-one}. On the other hand, $\pi(S)$ determines whether two indices equals or off by $N/2$ depending on whether a match joins the labels on the same side or not. Thus, we say the labels are compatible with $\pi$ if they satisfy the correct relations in this sense. Defining $F_{-}(\Sigma_{\mrm{OS}})$ according to Proposition~\ref{prop-compatible-one} and reusing~\eqref{eq:w-T-SON-def} for $\ep=-1$, we obtain the Poisson sum for $G=\SphN$.
\begin{lem}\label{prop:SU-brownian-motion-expectation}
    Let $G = \SphN$ and $B_T$ be a $G$-valued Brownian motion at time $T$. For $\mbf{i} = (i_1, \ldots, i_n), \mbf{j} = (j_1, \ldots, j_n) \in [N]^n$, we have that
\begin{equs}
    \E[(B_T)_{i_1 j_1} \cdots (B_T)_{i_n j_n}] = \sum_{\pi\in \mc M^{\mrm{OS}}_n}  w_T^{-}(\pi) \ind(\text{indices are compatible with $\pi$})
\end{equs}
where $w_T(\pi)$ is defined as~\eqref{eq:w-T-def}.
\end{lem}

For $G=\SUN$, we have a similar interpretation as $G=\UN$, but now consider another copy of $\Sigma$ that represents each pair related to contracting-edge operation in Table~\ref{table-gluing}. The corresponding object in strand diagrams can be also visualized as a line segment between a pair of strands, but it does not swap anything and only contributes an extra factor $\pm 1/N^2$ in the weight. Thus, $\pi(S)$ is only induced from the other usual swaps as in the case of $G=\UN$.

Let $\Sigma_{\mrm{SU}}$ be the union of these two copies of $\Sigma$. For this case, as a local picture of on lasso, we consider $n$ right-directed arrows and $m$ left-directed arrows. As per usual, let $F_{\mrm{SU}}$ be the weight defined from Proposition~\ref{prop-compatible-one} and $\mc M_{n,m}$ be the set of possible matchings induced by $\Sigma_{\mrm{SU}}$. Then for the weight
\begin{equs}\label{eq:w-T-SON-def}
w_T^{\mrm{SU}}(\pi) := \exp\bigg(2\binom{n+m}{2}T - \frac{n+m}{2}\left(1-\frac{1}{N^2}\right)T\bigg) \E\big[ F_{\mrm{SU}}(\Sigma_{\mrm{SU}}) \ind(\pi(S) = \pi) \big],
\end{equs}
we have a desired result for $G=\SUN$.
\begin{remark}
    In our story, it is natural to allow contracting edges only on a pair of two \emph{distinct} strands. 
    Alternatively, as in~\cite{cao2023random}, we may also allow this on the \emph{same} strand by introducing additional Poisson point process. In this case, the exponential factor in the weight~\eqref{eq:w-T-SON-def} changes by $\left(1-\frac{1}{N^2}\right)(n+m)T$ which is the contribution from this extra point process on single strands, resulted in
\begin{equs}
\wt{w}_T^{\mrm{SU}}(\pi) := \exp\bigg((n+m)^2T - \frac{n+m}{2}\left(1+\frac{1}{N^2}\right)T\bigg) \E\big[ \wt{F}_{\mrm{SU}}(\wt{\Sigma}_{\mrm{SU}}) \ind(\pi(S) = \pi) \big],
\end{equs}
if $\wt{\Sigma}_{\mrm{SU}}$ and $\wt{F}_{\mrm{SU}}$ are the union of related Poisson point processes and the corresponding weight.
\end{remark}

\begin{lem}\label{prop:su-brownian-motion-expectation}
Let $G=\SUN$ and $i_1, \ldots, i_n$, $i_1', \ldots, i_n'$, $j_1, \ldots, j_n$, $j_1', \ldots, j_n' \in [N]$. Then
\begin{equ}
\E \big[(B_T)_{i_1j_1} \cdots (B_T)_{i_n j_n} (\ovl{B}_T)_{i_1' j_1'} \cdots (\ovl{B}_T)_{i_n' j_n'} \big] = \sum_{\pi\in \mc M_{n,m}} w_T^{\mrm{SU}}(\pi) \ind(\text{indices match with $\pi$}).
\end{equ}
\end{lem}

\begin{remark}
    All the above lemmas easily follow from previous results, e.g.\ from~\cite{Dahlqvist2017}. Still, this section (as well as the whole paper) provides more probabilistic viewpoints on various formulas related to this topic. These turn out to be useful in the companion paper~\cite{cao2023random} where we consider some explorations on the strand diagrams to derive some generalized results on lattice Yang-Mills theory. We also note that in the existing literature, most formulas are written in tensor notations for brevity. In contrast, in this paper we keep expressions with the actual matrix entries.
\end{remark}

\end{appendix}

\bibliographystyle{hmralphaabbrv}
\bibliography{references}

\end{document}